\documentclass[aos]{imsart}
\usepackage{xr}
\RequirePackage[OT1]{fontenc}
\RequirePackage{amsthm,amsmath}
\RequirePackage[colorlinks,citecolor=blue,urlcolor=blue]{hyperref}
\RequirePackage{yk}

\usepackage{amsthm,amsmath,amssymb,enumerate,mathrsfs}
\usepackage{color}
\usepackage{dsfont}
\usepackage[listings]{tcolorbox}
\tcbuselibrary{listings,theorems}
\usepackage{mathtools}
\usepackage{algorithm}
\usepackage[noend]{algpseudocode}
\usepackage{mathtools}
\usepackage{caption}
\usepackage{subcaption}
\usepackage[sectionbib,square]{natbib}

\renewcommand{\S}{\mathbb{S}}
\newcommand{\pen}{\textbf{pen}}
\newcommand{\m}{\mathbf}
\usepackage{hyperref}
\hypersetup{
    colorlinks,%
    citecolor=blue,%
    filecolor=red,%
    linkcolor=red,%
    urlcolor=black,
}
\usepackage[top=2.5cm, bottom=3.0cm, left=1.25in, right=1.25in]{geometry}

\newcommand{\E}{{\mathbb E}}

\newcommand{\R}{{\mathbb R}}

\renewcommand{\P}{{\mathbb P}}

\newcommand{\M}{{\mathcal{M}}}

\newcommand{\indep}{\perp\hspace{-.5em}\perp}%{\rotatebox[origin=c]{90}{$\models$}}
\newcommand{\mb}{\mathbf}

\RequirePackage{xcolor}

\newcommand{\s}{\text{sgn}}
\theoremstyle{definition}

\theoremstyle{remark}

\newtcbtheorem{mytheo}{Theorem}%
{colback=green!5,colframe=green!35!black,fonttitle=\bfseries}{th}

\newtcbtheorem{mylemma}{Lemma}%
{colback=blue!5,colframe=purple!35!black,fonttitle=\bfseries}{th}

\newtcbtheorem{myassmp}{Assumption}%
{colback=green!5,colframe=black!35!black,fonttitle=\bfseries}{th}

\newdimen\AAdi%
\newbox\AAbo%
\def\AAk#1#2{\setbox\AAbo=\hbox{#2}\AAdi=\wd\AAbo\kern#1\AAdi{}}%

\newcounter{rcnt}[section]

\def\argmin{\mathop{\rm argmin}}
\def\argmax{\mathop{\rm argmax}}

\arxiv{arXiv:1903.10063}

\startlocaldefs
\numberwithin{equation}{section}
\theoremstyle{plain}

\endlocaldefs
\renewcommand{\S}{S}

\begin{document}

\begin{frontmatter}
\title{Optimal Linear Discriminators For The Discrete Choice Model In Growing Dimensions}
\runtitle{Discrete Choice Model in Growing Dimensions}

\begin{aug}
\author{\fnms{Debarghya} \snm{Mukherjee}\ead[label=e1]{mdeb@umich.edu}},
\author{\fnms{Moulinath} \snm{Banerjee}\ead[label=e2]{moulib@umich.edu}}
\and
\author{\fnms{Ya'acov} \snm{Ritov}\ead[label=e3]{yritov@umich.edu}
%\ead[label=u1,url]{http://www.foo.com}
}

\thankstext{t1}{Supported by NSF Grant DMS-1712962}
%\thankstext{t2}{Supported by Labex MME-DII (ANR11-LBX-0023-01)}
%\thankstext{t3}{Supported by NSF CAREER Grant DMS-1150435}
\runauthor{Mukherjee, D., Banerjee, M. and Ritov, Y.}

\affiliation{University of Michigan}

\address{University of Michigan \\
437, West Hall, \\
1085 South University \\
Ann Arbor, MI 48109 \\
\printead{e1}
%\phantom{E-mail:\ }\printead*{e2}
}

\address{University of Michigan\\
275, West Hall, \\
1085 South University \\
Ann Arbor, MI 48109 \\
\printead{e2}}

\address{University of Michigan\\
462, West Hall, \\
1085 South University \\
Ann Arbor, MI 48109 \\
\printead{e3}}
\end{aug}

\begin{abstract}
Manski's celebrated maximum score estimator for the discrete choice model, which is an optimal linear discriminator, has been the focus of much investigation in both the econometrics and statistics literatures, but its behavior under growing dimension scenarios largely remains unknown. This paper addresses that gap. Two different cases are considered: $p$ grows with $n$ but at a slow rate, i.e. $p/n \rightarrow 0$; and $p \gg n$ (fast growth). In the binary response model, we recast Manski's score estimation as empirical risk minimization for a classification problem, and derive the $\ell_2$ rate of convergence of the score estimator under a \emph{transition condition} in terms of our margin parameter that calibrates the level of difficulty of the estimation problem. %involving a smoothness parameter $\alpha > 0$, the rate of the score estimator in the slow regime is essentially $\left((p/n)\log n\right)^{\frac{\alpha}{\alpha + 2}}$, while, in the fast regime, the penalized score estimator essentially attains the rate $((s_0 \log{p} \log{n})/n)^{\frac{\alpha}{\alpha + 2}}$, where $s_0$ is the sparsity of the true regression parameter. 
%For the most interesting regime, $\alpha = 1$, the rates of Manski's estimator are therefore $\left((p/n)\log n\right)^{1/3}$ and $((s_0 \log{p} \log{n})/n)^{1/3}$ in the slow and fast growth scenarios respectively, which can be viewed as high-dimensional analogues of \emph{cube-root asymptotics}\cite{kp90}. 
We also establish upper and lower bounds for the minimax $\ell_2$ error in the binary choice model that differ by a logarithmic factor, and construct a minimax-optimal estimator in the slow growth regime. Some extensions to the general case -- the multinomial response model -- are also considered. Last but not least, we use a variety of learning algorithms to compute the maximum score estimator in growing dimensions.
\end{abstract}
\end{frontmatter}

\section{Introduction}\label{intro}
%\emph{This needs to be fleshed out. Some history and what work has been done in stat and econ lits. Point to the gap. Also the more general gap that
%non-standard problems in high dimensions are not well-studied to the best of our knowledge. Pin-point contributions up front. Some of it is present here but the literature review is cursory at best and the writing needs much more polish.}
The maximum score estimator for the discrete choice model was introduced by Charles Manski in his seminal paper \cite{manski1975} in connection with the stochastic utility model of choice, and has been extensively studied in both the econometrics and the statistics literatures. The binary choice model can be considered as a linear regression model with missing data.  More specifically, let $$Y_i^* = X_i'\beta^0 + \epsilon_i$$ where $\{X_i, \epsilon_i\}$ are we $n$ i.i.d pairs, the distribution of $\epsilon_i$ is allowed to depend on $X_i$ and $\med(Y_i^*|X_i) = X_i'\beta^0$ (i.e. $\med(\epsilon_i|X_i) = 0$), but instead of observing the full data, we only get to see $\{Y_i, X_i\}$ where 
%$$Y_i = \s(Y_i^*) \equiv $. We observe $\{(X_i, Y_i)\}_{i=1, \dots, n}$ pairs where the relation is as described: 
\begin{equation}
\label{eq:binary_response}
Y_i =  \s(Y_i^*) = \s(X_i^{\top} \beta^0 + \epsilon_i) \,.
\end{equation}
The regression parameter $\beta^0$ is of interest. 
%Since $\beta^0$ can only be identified up to a scale factor in the observed data model, it is typically assumed that $\|\beta^0\| = 1$. 
The population score function is defined as: $$S(\beta) = \mathbb{E}(Y\s(X'\beta)) = \mathbb{E}(\s(Y^*)\s(X'\beta))$$ and the corresponding sample score function is:
$$S_n(\beta)=\frac{1}{n}\sum_{i=1}^n Y_i \s(X_i^\mathsf{T}\beta) = \frac{1}{n}\sum_{i=1}^n \s(Y_i^*)\s(X_i^\mathsf{T}\beta) \,.$$
The maximum score estimator is defined as any value of $\beta$ that maximizes the sample/empirical score function: $$\hat{\beta}_n = \argmax_{\beta: \|\beta\|=1} S_n(\beta) \,.$$
Note that some norm restriction on $\beta$ is important both for identifiability of $\beta$ in this model, as well as for meaningful optimization. 
% since the value of the sample score function is invariant to scale changes. 
As $\beta^0$ is only identifiable and estimable up to direction, in what follows, we take $\|\beta_0\| = 1$. We also note that the choice of the maximizer is not important; in fact there is no unique maximizer. In follow-up work \cite{manski1985}, Manski proved the consistency of  $\hat{\beta}_n$ to the true $\beta^0$ and some large deviation results under mild assumptions. The asymptotic distribution properties of the maximum score estimator were established by Kim and Pollard \cite{kp90} who proved that under additional assumptions $(\beta^0 - \hat{\beta}_n) = O_p(n^{-\frac{1}{3}})$ and that the normalized difference converges in distribution to a non-Gaussian random variable that is characterized as the maximizer of a quadratically drifted Gaussian process. Shortly thereafter, Horowitz \cite{horowitz1992} established, under smoothness conditions beyond those in \cite{kp90}, the estimator obtained by maximizing a kernel smoothed version of the score function can improve the rate of the smoothed estimator. One advantage of Horowitz's estimator over the original maximum score estimator, from a practical viewpoint, is that the limit distribution in his setting is Gaussian and therefore more amenable to inference, while the quantiles of the non-Gaussian limit are hard to determine. Also, around the same time Klein and Spady (see \cite{klein1993efficient}) proved that under the additional assumption $\E(Y|X) = \E\left(Y|X^{\top}\beta^0\right)$, one can obtain a consistent and asymptotically normal estimator of $\beta^0$, which is also semi-parametrically efficient. More recently, Seo and Otsu (\cite{seo2018local}, \cite{seo2015asymptotics}) have extended the asymptotic results on the score estimator to dependent data scenarios. Alternatively, resampling techniques can also be used for inference.  Manski and Thompson \cite{manski1986} suggested that the usual bootstrap yields a good approximation of the distribution of the maximum score estimator, but it turns out that the bootstrap is actually inconsistent, as shown in Abrevaya and Huang \cite{abrevaya2005} (but see also \cite{sen2010inconsistency}). More recently, a model--based smoothed bootstrap approach was proposed by Patra et.al. \cite{patra2018}. Generic ($m$ out of $n$) subsampling techniques \cite{politis1999subsampling} can, of course, be used in principle, but typically suffer from imprecise coverage unless the subsample size $m$ is well-chosen, which is typically a difficult problem. For  applications of maximum score estimators and their variants, see \cite{briesch2002semiparametric}, \cite{fox2013measuring}, \cite{bajari2008evaluating} and references therein. 
\noindent
\newline
\newline
{\bf Connections to empirical risk minimization:} The maximum score estimator is naturally connected to a classification problem with two classes. In Manski's problem, we have observations
$\{X_1, Y_1\}, \cdots, \{X_n, Y_n\}$, where $X_i \in \mathbb{R}^p$ and $Y_i \in \{-1,1\}$, these being the labels of the two classes. The conditional class probabilities are specified by $$ \eta(x) = \P(Y = 1| X=x)  = 1 - F_{\epsilon|X=x} (-x^\mathsf{T}\beta^0)\,.$$ For classifying the $Y_i$'s  using an arbitrary classifier $h$ under 0-1 loss,  the population risk is given by $L(h) = \mathbb{P}(Y \ne h(X))$. Consider the set of classifiers corresponding to all possible hyperplanes, i.e. $$\mathcal{G} = \{g_{\beta}: g_{\beta}(x) = \s(x^\mathsf{T}\beta), \|\beta\| = 1\} \,.$$
The population risk under 0-1 loss for this family is then given by: $$L(\beta) = L(g_{\beta}) = P(Y \neq \s(X^\mathsf{T}\beta)) \,,$$
and is consistently estimated by the empirical risk $L_n(\beta) = \mathbb{P}_n(Y \neq \s(X^\mathsf{T}\beta))$.
From the structure of the model, it is easy to see that the Bayes' classifier, i.e. the classifier which minimizes the population risk in this model (over all possible classifiers) is precisely $g_{\beta_0}$:
\begin{equation*}
\begin{split}
\s(\eta(x) - 0.5) & = \s\left( 1 - F_{\epsilon|X=x} (-x^\mathsf{T}\beta^0) - 0.5\right) \\
& = \s\left( 0.5 - F_{\epsilon|X=x} (-x^\mathsf{T}\beta^0) \right) \\
& = \s(x^\mathsf{T}\beta^0)  = g_{\beta^0}\hspace*{0.3in} [\because \med(\epsilon|X) = 0] \,.
\end{split}
\end{equation*}
Thus $\tilde{\beta}_n := \arg \min_{\beta}\,L_n(\beta)$ empirically estimates the Bayes classifier. By simple algebra $S(\beta) = 1 - 2 L(\beta)$ and $S_n(\beta) = 1 - 2 L_n(\beta)$. Since the former is maximized at $\beta_0$ and the latter at $\hat{\beta}_n$, it follows that $\hat{\beta}_n$ is one particular choice for $\tilde{\beta}_n$. Thus, the maximum score estimator is the minimizer of the empirical risk in this classification problem. The rate of estimation of $\beta^0$ depends on two crucial factors: (1) The manner in which $P(Y=1|X)$ changes across the hyperplane and (2) The distribution of $X_i$'s near the hyperplane. If the conditional probability shifts from $1/2$ rather slowly as we move away from the hyperplane, we have a `fuzzier' classification problem and estimation becomes more challenging. On the other hand, the distribution of the $X_i$'s governs the density of observed points around the hyperplane, with higher concentration of points being conducive to improved inference. As far as our knowledge goes, there is no work on the high-dimensional aspect of this model, so this paper bridges a gap in the literature.
\newline
\newline
{\bf The multinomial response discrete choice model:}  This model, which is a natural extension of its binary counterpart, arises in practice when an individual has to choose among finitely many elements, e.g. picking out a movie among several choices proposed by Netflix. 
%A standard technique in the literature is to cast this problem as a multinomial logit or probit model. Although these models are easy to interpret, they are heavily dependent on the distribution of the error $\epsilon_i$: specifically, the logit model postulates that the error has the logit distribution, while the probit assumes normal errors. 
In \cite{manski1975}, Manski also proposed an extension of the maximum score estimator for multinomial responses. We first describe the model. Assume that each individual has to choose from $m$ many alternatives, for each of which they have a utility value. Denote by $u_{i,j}$, the utility value of the $j^{th}$ alternative for the $i^{th}$ individual. Hence, $i$ will choose the $k^{th}$ alternative only if it provides them maximum utility, i.e. 
$$u_{i,k} \ge u_{i,j}  \ \text{for all } j \neq k, j \in \{1,2, \dots, m\} \,.$$
The utility values are modeled as $u_{i,j} = \m{x}_{i,j}^{\top} \beta^0 + \epsilon_{i,j}$ where $\m{x}_{i,j} \in \mathbb{R}^p$ is a vector of observable covariates and $\epsilon_{i,j}$ is an unobservable error. For notational simplicity, define an $m \times p$ matrix $\m{X}_i$ for  individual $i$ whose $j^{th}$ row is $\m{x}_{i,j}^{\top}$, the co-variate corresponding to their $j^{th}$ utility.  The $u_{i,j}$'s are not observed, but we do observe a multinomial vector $\m{y}_i \in \{0,1\}^m$  for each $i$, where 
$$\m{y}_{i,k}  = \begin{cases}
1,  \ \text{if }  u_{i,k} \ge u_{i,j}  \ \text{for all } j \neq k, j \in \{1,2, \dots, m\} \\
0, \ \text{otherwise}
\end{cases}\,.$$
In words, this vector indicates which alternative has been chosen by individual $i$. The available data on $n$ individuals are therefore the $n$ pairs 
$\{\m{y}_i, \m{X}_i\}_{i=1}^n$ which can be viewed as i.i.d replicates of a random object $(\m{y}_{1 \times m}, \m{X}_{m \times p})$, with the $j$'th row of $\m{X}$ written as $\m{x}_j^{\top}$. The response vector $\m{y}$ is related to the unobserved utility vector $(u_1, u_2, \ldots, u_m)$ through the linear model: $u_j = \m{x}_j^{\top} \beta_0 + \epsilon_j$. 
\newline
\newline
Under certain assumptions on the distribution of $(\m{y}_{1 \times m}, \m{X}_{m \times p})$ (see e.g. Assumption 2 of \cite{manski1975} or the more relaxed version, Assumption 1 of \cite{fox2007semiparametric}), which stipluates that the joint density of $(\epsilon_1, \epsilon_2, \ldots, \epsilon_m)$ conditional on $\m{X}$ is exchangeable, it can be shown that the probability of choosing the $i^{th}$ utility is driven by the ordering of the \emph{deterministic part of the utility function}. This is formalized in the \emph{rank ordering property} described below. 
\begin{assumption}[Rank ordering property]
\label{ass:rank_mult}
Define $p(j | \mb{X}, \beta)$ as the probability of the $j^{th}$ product having maximum utility under a generic regression parameter $\beta$ and conditional on $\m{X}$ being the covariate matrix: 
$$p(j | \mb{X}, \beta) = \P(\mb{x}_j^{\top}\beta +\epsilon_j \ge \mb{x}_k^{\top}\beta + \epsilon_k 
\ \forall \ k \in \{1, \dots, m\}, j \neq k \ | \mb{X})\,.$$ 
The rank ordering property says: $p(j | \mb{X}, \beta_0) \ge p(k | \mb{X}, \beta_0)$ if and only if $\mb{x}_j^{\top}\beta_0 \ge \mb{x}_k^{\top} \beta_0$. Note that the probability is taken over the joint distribution of $\{\epsilon_k\}_{k=1}^m$ given $\m{X}$. 
\end{assumption}
%According to our notation, $p_j(\mb{X}) = p(j | \mb{X}, \beta^0)$. For a detailed discussion of this assumption, the readers should take a look at \cite{fox2007semiparametric}, who also proved that the a sufficient condition for the above assumption is the exchangeability of $f(\epsilon_1, \dots, \epsilon_m | \mb{X})$. This property was first introduced in \cite{manski1975}. Briefly speaking, the assumption states that the probability of choosing $i^{th}$ utility is driven by the ordering of the deterministic part of the utility function.
This motivates the estimation of the true parameter $\beta^0$ by maximizing the following score function: 
\begin{equation}
\label{main_mult_eq}
S^{(mult)}_n(\beta)  = \frac{1}{nm(m-1)}\sum_{i=1}^n \sum_{j=1}^m \m{y}_{i,j} \left[\sum_{k \neq j}\mathds{1}\left(\m{x}_{i,j}^{\top}{\beta} > \m{x}_{i,k}^{\top}\beta\right)\right] \,.
\end{equation} 
%The idea behind the loss function is the following: as we have assumed that the ordering of $\{u_{i,j}\}$ is mainly governed by the ordering of $\{\m{x}_{i,j}^{\top}\beta^0\}$, then with high probability: 
%$$ \m{y}_{i,j} = 1 \iff u_{i,j} > u_{i,k}\ \forall \ k \neq j \iff \m{x}_{i,j}^{\top}\beta^0 > \m{x}_{i,k}^{\top}\beta^0 \ \forall \ k \neq j $$
This is a natural generalization of the maximum score to multinomial responses. The idea is to find a $\beta$ that is most commensurate with the observed data. If $j(i)$ is the observed utility for the $i$'th individual, only the $j(i)$'th term in the inner sum is relevant, and given this information, we look for $\beta$ that makes the deterministic part of the $j(i)$'th utility larger than those of most other utilities across all $n$ observations. 
Hence, with enough data, any maximizer of $\S^{(mult)}_n(\beta)$ can be expected to be close to $\beta^0$ with high probability under Assumption \ref{ass:rank_mult}. 
%Manski, in fact,  proposed a more general version of this loss function, which in the later formulation by \cite{yan2019semiparametric} is given by: 
%\begin{equation}
%\label{weighted_mult_loss}
%\tilde \S^{(mult)}_n(\beta)  = \frac{1}{nm(m-1)}\sum_{i=1}^n \sum_{j=1}^m \m{y}_{i,j} \left[\omega \left(\sum_{k \neq j}\mathds{1}\left(\m{x}_{i,j}^{\top}{\beta} > \m{x}_{i,k}^{\top}\beta\right)\right)\right]\,,
%\end{equation}
 %where $\omega(1) \ge \omega(2) \ge \dots \ge \omega(j-2) > \omega(j-1)$ is a monotone sequence of weights.
 % As mentioned cogently in \cite{yan2019semiparametric}, this is not exactly same as the one proposed in \cite{manski1975}, but they are equivalent under linear transformation. For more details, one should read \cite{yan2019semiparametric}. 
%In our paper, for simplicity, we adhere to loss \eqref{main_mult_eq} instead of \eqref{weighted_mult_loss}. 

We also note that this directly reduces to the binary response model presented at the beginning, when $m=2$. In this case, there are only two utility values for the $i^{th}$ individual who chooses the first option only if $u_{i,1} > u_{i,2}$. Now, 
\begin{equation*}
u_{i,1} - u_{i,2} > 0 \iff (\m{x}_{i,1} - \m{x}_{i,2})^{\top}\beta_0 + (\epsilon_{i,1} - \epsilon_{i,2}) > 0\,,
\end{equation*}
and hence, taking $X_i = (\m{x}_{i,1} - \m{x}_{i,2})$, $\epsilon_i =  (\epsilon_{i,1} - \epsilon_{i,2})$ and $Y_i$ to be a binary response which takes value $1$ when item $1$ is chosen and $0$ otherwise, we recover the binary response model as mentioned in equation \eqref{eq:binary_response} %and $\S^{(mult)}_n(\beta)$ reduces to $\S_n(\beta)$%
via a simple linear transformation.  

There is a vast literature, especially in economics, which deals with the discrete choice model, although most of it is confined to the binary response model. Lee \cite{lee1995semiparametric} extended the analysis of \cite{klein1993efficient} for the binary response model to the multinomial case under an appropriate version of the assumptions in the latter paper to obtain a consistent and asymptotically normal semi-parametric efficient estimator. Fox (\cite{fox2007semiparametric}) proved the consistency of the maximum score estimator for the multinomial response model under a partially missing data assumption, where the chosen utility and a subset of alternative utilities are observed, without Manski's assumption of conditionally independent  errors (Assumption 2 of \cite{manski1975}). Recently, Yan (\cite{yan2019semiparametric}) extended the analysis of Horowitz (\cite{horowitz1992}) to establish asymptotic normality of a kernel smoothed estimator in the multinomial model. 

To the best of our knowledge, all previous work on the binary as well as the multinomial discrete choice model has been done under the setting of fixed dimensional covariates and in the latter model, also under a fixed number of utilities. Our motivation for studying the maximum score estimator in these models is two-fold. Firstly, the score estimator works under very mild conditions on the underlying data generating mechanisms (particularly, through the flexible dependence of the error given the covariate), and is therefore robust to model-misspecification as a consequence of which it has attracted the attention of multiple researchers in both economics and statistics. Through a study of this model in growing dimensions, and results on the concentration properties of the estimator as well as minimax estimation rates in this problem, we provide a novel and interesting direction to the literature on this topic, which we hope will be carried forward by others interested in this genre of problems. Second, from a purely statistical point of view, the score estimator is one of the classic examples of non-regular estimators which arise either through the optimization of criterion functions that are discontinuous in the parameter (note the indicator functions involved in $S_n$ and $S_n^{mult}$), or through optimization problems where the estimator falls on the boundary of the parameter space (e.g., in modern statistical problems involving convex optimization where the estimator lies on a face of a convex cone or more generally a convex set). Such estimators have been known in the literature from as early Chernoff's work in the 1960s (e.g. see \cite{chernoff1964estimation}), and were investigated through an integrated approach by Kim and Pollard \cite{kp90}, in the specific setting of `cube-root asymptotics' -- the estimators treated in that paper demonstrated an $n^{1/3}$ convergence rate and non-Gaussian limits -- and an important example in that paper was the maximum score estimator. There have been a variety of related developments but all work in this arena has also been in the fixed dimension paradigm. Our current study of the score estimator, to the best of our knowledge, is the first example of a systematic study of a non-regular estimator in growing dimensions. While concentration and minimaxity properties have been dealt with quite thoroughly, inferential questions remain open, and we view our contributions as an important foray into hitherto uncharted territory, but we are only scratching the tip of an iceberg. 
%In this paper, we also investigate the rate properties of the maximum score estimator when both the dimension of the covariate and the number of utilities can grow with $n$. As before, we consider two scenarios: $n,m,p \uparrow \infty, p/n \rightarrow 0$ and $n, m, p \uparrow \infty, p \gg n$ . 
\newline
\newline
{\bf Major findings: } Here we articulate our findings and give a brief description about the organization of the rest of the paper. We note at the outset that the $\ell_2$ metric is a natural measure of distance in this problem since the angle between two unit-norm vectors, which measures their directional divergence, is a function of the $l_2$ norm of their difference. 

Section \ref{rate-p-less-n} deals with the moderate growth setting i.e. $p = o(n)$, while Section \ref{p-more-n} investigates the fast growth regime: $p \gg n$. %In Section 3 we study the $\ell_2$ rate of convergence of the maximum score estimator in growing dimensions.
%We find that the rate of convergence in a fixed dimension setup is $O_p\left(n^{-\frac{1}{3}}\right)$.
In the moderate growth setting, we establish the rate of convergence of the maximum score estimator in the $\ell_2$ norm in terms of $(n,p,C_n)$ along with an exponential concentration bound, where $C_n$ is a sequence of constants appearing in Assumption \ref{ass:low_noise_binary} assumed non-increasing in terms of $n$. The magnitude of $C_n$ calibrates the difficulty of the estimation problem: sequences with $C_n$ bounded away from 0 present the hardest problems while $C_n$ decreasing to 0 makes the estimation problem easier, which reflects in the convergence rate derived in Theorem \ref{rate-manski-p-less-n}. An elaborate discussion on Assumption \ref{ass:low_noise_binary} and comparisons to a standard low noise Assumption (Assumption \ref{ass:low_noise_original}) is provided in Section \ref{sec:theory_binary}. We also establish both minimax lower and upper bounds for estimating $\beta^0$ and show that the maximum score estimator is minimax optimal up to a log factor. Furthermore, when $C_{n} \equiv C$, which is later argued to be statistically the most interesting regime, we are able to construct an alternative estimator with minimax optimal rate of convergence. 
%
%We also note that the general form of the convergence rate displayed previously in terms of a maximum of two quantities imposes an upper bound on the best possible rate for the score estimator: when $ p/n = o(C_n)$, the first term in the maximum dominates, whilst the situation reverses otherwise, and it is easy to see that the convergence rate can never exceed $p \log n/n$. This is not actually surprising; if we ignore the log factor and think of $p$ as constant, we get a rate of $n^{-1}$ which arises in problems with jump-discontinuities (change points), and cannot be exceeded in statistical parameter estimation. 

In the $p \gg n$ regime, we demonstrate that under a sparsity constraint, an appropriate penalized risk minimization method provides a super-set of the active covariates with exponentially high probability. As before, we derive an exponential concentration bound for the penalized maximum score estimator in the $\ell_2$ norm, which now depends on $s_0$, the sparsity of $\beta_0$, in addition to $n, p, C_n$. Here also, smaller values of $C_n$ translate to improved convergence rates. We derive minimax lower and upper bounds which are again discrepant up to a log factor. 
%We also establish that in this regime, a superset of the active covariates will be chosen with exponentially high probability. 
%All proofs are collected in the appendices.\\\\

In Section \ref{sec:mult} we deal with the multinomial response model. Assumption \ref{ass:rank_mult} guarantees the uniqueness of the population maximizer, while Assumptions \ref{ass:low_noise_mult} and \ref{ass:wedge_mult} are modified versions of Assumption \ref{ass:low_noise_binary} and \ref{ass:wedge_binary} tailored for the multinomial response model. Under these modified assumptions, we establish finite sample concentration bounds for the score estimator both in the slowly growing regime and the fast growing regimes. When $m = 2$, our obtained rates of convergence reduce to those obtained for the binary response model in Section \ref{sec:theory_binary}. 

In Section \ref{computation}, we present some simulation results for the binary choice model. As mentioned earlier, the maximum score estimator can not be computed in polynomial time in the dimension, owing to the discontinuity of the loss function $L_n$ defined previously in this section. A standard approach is to compute an approximate solution by minimizing a convex surrogate of the $0\mbox{-}1$ loss, as is evident from the copious amount of work in both the statistics and machine learning literatures on this topic (see e.g. \cite{friedman2001elements}): e.g., logistic regression replaces $0\mbox{-}1$ loss by the logit loss, SVM uses the hinge loss, while adaboost relies on the exponential loss. Another direction involves smoothing the $0\mbox{-}1$ loss via some distribution kernel (which makes the loss function differentiable) and computing the minimizer by some variant of gradient descent. Recently, a \emph{homotopic path following} approach to this problem has been proposed in \cite{feng2019nonregular}. We present a comparative study of three methods: SVM, logistic regression and the homotopic path following algorithm mentioned above. 
%For the regime $p/n \rightarrow 0$, we compare the estimation errors of SVM and logistic regression, but not the path homotopy algorithm, since the av. In case of $p \gg n$, we compare the performances of all the methods in terms of various metrics like estimation error, errors in variable selection etc. 
The main take away from this simulation study is that SVM performs better than logistic regression when $p/n \rightarrow 0$ under heterogeneity of errors, while the performance of the method proposed in \cite{feng2019nonregular} is comparable to SVM for $p \gg n$. As a matter of fact, the method based on homotopic path following performs somewhat better than SVM, but its run-time is also higher. 

Section \ref{discussion} presents a brief discussion of certain aspects of our work including certain natural extensions, some of which are elaborated on in the supplement, as well as future challenges of this direction of research. Section \ref{proofs} presents the proofs of two key results while the remaining proofs are relegated to the supplement in the interests of space. 

%{\bf Computational Recipes: } In Section \ref{computation}, we propose computational procedures for the maximum score estimator in growing dimensions. As our model is non-convex and non-differentiable, we cannot use techniques like gradient descent due to the presence of multiple local maxima. Our idea is to obtain a good initial estimator, to which end we propose an algorithm, namely, \emph{linear smoothing via adaboost}. Next, starting from this initial estimator, we perform gradient descent using a sigmoid loss function, a smoothed version of $0-1$ loss, to arrive at the final estimate. When $p \gg n$, we propose another method called \emph{Penalised Adaboost Method} to select the active variables. The idea behind this method is to penalize a base classifier if it uses some covariate that has not been selected previously. This produces an estimated set of active variables as well as initial values of the corresponding coefficients via linear smoothing. Now, using only the selected variables and their initial values, we apply gradient descent on the smoothed sigmoid function to obtain the final estimate. As will be seen, our computational procedures give reasonably satisfactory results.

\section{Asymptotic properties and minimax bounds} 
\label{sec:theory_binary}
We now present concentration and rate of convergence results for the maximum score estimator in the binary response model in growing dimensions.  To that end, we start with some assumptions on the distribution on $X$ and the behavior of $P(Y = 1|X)$ near the Bayes hyperplane $X^{\top} \beta^0 = 0$, which play a central role in the subsequent development. To control the behavior of $P(Y = 1| X)$, we introduce a \emph{version} of \emph{Tsybakov's low noise assumption} (\cite{mammentsy},  \cite{2004optimal}) which has been used extensively in the classification literature. For the sake of convenience of the reader we first state the regular low noise condition below.  
\begin{assumption}[Soft margin Assumption]
\label{ass:low_noise_original}
Let $\P$ denote the joint distribution of $(X, \epsilon)$ in dimension $p  \equiv p_n$. Then, with $\eta(X) := \P(Y=1 | X)$,
$$\P\left(\left| \eta(X) - \frac{1}{2}\right| \le t \right) \le C t^{\alpha} \,\,\, \forall \,\, 0 \le t \le t^*\,,$$
for some constant $C$ and $0 < t^{\star} < 1/2$ and $\alpha > 0$. 
\end{assumption}
The soft margin condition quantifies how the conditional class probability deviates from $1/2$ near the Bayes' hyperplane in terms of a smoothness parameter $\alpha > 0$. Larger values of $\alpha$ translate to sharper changes of $\eta(X)$ around the Bayes' hyperplane and correspond to easier classification problems. For reasons to be explained below, we do not work with the above condition but a slightly tuned version of it: 
\begin{assumption}[Transition condition]
\label{ass:low_noise_binary}
Let $\P$ denote the joint distribution of $(X, \epsilon)$ in dimension $p  \equiv p_n$. Then, with $\eta(x) := \P(Y=1 | X=x)$,
$$\P\left(\left| \eta(X) - \frac{1}{2}\right| \le t \right) \le C_n t\,\,\, \forall \,\, 0 \le t \le t^*\,,$$
where $\{C_n\}$ is a bounded sequence of constants, and $t^{\star}$ lies strictly between 0 and  $1/2$. 
\end{assumption}
\noindent
{\bf Discussion of Assumption \ref{ass:low_noise_binary}: } To understand the effect of $C_n$, consider the special case when $C_n = C$, a fixed constant. Then, the modified condition is just the low noise condition with smoothness parameter $\alpha = 1$. 
Next, consider a situation where $C_n$ decreases to 0 with $n$ (we view $p$ as a function of $n$). In this case, the transition of $\eta(x)$ from below $1/2-t$ on one side of the hyperplane to above $1/2 + t$ on the other side is \emph{sharper} compared to the fixed $C$ case, since the probability mass assigned by the covariate distribution to the region where $\eta(x)$ is close to $1/2$ is of a \emph{smaller} order than with fixed $C$. 
%As $t$ is a fraction, this means that $C_nt$ is extremely small and the condition $\P\left(\left| \eta(X) - \frac{1}{2}\right| \le t \right) \le C_nt \,\,\, \forall \,\, 0 \le t \le t^*$ therefore implies that with very high probability $\eta(x) \not\in [1/2 - t^*, 1/2 + t^*]$: i.e. when $x'\theta_0 < 0$, mostly $\eta(x) < 1/2 - t^*$ and when $x'\theta_0 > 0$, mostly $\eta(x) > 1/2 + t^*$. Hence, the transition of $\eta(x)$ is \emph{sharp} near the Bayes' hyperplane, 
This translates to an easier estimation problem as $n$ grows, and a corresponding improved rate of estimation: the smaller the order of $C_n$, the faster the rate. In fact, $C_n = 0$, corresponds to a jump around the Bayes' hyperplane and a best possible rate of order $1/n$ in fixed dimension. On the other hand, when $C_n$ is large, $\P\left(\left| \eta(X) - \frac{1}{2}\right| \le t \right)$ is substantially larger, which implies the presence of a fair amount of fuzziness near the Bayes' hyperplane -- there is now a substantial mass of points around the hyperplane with $\eta(x)$ values very close to $1/2$ which are hard to classify -- resulting in a slower rate of estimation. 
\newline
\newline
The \emph{transition condition} captures the intrinsic difficulty of the estimation problem in terms of the sequence of constants $\{C_n\}$ whereas the low-noise condition describes it in terms of the exponent $\alpha$ of $t$, with larger values of $\alpha$ 
corresponding to easier estimation problems (enhanced convergence rates for larger $\alpha$). Both formulations therefore capture the same phenomenon, albeit in somewhat different manner. Note that, the low noise assumption was originally formulated (\cite{mammentsy}) to deal with irregular boundaries, whereas, our condition is more naturally tuned to smooth hyperplane boundaries in discrete choice model. Our reason for favoring the modified low noise condition is that it is much more intuitive and allows a clean and integrated presentation of the minimax rates of convergence in this problem in terms of $\{C_n\}$, which does not appear to be the case with the low noise assumption. For a slightly different treatment of this problem under Assumption \ref{ass:low_noise_original}, see a previous draft of this manuscript 
\cite{mukherjee2019non}. 
\newline
\newline
We now show that the case $C_n = C$ in our transition condition arises naturally for a rich family of distributions 
%satisfy the margin condition for fixed $C_n$ (or $\alpha = 1$ in Assumption \ref{ass:low_noise_original}) 
under some natural assumptions. Observe that, the family of distributions with margin condition involving $C_n \downarrow $ is a sub-class of the family of distributions with  $C_n \equiv C$.  Assume, for example, that (a) $X \indep \epsilon$ and the density of $\epsilon$, say $f$, does not depend on $p$; (b) $f(x) \ge c_{\delta} > 0$ on $(-\delta, \delta)$ for some $\delta > 0$; (c) the density of $X^\mathsf{T}\beta^0$ is bounded by a positive number $\le k$ on $(-\delta', \delta')$ for some $\delta^{'} > 0$, with $k, \delta^{'}$  not depending on $p$. Then, for $0 \le t \le t^*$, where $\delta \wedge \delta'> (F_{\epsilon}^{-1}(0.5 + t^*) \vee -F_{\epsilon}^{-1}(0.5 - t^*))$:
\allowdisplaybreaks
\begin{align*}
\P_X\left(|\eta(X) - 0.5| \le t\right) &= \P_X\left(|F_{\epsilon}(-X^\mathsf{T}\beta^0) - F(0)| \le t\right) \\
& = \P_X\left( F_{\epsilon}^{-1}(0.5 - t) \le -X^{\top}\beta^0 \le F_{\epsilon}^{-1}(0.5 + t) \right) \\
& \le \P_X\left(|X^{\top}\beta^0| \le (F_{\epsilon}^{-1}(0.5 + t) \vee -F_{\epsilon}^{-1}(0.5 - t))\right) \\
& \le 2k(F_{\epsilon}^{-1}(0.5 + t) \vee -F_{\epsilon}^{-1}(0.5 - t)) \\
& \le \frac{2kt}{c_{\delta}}
\end{align*}
which is the condition corresponding $C_n = C = 2k/c_{\delta}$.
\newline
\newline
Using an inverse-Lipschitz type condition, one can also let $\epsilon$ depend on $X$. Suppose that the conditional distribution of $\epsilon$ given $X$ satisfies: $$|F_{\epsilon|X=x}(x^\mathsf{T}\beta^0)- 0.5| = |F_{\epsilon|X=x}(x^\mathsf{T}\beta^0)- F_{\epsilon|X=x} (0)| \ge C(|x^\mathsf{T}\beta^0| \wedge \xi)\,\, a.e. \,\, X$$ for some $C, \xi > 0$ independent of $p$, for almost surely $X \sim \P_X$. This holds, for example, if for $P_X$ almost all $x$, the conditional density $f_{\epsilon|X=x}(\zeta) \ge c > 0$ on a fixed neighborhood $(-\delta', \delta')$ around $0$, with $(c,\delta)$ not depending on $p$. The transition condition is now satisfied for fixed $C$ under the same condition on the density of $X^\mathsf{T}\beta^0$ as before.  An example of the dependence requirement of $\epsilon$ on $X$ is $\epsilon|X=x \sim N(0, 1+(\|x\|_2 \wedge 1))$. 
%\newline
%\newline
%On the other hand, Assumption \ref{ass:low_noise_binary} is satisfied with $C_n \downarrow$ in the following scenario: Suppose the distribution of $X^{\top}\beta^0$ satisfies $P(|X^\mathsf{T}\beta^0| \le t) \le C_nt$ for decreasing $C_n$, where $t < t^{\star}$. This, along with the Inverse-Lipschitz type assumption on the conditional distribution of $\epsilon|X$, yields models where the margin condition is satisfied for $C_n$. Note, however, that the upper bound on the distribution function of $|X^\mathsf{T}\beta^0|$ close to 0 forces its density to be extremely small near the origin, under a continuity assumption. This is  
%\newline
%\newline
%The second natural scenario arises when $P(Y=1|X=x)$ changes rapidly near the hyperplane. Consider the following setup: $$F_{\epsilon | X=x} (t) = \frac{1}{2} + \frac{t}{2C_n} \,\,\, \forall \,\, -C_n \le t \le C_n $$ and $0$ and $1$ outside the interval. Then, the assumption that the density of $X^\mathsf{T}\beta^0$ is uniformly bounded in $p$ in a neighborhood of 0 that is also free of $p$, leads to the margin condition with parameter $C_n$. Here, it is the rather special structure of $\eta(x)$ that gives the margin condition. 

Our next assumption regarding the marginal distribution of $X$ is that the probability of the wedge shaped region between the true hyperplane and any other hyperplane under the distribution of $X$ is related to the angle between the corresponding normal vectors. 
%{\bf We need to say how this helps to get curvature lower bound and why it is not something that we have used for painless application of Massart's theorem.}

\begin{assumption}[Distribution assumption on covariates]
\label{ass:wedge_binary}
The distribution of $X$ satisfies the following condition: $$\P_X(\s(X^\mathsf{T}\beta) \neq \s(X^\mathsf{T}\beta^0)) \ge c_1\|\beta - \beta^0\|_2 $$ for all $\beta \in S^{p-1}$, where the constant $c_1 > 0$, does not depend on $n,p$.
\end{assumption}

\noindent
{\bf Discussion of Assumption \ref{ass:wedge_binary}: }The above assumption plays a critical role in this paper, relating the underlying geometry in the problem to the probability distribution of the covariates. It is used, for example, in the below proposition, to relate the curvature of the population score function around its maximizer $\beta_0$ to the angle between $\beta_0$ and a generic unit vector $\beta$. The magnitude of the curvature plays a pivotal role in deriving the rate of convergence of Manski's estimator in both the slow and fast growth regimes (Theorems \ref{rate-manski-p-less-n} and \ref{high-dim-rate} respectively), where upper tail probabilities for $\|\hat{\beta} - \beta_0\|$ are related to upper tail probabilities for $S(\beta_0) - S(\hat{\beta})$ (which is also the difference in the population risks at these two vectors) via Assumption \ref{ass:low_noise_binary}. In that respect, this assumption can be viewed as an analogue of the compatibility or restricted eigenvalue condition in the classical high-dimensional linear regression problem, which helps convert bounds on the prediction error of the Lasso estimator to its estimation error. In this context, it is interesting to consider a specific violation of the assumption: namely, when $\P_X(\s(X^\mathsf{T}\beta) \neq \s(X^\mathsf{T}\beta^0)) = 0$ for all $\beta$ sufficiently close to $\beta_0$ in angular distance. In this case, if $X$ for example is supported on a compact domain, it is not difficult to see that one can perturb the Bayes hyperplane by small rotations, but as the corresponding wedges will not have any mass under $P_X$, there are no points available in such regions, and the Bayes hyperplane cannot be even uniquely identified. Examples of families of distributions (e.g. elliptically symmetric $X$) that satisfy Assumption \ref{ass:wedge_binary} are available in Section \ref{discussion}.  
%Our next proposition quantifies how $S(\beta)$, the population criterion function, varies around $S(\beta_0)$, its maximum value, in terms of the deviation of $\beta$ around $\beta_0$.
%We will prove that, higher the value of $\alpha$ in margin condition, higher is the curvature of $S(\beta)$ which reflects upon the faster rate of convergence:
\begin{proposition}
\label{alpha-kappa}
Under Assumptions \ref{ass:low_noise_binary} and \ref{ass:wedge_binary}, the curvature of the population score function around the truth satisfies:
$$S(\beta_0) - S(\beta) \ge  \left[\frac{\|\beta - \beta^0\|_2^2}{C_n}\mathds{1}_{\left(d_{\Delta}(\beta, \beta^0) \leq 2t^*C_n\right)} + 2t^*\|\beta - \beta^0\|_2 \mathds{1}_{\left(d_{\Delta}(\beta, \beta^0) > 2t^*C_n\right)}\right]$$ 
for all $\beta \in S^{p-1}$, where $d_{\Delta}(\beta, \beta^0) = \P_X(\s(X^{\top}\beta) \neq \s(X^{\top}\beta^0))$ and $t^*, C_n$ are same constants defined in Assumption \ref{ass:low_noise_binary}.
\end{proposition}

The proof of this proposition relies on relating $S(\beta) - S(\beta^0)$ to $\P_X(\s(X^{\top}\beta) \neq \s(X^{\top}\beta^0))$ via Assumption \ref{ass:low_noise_binary}, and the latter to $\|\beta - \beta^0\|_2$,  via Assumption \ref{ass:wedge_binary}. One takeaway from the proposition is that the excess risk is lower bounded by a dichotomous distance in terms of $\|\beta- \beta_0\|_2$ and $d_{\Delta}(\beta, \beta_0)$. For $\beta$ close to $\beta_0$ in the sense that  $d_{\Delta}(\beta, \beta_0)$ is small relative to $C_n$, we have a quadratic curvature whose sharpness is determined by the magnitude of $C_n$, while for $\beta$ away from $\beta_0$, the curvature is linear.  As we will see below, the dichotomous nature of the distance imposes a natural lower bound on the estimation error of the maximum score 
estimator, irrespective of how small $C_n$ is. 

%Note that smaller values of $C_n$ produce larger slope of change of $\eta(X)$ near Bayes' hyperplane, which translates to a larger curvature of $S$ around its maximizer. Larger curvatures correspond to greater sensitivity of $S$ to $\beta$ and translate to better convergence rates, as seen in the subsequent theorems. 

\begin{remark}
Note that when $C_n = C$ fixed and assuming without loss of generality $2t^*C > 1$, we conclude: 
$$S(\beta^0) - S(\beta) \ge \frac{1}{C} \|\beta - \beta^0\|_2^2$$
for all $\beta \in S^{p-1}$. This same condition can be achieved using Assumption \ref{ass:low_noise_original} with $\alpha = 1$. 
\end{remark}

%In what follows, we will denote by $\mathcal{P}(C,c_1,\alpha)$ denote the class of distributions of $(X,\epsilon)$ in dimension $p_n + 1 \equiv p + 1$ that satisfy Assumption \ref{ass:low_noise_binary} with constant $C$ and exponent $\alpha$, Assumption \ref{ass:wedge_binary} with constant $c_1$, and with $\med(Y^*|X) = X^{\top}\beta$ for some $\beta \in S^{p-1}$. 

\subsection{\bf{Rate of convergence when: $p/n \rightarrow 0$}}\label{rate-p-less-n}
We first establish a rate of convergence for $\hat{\beta}$. 
%As mentioned earlier, the higher the value of $\alpha$, the faster the
%convergence rate, which is reflected in the following theorem:
\begin{theorem}
\label{rate-manski-p-less-n}
Let $\hat \beta_n$ and $\beta_0$ are the maximizer of $S_n(\beta)$ and $S(\beta)$ respectively. Then under Assumptions \ref{ass:low_noise_binary} and \ref{ass:wedge_binary}, for some constant $K > 0$ (not depending on $n, p$):
$$\P\left(\left(\frac{r_n}{\sqrt{C_n}} \wedge r_n^2\right)\|\hat{\beta}_n - \beta^0\|_2 \ge Ky\right) \le 2e^{-y}$$ for all $y \ge 1$, 
where: 
$$r_n = \left(\frac{n}{p\sqrt{C_n}\log{(n/pC^2_n)}}\right)^{1/3} \wedge \left(\frac{n}{p\log{(n/p)}}\right)^{1/2} \,.$$
This implies that, 
$$\sup_{\beta \equiv \beta(P)}\mathbb{E}_{\beta}\left(\left(\frac{r_n}{\sqrt{C_n}} \wedge r_n^2\right) \|\hat{\beta}_n - \beta\|_2\right) \le  K_1 \,, $$ 
where $K_1 > 0$ is some constant which depends on the model constants $t_*, c_1$ introduced in the assumptions and some other universal constants. Note that the supremum in the above 
display is taken over all distributions $P$ corresponding to binary response models 
satisfying Assumptions \ref{ass:low_noise_binary} and \ref{ass:wedge_binary} for some regression parameter $\beta \in \mathcal{S}^{p-1}$ (viewed as a functional of $P$) but with $t^*, c_1$ held fixed. 
\end{theorem}
\noindent

\begin{remark}
\label{rem:low_dim_rate}
Note that our rate of convergence depends on three parameters $(n,p,C_n)$. To understand the implications of the obtained expression for the rate, assume initially that $C_n$ is a constant (statistically of primary interest as discussed in Section \ref{intro}), assumed without loss of generality to be 1. Then the value of $r_n$ reduces to: 
$$r_n = \left(\frac{n}{p\log{(n/p)}}\right)^{1/3} \wedge \left(\frac{n}{p\log{(n/p)}}\right)^{1/2} =  \left(\frac{n}{p\log{(n/p)}}\right)^{1/3}$$
and 
$$\frac{r_n}{\sqrt{C_n}} \wedge r_n^2 = r_n \wedge r_n^2 = r_n = \left(\frac{n}{p\log{(n/p)}}\right)^{1/3} \,.$$
Hence, up to a log factor, we recover the analogue of the cube-root rate for growing dimension. 
\\

One may wonder what is the best possible rate that can be obtained from the above expression. An inspection of the rate expression immediately implies that we cannot improve upon $n/(p\log{(n/p)})$, the high dimensional rate analogue of change-point estimation. Some more insight can be gleaned by ignoring the log-factor in the rate expression. In that case,  
$$r_n \approx \left(\frac{n}{p\sqrt{C_n}}\right)^{1/3} \wedge \left(\frac{n}{p}\right)^{1/2}\,$$ 
and the rate of convergence based on this approximation is given by  
\begin{align*}
\frac{r_n}{\sqrt{C_n}} \wedge r_n^2 & \approx \left(\frac{n}{pC_n^2}\right)^{1/3} \wedge \left(\frac{n}{pC_n}\right)^{1/2} \wedge \left(\frac{n}{p\sqrt{C_n}}\right)^{2/3} \wedge \left(\frac{n}{p}\right) \\
& = \left(\frac{n}{pC_n^2}\right)^{1/3}  \wedge \left(\frac{n}{p}\right) 
\end{align*}
which is shown to be the minimax optimal in Theorem \ref{minimax-lower-bound}. The above equality follows from the observation that, if $C_n \ge p/n$, then $(n/pC_n^2)^{1/3}$ is minimum among the four terms, while $(n/p)$ is the minimum otherwise. This indicates that the rate of convergence improves with decreasing $C_n$, but only up to $n/p$ modulo a log factor. 

Alternatively, one can study the exact expression for the rate by taking 
special but natural choices for $p, C_n$ in terms of $n$. Concretely, let 
$p \sim n^{\tau}$ and $C_n \sim n^{-\lambda}$ for $0 < \tau < 1$ 
and $\lambda > 0$. Note that $C_n$ is of order larger than $p/n$ when $\lambda > 1-\tau$, in which case some simple algebra shows the rate of convergence $(r_n/\sqrt{C_n}) \wedge r_n^2$ to be 
$((n/pC_n^2)/\log(n/pC_n^2))^{1/3}$. On the other hand when $\lambda 
\leq 1-\tau$, i.e. $C_n$ is of the same or lower order than $p/n$, the rate of convergence becomes $(n/p)/\log(n/p)$.
\end{remark}

The proof of Theorem \ref{rate-manski-p-less-n} relies on a concentration inequality (Theorem 2 from \cite{massart2006risk}) to obtain a bound on the excess risk $S(\hat{\beta}) - S(\beta^0)$, which, along with Assumption \ref{ass:wedge_binary}, yields a concentration bound on $\|\hat{\beta} - \beta^0\|_2$. A natural question that arises here is whether the logarithm in the above rate, which arises from the effect of growing dimension on the shattering numbers of the linear classifiers involved, can be dispensed with. While it is unclear whether the exact $(pC^2_n/n)^{1/3} \vee (p/n)$ rate is achievable, we demonstrate, in what follows, that for $C_n = C$, it is possible to construct an estimator whose rate of convergence is $(p/n)^{1/3}$ under the following additional assumption.

\begin{assumption}
\label{ass:A2:upper}
We impose some further constraints on the distribution of $X$ and the population score function: 
\begin{enumerate}
\item The distribution of $X$ satisfies 
$$ \P_X(\s(X^\mathsf{T}\beta) \neq \s(X^\mathsf{T}\beta^0)) \le C'\|\beta - \beta^0\|_2 $$ for all $\{\beta: \|\beta - \beta_0\|_2 \le 1\}$, where the constant $ C' > 0$ does not depend on $n$ and $p$.
\item For some small $u_0 > 0$, 
$$ S(\beta_0) - S(\beta) \le u_+\|\beta - \beta^0\|_2^2 $$ for all $\{\beta: \|\beta - \beta_0\|_2 \le u_0\}$, where the constant $ u_+ > 0$ does not depend on $n$ and $p$.
\end{enumerate}
\end{assumption}
The construction of the estimator can be briefly described as follows: Generate (enough) points randomly on the surface of the unit sphere, such that with high probability some of the generated points are in a sufficiently small neighborhood of $\beta^0$. Then, maximize the empirical score function on the generated points. We show in the following theorem that this empirical maximizer converges to the truth at rate $(p/n)^{1/3}$: 
\begin{theorem}
\label{multistage-estimator}
Suppose the margin condition (Assumption \ref{ass:low_noise_binary}) is satisfied for $C_n = C$ fixed, and that Assumptions \ref{ass:wedge_binary} and \ref{ass:A2:upper} hold. Then, there exists an estimator $\tilde{\beta}$, which can be constructed by the above recipe [with technical details of the construction available in the proof], such that $$\|\tilde{\beta} - \beta^0\|_2 = O_P\left(\left(p/n\right)^{\frac{1}{3}} \right)\,.$$
\end{theorem}
\noindent
\begin{remark}
Assumption \ref{ass:A2:upper} (2) as well as the construction of the grid estimator take into account the the fact that $C_n = C$ fixed. In that sense, the new estimator is not adaptive, whereas the maximum score estimator is agnostic to the value of $C_n$. We believe that the log factor in the  convergence rate is the price paid for adaptivity. For more insight into the Assumption \ref{ass:A2:upper}, see Section \ref{discussion}.
\end{remark}
%{\bf Remark:} One may wonder why our estimator is defined in terms of a maximization problem over a random collection of points on the sphere. In fact, the same convergence rate could have been established by maximizing the empirical score function over a deterministic $\left(p/n\right)^{\frac{1}{3}}$ cover of $S^{p-1}$. The reason why we use the random collection is that it is quite easy to generate points from the uniform distribution on the surface of a Euclidean sphere. Specifying a $(p/n)^{1/3}$ deterministic cover (i.e. the centers of the $(p/n)^{1/3}$ radius balls) is substantially harder.  
%In contrast, the maximum score estimator is completely agnostic to $\alpha$ and the extra logarithmic term in its rate of convergence could reflect the price of adaptivity in growing dimensions. 
Finally, we show that the generic minimax lower bound for this estimation problem (i.e the $C_n$'s are not restricted to be constant)  is $\left(pC^2_n/n\right)^{1/3} \vee \left(p/n\right)$ i.e. we cannot estimate the linear discriminator at a better rate without more assumptions:

\begin{theorem}[\textbf{Minimax Lower bound}]
\label{minimax-lower-bound}
%We have the following lower bound on the minimax risk of the estimation error of $\beta^0
We have : 
$$\inf_{\hat{\beta}_n}\sup_{\beta \equiv \beta(P)}\mathbb{E}_{\beta}\left(\|\hat{\beta}_n - \beta^0\|_2^2\right) \ge K_{L}\left[\left(\frac{pC^2_n}{n}\right)^{2/3} \vee \left(\frac{p}{n}\right)^2\right]$$
for some constant $K_{L}$ that does not depend on $(n,p)$. For $C_n = C$ fixed, the lower bound is of the order $(p/n)^{2/3}$. The supremum is taken over the same class of distributions as in Theorem \ref{rate-manski-p-less-n}. 
\end{theorem}

\begin{remark}
\label{minimaxremark} 
The proof of the above result relies on constructing competing models from the collection of distributions that approach each other at the optimal rate, $(p C_n^2/n)^{1/3}$. The core challenge lies in constructing these alternative models with sufficient care, and then invoking Assouad's lemma (e.g. see chapter 2 of \cite{tsybakov2009introduction}) to establish the rate. The same minimax rate is true for the smaller class of distributions formed by intersecting $\mathcal{P}$ with the class of distributions satisfying Assumption \ref{ass:A2:upper} for some positive constants $C', u_0, U_{+}$, since the local alternatives constructed in the proof satisfy this assumption as well. Therefore, the grid estimator is minimax rate optimal for this smaller class of distributions. 
\end{remark} 
%\newline
%For the case $\alpha > 1$ which Part 2 deals with, we establish the minimax bound for a slightly enhanced class of distributions which satisfies a relaxed version of the soft margin condition, but which approaches the class $\mathcal{P}$ as $n$ goes to $\infty$. This slight enhancement of the class facilitates the application of the tools used to derive the minimax bound. While this not translate to a lower bound for the class $\mathcal{P}$, the upper bound on the expected $\ell_2$ distance derived in Theorem \ref{rate-manski-p-less-n}, which is within a log factor of the lower bound in Part 2, strongly suggests that the same minimax lower bound holds for the class $\mathcal{P}$.  
%\end{remark}

\subsection{\bf{Rate of convergence when $p \gg n$}}\label{p-more-n}
We now turn to the case where $p$, the dimension of the covariate vector, is larger than $n$. In this case, meaningful estimation and inference is only possible under structural assumptions on $\beta^0$ that regulate its complexity relative to the size of the data, and any meaningful estimation procedure needs to incorporate this constraint. Usually, such structural assumptions are handled by imposing a penalty on the underlying loss function. 
The most natural structural constraint on a regression type parameter is one of sparsity, i.e. only a small subset of the co-ordinates of $\beta_0$ influence the response (i.e. are different from zero). In the high-dimensional linear regression or GLM framework, the natural loss function is convex and the standard approach is to penalize the (convex) $\ell_1$ norm of the parameter, which gives rise to a clean convex optimization problem with a well-characterized solution (see \cite{tibshirani1996regression}, \cite{greenshtein2004persistence}, \cite{van2014asymptotically}, \cite{buhlmann2011statistics}, \cite{bickel2009simultaneous}, \cite{miolane2018distribution} and references therein). The corresponding optimizers are seen to have desirable statistical properties, e.g. consistency in various norms, minimax convergence rates and so-forth. Furthermore, $\ell_1$ penalization is a natural convex relaxation of $\ell_0$ penalization which is the most direct approach to the sparsity constraint. Another key feature of high dimensional inference is model selection. Under the sparsity constraint, most variables are inactive and a good model selection algorithm needs to include the active set with high probability but relatively few inactive variables. Though model selection/ feature selection in the high-dimensional linear regression model has been studied extensively over the past two decades (e.g. see \cite{zhao2006model}, \cite{huang2008adaptive}, \cite{wei2010consistent}, \cite{yuan2006model}, \cite{zhang2008sparsity} and references therein), the problem remains relatively unaddressed in the classification set-up.

Be that as it may, the optimization problem that produces the maximum score estimator is not only non-differentiable and non-convex, it is actually discontinuous and therefore adding a convex penalty like $\ell_1$ affords no computational advantage. While one possible route is to use an $\ell_1$ penalized version of a kernel smoothed loss function $L_n$ (following the line of work of Horowitz \cite{horowitz1992}), an approach that has recently been adapted by \cite{feng2019nonregular} in related problems, our goal in this section is to \emph{understand the behavior of the primal non-regular score estimator in high dimensions under minimal assumptions}. In what follows, we therefore penalize the score function in a way that is amenable to a proper analysis and produces a sparse estimator with near-optimal rate and a desirable screening property. A smoothed estimator can possibly yield a better convergence rate along with computational benefits but will require substantially stronger assumptions on the model. Recall, for example, that the smoothed score estimator in the fixed $p$ setting as studied in \cite{horowitz1992} does converge at a faster than $n^{1/3}$ rate to a Gaussian limit, but the model assumptions required to make this work are significantly stronger than Manski's original assumptions as well as Kim and Pollard's \cite{kp90}.

In what follows, we use the structural risk minimization (SRM) approach introduced in \cite{vapnik1974theory} for variable selection and estimation in this regime, which is closely related to $\ell_0$-penalized risk minimization or the best subset selection problem. Briefly speaking, the SRM approach consists of the following steps: 

\begin{enumerate}
\item Start with a large class of functions over which the loss function will be minimized. 
\item Divide this class into nested subsets of increasing complexity, and find empirical risk minimizer for each of these subsets. 
\item Add a penalty (here denoted by $\pen$) based on the complexity of the subclass to the minimum empirical risk for that subclass and return the classifier (and its corresponding subclass) with \emph{minimum penalized empirical risk}.
\end{enumerate}

The first step generally ensures that there is no bias (or very low bias) in the estimation problem. If one starts with a large function class, it is more likely that the population minimizer will be close (if not identical to) the minimizer within the selected class. 
%On the contrary, if one starts with a small class of functions and If the true minimize is not in the class, then at best one can estimate the best possible function in the class, i.e. the population minimizer of the class, which is not the truth. Hence there will always be a bias which cannot be avoided. 
But though bias can be largely eliminated by in this manner, the process of searching over a large function class incurs high variability and can lead to pessimistic convergence rates. Therefore, one needs to optimize the bias-variance trade-off, which happens over steps two and three. In step two, nested subsets are considered, hence the minimum value of the empirical risk keeps decreasing as the nesting (complexity) increases. The role of the  penalty function is to stabilize the bias-variance trade-off and strike a balance between risk minimization and complexity. The nature of the penalty is typically related to the complexity of the class of functions (complex classes are penalized at higher levels) as well as to the structure of the problem. For parametrically specified classes, one may use the $\ell_0, \ell_1$ or a more general $\ell_p$ norm of the parameter, or variants (e.g. Mallow's CP, AIC, BIC) as a notion of complexity, or may resort to other notions like VC dimension (see e.g. Chapter 8 of \cite{massart2007concentration} and the references therein). 
%Estimation of high-dimensional parameters using some penalty function is a well-studied topic in statistics, dating back, at least, to the work of best subset selection by Hocking et.al. \cite{hocking1967selection}. The consistency of the population risk of the best linear predictor under an $\ell_0$ constraint in a general regression set-up was theoretically studied in Greenshtein and Ritov \cite{greenshtein2004persistence}, while minimax rates for $\ell_q$ constrained estimation [$q \in [0,1]$] in high dimensional linear regression (with a known level of sparsity) were more recently calculated by Raskutti et.al. \cite{raskutti2011minimax}.  A key feature of high dimensional inference is model selection. Under the sparsity constraint, most variables are inactive and a good model selection algorithm needs to include the active set with high probability but relatively few inactive variables. Though model selection/ feature selection in the high-dimensional linear regression model has been studied extensively over the past two decades (e.g. see \cite{zhao2006model}, \cite{huang2008adaptive}, \cite{wei2010consistent}, \cite{yuan2006model}, \cite{zhang2008sparsity} and references therein), the problem remains relatively unaddressed in the classification set-up. Our results are based on maximizing a penalized loss function. We search over all possible model upto a suitable upper bound on the sparsity (will be specified later) and then choose the best one using the penalized loss function. We have the following sparsity assumption on the model: 

We now describe the details of the implementation of our SRM based method. We start by articulating our assumption on the sparsity of the Bayes' hyperplane: 
\begin{assumption}[Sparsity Assumption]
\label{ass:sparsity_binary}
There exists $s_0$ with $\|\beta^0\|_0 \le s_0$, where $s_0$ depends on $n,p$ in such a way that $\frac{s_0\log{p}}{n} \rightarrow 0$ as $n \rightarrow \infty$.
\end{assumption}

Under the above assumption, it is reasonable to search among all models with sparsity (by which we mean the number of active coefficients) bounded by  $C_1\lfloor n/\log{p} \rfloor$ for some universal constant $C_1$. For mathematical simplicity, we take $C_1=1/4$. Let $\mathscr{M}_i$ be the collection of all models with sparsity bounded by $i$ for $1 \le i \le \lfloor n/\log{p} \rfloor$, i.e.: 
$$\mathscr{M}_i = \{\text{all models such that } \|\beta^0\|_0 \le i \}\,.$$
Define $\mathscr{M}$ be the collection of all admissible models, i.e. $$\mathscr{M} = \cup_{i=1}^{\lfloor n/\log{p} \rfloor} \mathscr{M}_i \,.$$ 
Also define:
\begin{enumerate}
\item $\hat{\beta}_m = \argmin_{\beta: \|\beta\|_0 \le m}\left[ -\S_n(\beta) \right]$
\item $\hat{m} = \argmin_{1 \le m \le \lfloor n/\log{p} \rfloor} \left[-\S_n(\hat \beta_m) + \pen(\mathscr{M}_m)\right]$
\item $\beta^0_m = \argmax_{\beta: \|\beta\|_0 \le m} S(\beta)$
\vspace{0.05in}
\item $V_m = $ VC dimension of the collection $\mathscr{M}_m$. This is of the order $m \log (ep/m)$. 
\end{enumerate}
By the SRM principle, the best possible estimate is given by $\hat \beta_{\hat m}$. For the model collection $\mathscr{M}_i$, we use the penalty 
$$\pen(\mathscr{M}_i) = 2K\left[\left(\frac{V_i\sqrt{C_n}\log{(n/V_i\sqrt{C_n})}}{n}\right)^{2/3} \vee \left(\frac{V_i\log{(n/V_i)}}{n}\right)\right]$$
where $K$ is some absolute constant. Up to a (neligible) logarithmic term, the penalty function is proportional to $V_i$, the VC dimension of the model $\mathscr{M}_i$, which captures the richness of this collection. %\DM{This estimator is not adaptive in a sense that we need to know $\alpha$ beforehand to get the apt. penalty function.} 
The following theorem provides a finite sample concentration bound of our estimator:
\begin{theorem}
\label{high-dim-rate}
Let $\hat \beta_{\hat m}$ and $\beta_0$ denote the penalized empirical minimizer and population minimizer of the binary choice model respectively. Then under assumptions \ref{ass:low_noise_binary}, \ref{ass:wedge_binary} and \ref{ass:sparsity_binary}, there exist constants $\Sigma, K_1>1, K_2$ (which are independent of $n,d,s_0$) such that for all $t \ge 1$: 
\begin{equation*}
   \P\left(\left(\frac{r_n}{\sqrt{C_n}} \wedge r_n^2\right) \left\|\hat \beta_{\hat m} - \beta^0\right\|_2 > K_1 + K_2s_n t\right) \le 2\Sigma e^{-t}
\end{equation*}
where,
$$r_n = \left(\frac{n}{V_{s_0}\sqrt{C_n}\log{(n/V_{s_0}C^2_n)}}\right)^{1/3} \wedge \left(\frac{n}{V_{s_0}\log{(n/V_{s_0})}}\right)^{1/2}$$
and $s_n$ is a specific sequence of constants going down to 0 (with details available in the proof). 

As a consequence of the exponential tail bound, one can establish the following upper bound on the minimax risk: 
$$\sup_{\beta \equiv \beta(P)}\E_{\beta}\left(\left(\frac{r_n}{\sqrt{C_n}} \wedge r_n^2\right) \|\hat \beta_n - \beta\|_2\right) \le K_1 + 2\Sigma K_2\sqrt{s_n} \,,$$
where 
the supremum in the above 
display is taken over all distributions $P$ corresponding to binary response models 
satisfying Assumptions \ref{ass:low_noise_binary}, \ref{ass:wedge_binary} and \ref{ass:sparsity_binary} with some regression parameter $\beta \in \mathcal{S}^{p-1}$ (viewed as a functional of $P$) with $\ell_0$ norm bounded below by $s_0$. 
\end{theorem}

\begin{remark} 
\label{rem:high-dim-rate}
A discussion similar to Remark \ref{rem:low_dim_rate} is in order. Here, the rate of convergence depends on four parameters $(n,p,C_n, s_0)$. Assuming $C_n = 1$, the value of $r_n$ becomes: 
$$r_n = \left(\frac{n}{V_{s_0}\log{(n/V_{s_0})}}\right)^{1/3} \wedge \left(\frac{n}{V_{s_0}\log{(n/V_{s_0})}}\right)^{1/2} =  \left(\frac{n}{V_{s_0}\log{(n/V_{s_0})}}\right)^{1/3}$$
and recalling that $V_{s_0} \asymp s_0 \log (ep/s_0)$, 
$$\frac{r_n}{\sqrt{C_n}} \wedge r_n^2 = r_n \wedge r_n^2 = r_n = \left(\frac{n}{V_{s_0}\log{(n/V_{s_0})}}\right)^{1/3} \asymp \left(\frac{n}{s_0\log{p}\log{n}}\right)^{1/3}\,,$$
where the last asymptotic equivalence while not always being true, nonetheless holds for most common scenarios. As an example, if we take $s_0 = n^{\gamma}, p = e^{n^{\delta}}$ for some $0 < \gamma, \delta < 1$ with $\gamma + \delta < 1$, the equivalence is valid. The condition $\gamma + \delta  < 1$ is forced by Assumption \ref{ass:sparsity_binary}. 
\\

It is immediate that the rate of convergence of our estimator cannot be faster than $n/\left(V_{s_0}\log{(n/V_{s_0})}\right)$. As before, one can gain useful insights by ignoring the log-factors in the rate expression. Thus, 
$$r_n \approx \left(\frac{n}{V_{s_0}\sqrt{C_n}}\right)^{1/3} \wedge \left(\frac{n}{V_{s_0}}\right)^{1/2}\,$$ 
and hence
\begin{align*}
\frac{r_n}{\sqrt{C_n}} \wedge r_n^2 & \approx \left(\frac{n}{pV_{s_0}C_n^2}\right)^{1/3} \wedge \left(\frac{n}{V_{s_0}C_n}\right)^{1/2} \wedge \left(\frac{n}{V_{s_0}\sqrt{C_n}}\right)^{2/3} \wedge \left(\frac{n}{V_{s_0}}\right) \\
& = \left(\frac{n}{V_{s_0}C_n^2}\right)^{1/3}  \wedge \left(\frac{n}{V_{s_0}}\right) \,.
\end{align*}
As in the case of slowly growing regime, this rate is also shown to be minimax optimal in Theorem \ref{minimax-high-lower}. The last equality follows from the fact that, if $C_n \ge V_{s_0}/n$, then $(n/V_{s_0}C_n^2)^{1/3}$ is the minimum among the four terms, otherwise $(n/V_{s_0})$ is the minimum. This implies that the rate can be made faster by decreasing the value of $C_n$, but cannot be improved upon $(n/V_{s_0})$ (up to log factors).

\end{remark}

%This proof also relies on concentration results from the the ERM literature and the magnitude of the VC dimension of linear classifiers with sparsity $\le s_0$ which is of the order $s_0\log{p}$.
As an immediate corollary of the above theorem, we establish that a superset of true model will be selected with high probability under an appropriate beta-min condition: 
\noindent
\begin{corollary}
\label{cor-high-dim-rate}
Suppose the minimum non-zero absolute value of $\beta^0$ satisfies the following bound: $$\beta^0_{\min} \ge (K_1 + K_2)\left(\frac{r_n}{\sqrt{C_n}} \wedge r_n^2\right)^{-1}$$ with $r_n$ is as in Theorem \ref{high-dim-rate}. Then under Assumptions \ref{ass:low_noise_binary},\ref{ass:wedge_binary} and \ref{ass:sparsity_binary}, we have: 
\begin{equation*}
\P\left(\hat m \supseteq m_0\right) \ge 1 - 2\Sigma e^{-\frac{1}{\sqrt{s_n}}} \,.
\end{equation*}
where $m_0$ is the true active set. This probability goes to 1 exponentially fast in $n$. 
\end{corollary}
\begin{remark}
Note that for $C_n = C$ fixed, the lower bound on $\beta^0_{\min}$ is proportional to $\left((s_0\log{p}\log{n})/n\right)^{1/3}$ which is same as the rate of convergence of $\hat \beta_{\hat m}$ (see Theorem \ref{high-dim-rate}). This should be compared to the $\beta^0_{\min}$ condition derived from the $\ell_2$ convergence analysis in high dimensional linear regression which is $(s_0 \log{p} /n)^{1/2}$. The slower convergence rate in this problem requires a more pronounced separation of the active coefficients of $\beta^0$ from the inactive ones in comparison to standard linear regression, to guarantee the screening property. 
\end{remark}
Our next result provides a lower bound on the minimax error rate.

%
%
%\begin{theorem}[\textbf{Minimax Lower bound}]
%\label{minimax-lower-bound}
%
%\begin{enumerate}
%\item For $\alpha = 1$, define $\mathcal{P} = \cup_{C,c_1}\mathcal{P}(C,c_1,1)$. Then, we have the following lower bound on the minimax risk:
%$$\inf_{\hat{\beta}_n}\sup_{\mathcal{P}}\mathbb{E}\left(\|\hat{\beta}_n - \beta^0\|_2^2\right) \ge K_{L}\left(\frac{p}{n}\right)^{2/3} \,, $$ for some constant $K_{L}$ that does not depend on $(n,p)$. 
%\item 
%\end{enumerate}
%\end{theorem}

\begin{theorem}
\label{minimax-high-lower}
We present our minimax lower bound result for the fast growth regime $p \gg n$: 
$$\inf_{\hat{\beta}} \sup_{\beta \equiv \beta(P)}\E_{\beta} \left(\|\hat{\beta} - \beta^0\|^2_2\right) \ge \tilde{K}_{L}\left[\left(\frac{s_0\log{(p/s_0)}C^2_n}{n}\right)^{2/3} \vee \left(\frac{s_0\log{(p/s_0)}}{n}\right)^2\right]$$
for some constant $\tilde{K}_{L} > 0$ not depending on $(n,p,s_0)$. For the case $C_n = C$ fixed, the lower bound is of the order of $\left(\frac{s_0\log{(p/s_0)}}{n}\right)^{2/3}$. The supremum is taken over the same class of distributions as in Theorem \ref{high-dim-rate}. 
\end{theorem}

\begin{remark}
As with the minimax lower bound proof in the moderate growth case, the proof of this theorem also relies on the construction of a sequence of competing models that approach one another, along with Fano's inequality (Chapter 2 of \cite{tsybakov2009introduction}). Similar to the moderate growth regime, by comparing the lower bound above to the rate of the score estimator in Remark \ref{rem:high-dim-rate}, we find that the former is better only by a logarithmic factor, which suggests that the penalized maximum score estimator is almost minimax optimal. 
\end{remark} 

%\noindent
%The proof depends on excess risk concentration results from the ERM (empirical risk minimization) literature along with the beta-min condition stated in Assumption (A4). The above theorem quickly leads to the observation that the empirical loss function will be minimized at some super-set of the true model $m^*$ with exponentially high probability, as stated in the following corollary:
%\begin{corollary}
%\label{model selection corollary}
%Denote the empirically selected model by $\hat{m}$, i.e. $\hat{m} = \argmin_{m \in \mathcal{M}} L_n(\hat{\beta}_m)$. Then, under Assumptions (A1 - A4), we have
%$$\mathbb{P}(\hat{m} \nsupseteq m^*)  \le 2e^{-s_0\log{(ep/s_0)}}$$
%\end{corollary}
%\noindent
%Now, consider our estimator for $\beta^0$ given by 
%$$\hat{\beta}_{\hat{m}} \overset{\Delta} = \hat{\beta} = \argmax_{\beta: \|\beta\| = 1, \|\beta\|_0 \leq s_0} \,S_n(\beta) \,.$$ 
%
%
%
%
\section{Multinomial discrete choice model}
\label{sec:mult} The score function for the multinomial discrete choice model, as explained in the Section \ref{intro}, is given by: 
$$S^{(mult)}_n(\beta) = \frac{1}{nm(m-1)}\sum_{i=1}^n \sum_{j=1}^m \left[Y_{i,j} \sum_{k \neq j} \mathds{1}(\mb{x}_{i,j}'\beta > \mb{x}_{i,k}'\beta)\right]$$ and the corresponding population version $\E S^{(mult)}_n(\beta)$ is given by:
 $$S^{(mult)}(\beta)= \frac{1}{m(m-1)}\sum_{j=1}^m \E\left[p_j(\mb{X})\left(\sum_{k \neq j} \mathds{1}((\mb{x}_j - \mb{x}_k)'\beta \ge 0)\right) \right]$$
where $p_j(\m{X}) = p(j|\mb{X},\beta_0) = \P(Y_{i,j} = 1| \mb{X}_i = \mb{X})$ and $m$ denotes the number of choices. 

In what follows, $m$ and $p$ should be viewed as growing as functions  of $n$. The following proposition establishes that 
$\beta^0$ is indeed the unique maximizer of the population score function: 

\begin{proposition}
\label{multi_pop_min}
Under Assumption \ref{ass:rank_mult} we have, $\beta^0 = \argmax_{\beta: \|\beta\|_2 = 1} S^{(mult)}(\beta)$, and that the maximizer is unique. 
\end{proposition}

Let $\hat \beta_n$ denote a maximizer of $S^{(mult)}_n(\beta)$. This is a slight abuse of notation as $\hat \beta_n$ also is used to indicate the ERM estimator for the binary counterpart of this model. In this section, $\hat \beta_n$ will unambiguously denote the ERM estimator for the multinomial choice model. We state a set of further assumptions (which should be viewed as natural extensions of our assumptions for the binary case) to facilitate the asymptotic analysis of $\hat \beta_n$.  

\begin{assumption}[Transition condition]
\label{ass:low_noise_mult}
The multinomial choice model satisfies the modified transition condition uniformly for all pairs $(j,k)$, i.e. there exists constants $C > 0$ (not depending on $n$) such that 
for every $(j,k)$: $$\P\left(\left|\frac{p_j(\mb{X})}{p_j(\mb{X}) + p_k(\mb{X})}- \frac{1}{2}\right| \le t\right) \le C\,t$$ for all $0 \le t \le t^*$. We assume $2t^*C > 1$ for mathematical simplicity. 
\end{assumption}

\begin{assumption}[Restricted wedge assumption]
\label{ass:wedge_mult} There exist constants $c_1, c_2, c_3 > 0$ and $R > 0$ such that: 
\begin{enumerate}
\item $p_j(\mb{X}) \ge c_1 > 0$ for all $j \in \{1,2 \dots, m\}$ for all $\|\mb{X}\|_F \le R$, where $\|\;\|_F$ denotes Frobenius norm. Here the constant $c_1$ depends $m$, while the radius of choice $R$ does not depend on the specific utility, but may or may not depend on $p$. 
 \\
\item For all pairs $(j,k)$, $\P\left(\left\{\s((\mb{x}_j - \mb{x}_k)'\beta) \neq \s((\mb{x}_j - \mb{x}_k)'\beta_0)\right\} \cap \{\|\mb{X}\|_F \le R \}\right) \ge c_2 \|\beta - \beta^0\|_2$ where $c_2 > 0$ does not depend on $n$. 
\\
\item The effect of radius $R$ is asymptotically non-vanishing, i.e. for all pairs $(j,k)$: $$\frac{\P\left(\left\{\s((\mb{x}_j - \mb{x}_k)'\beta) \neq \s((\mb{x}_j - \mb{x}_k)'\beta^0)\right\}\right)}{\P\left(\left\{\s((\mb{x}_j -\mb{x}_k)'\beta) \neq \s((\mb{x}_j - \mb{x}_k)'\beta^0)\right\} \cap \{\|\mb{X}\|_F \le R \}\right)} \le c_3 \,.$$
and the constant $c_3$ does not depend on $n$. 
\end{enumerate}
\end{assumption}
\noindent
%\begin{remark}
%Note that both $m$ and $p$ are considered as functions of $n$. Hence, when we say some constant is independent of $n$, it is automatically implied that the constant is also independent of $m, p$. In the interest of notational simplicity, we sometimes omit the subscript $n$ when the dependence on $n$ is clear from the context.
%\end{remark} 

\begin{remark}
Assumption \ref{ass:low_noise_mult} should be viewed as the multinomial version of Assumption \ref{ass:low_noise_binary}. It quantifies the probability mass of the covariate space where the magnitude of the difference between $p_j(\m{X})$ and $p_k(\m{X})$ is small relative to their sum, in terms of a generic threshold $t$. This is easily seen by noting that $|p_j(\m{X})/(p_j(\m{X}) + p_k(\m{X})) - 1/2| = \left|p_j(\m{X}) - p_k(\m{X})\right|/(p_j(\m{X}) + p_k(\m{X}))$. The smaller this quantity, the harder it is to differentiate between utilities $k$ and $j$. 
We note that for the multinomial problem we confine ourselves to a fixed 
$C$ (as opposed to a general sequence $C_n$) in our low-noise assumption which allows a cleaner and less cumbersome presentation of our results. As in the binary case, the fixed $C$ assumption is statistically the most interesting version. The proof for a general $C_n$ that goes to 0 would work similarly as for fixed $C$, except for the fact that we would now need to keep explicit track of the $C_n$ throughout the steps of the proof. 
\end{remark}

\begin{remark}
It is clear that Assumption \ref{ass:wedge_mult} is in similar vein to Assumption \ref{ass:wedge_binary} for the binary response model, albeit somewhat more involved owing to the multinomial structure. Part (1) of Assumption \ref{ass:wedge_mult} postulates a ball of radius $R$ around the origin in $\mathbb{R}^{m \times p}$ where the probability of choosing any specific utility given $\m{X}$ is bounded away from 0: i.e., every alternative can be chosen with non-negligible probability, or in other words, all utilities are competitive. %For binary response part we don't need that assumption as the second minimum probability is nothing but the maximum probability of two items which is always lower bounded by $1/2$. 
Part (2) of the assumption resembles Assumption \ref{ass:wedge_binary} exactly, modulo the fact that we are now interested in the wedge-shaped region within the ball of radius $R$. This is because part (1) of the Assumption restricts the main action to the ball of radius $R$ where the probability of choosing any item is non-negligible. Part (3) of the Assumption ensures that, the probability of the wedge-shaped region intersected with a ball of radius $R$ is not negligible with respect to the probability of the entire wedge-shaped region. In other words, the region of primary action is non-ignorable with respect to the entire region. This assumption helps us establish an upper bound on the variability of the empirical process relevant to our analysis of the concentration bounds for the estimator. 
\end{remark}

\begin{remark}
The results in this section presented below can also be derived by taking $R = \infty$ in Part (1) of Assumption \ref{ass:wedge_mult}. In this case, Part (1) becomes stronger as we now assume the lower bound on the conditional probabilities of choosing utilities for all $\m{X}$. On the other hand, Part (2) of the assumption is weakened: if the lower bound in Part (2) holds for finite $R$, it holds for $R = \infty$. Part (3) of the assumption is trivially satisfied for $R=\infty$ with $c_3 = 1$. 
\end{remark} 

\begin{remark} 
\label{rem:binary_mult_match}
The assumption that the conditional probability $p_j(\m{X})$ of each utility is bounded away from 0 (Part (1) of Assumption \ref{ass:wedge_mult}) can be easily relaxed. For example one may assume that $p_j(\m{X}) \vee p_k(\m{X}) \ge c_1$ for all $\|X\|_F \le R$ for all $1 \le j \neq k \le m$ without disturbing any of our calculations. Indeed, an inspection of the proof of Proposition \ref{multi-alpha-kappa} shows that what we crucially require to establish the curvature of $S^{(mult)}(\beta) - S^{(mult)}(\beta_0)$ is a lower bound on $p_j(\m{X}) + p_k(\m{X})$ for $\|X\|_F \le R$, and this is obviously true under the relaxed assumption. For the binary choice model $m=2$ this weaker assumption is automatic: $p_1(X) \vee p_2(X) \geq 1/2$ for all $X$, so that we can clearly take $R = \infty$ and Assumption \ref{ass:wedge_mult} boils down to Assumption \ref{ass:wedge_binary}.
\end{remark} 

The following Proposition (similar to Proposition \ref{alpha-kappa}) establishes a lower bound on the excess population risk: 
 
 \begin{proposition}
 \label{multi-alpha-kappa}
Under Assumptions \ref{ass:rank_mult}, \ref{ass:low_noise_mult} and \ref{ass:wedge_mult}, we have the following curvature condition for multinomial choice model:
$$S^{(mult)}(\beta^0) - S^{(mult)}(\beta) \ge \frac{c_1c^2_2}{4C} \|\beta - \beta^0\|_2^2 $$ 
for all $\beta \in S^{p-1}$ where $c_1, c_2, C, t^*$ are same constants as mentioned in Assumption \ref{ass:low_noise_mult} and \ref{ass:wedge_mult}. 
 \end{proposition}
The proof of the above proposition is conceptually similar to that of Proposition \ref{alpha-kappa}, as it relies on relating $S^{(mult)}(\beta_0) - S^{(mult)}(\beta)$ to the average of the probabilities of truncated wedge-shaped regions for all possible pairs of $m$ utilities. Note that for $m=2$, this corresponds to a single wedge-shaped region. 
%One noticeable difference from the binary response case is the use of truncated regions in the proof, which relates to Part (2) of Assumption \ref{ass:wedge_mult}). 
The average probability is then bounded using Part (2) of Assumption \ref{ass:wedge_mult} to conclude the proof. 
%The following theorem establishes rate of convergence of the estimator $\hat \beta_n$ under the regime $p/n \rightarrow 0$ along with finite sample concentration bound: 
\begin{theorem}[When $p/n \rightarrow 0$]
\label{multi_low_dim}
If $nc_1^2/(m^2p) \to \infty$, then under Assumptions \ref{ass:rank_mult}, \ref{ass:low_noise_mult} and \ref{ass:wedge_mult}, we have: 
$$\P\left(r_n\|\hat{\beta}_n - \beta^0\|_2 \ge Ky\right) \le e^{-y^2}$$ for all $y \ge 1$, where 
$$r_n = \left(\frac{nc^2_1}{m^2p}\right)^{1/3}\left(\log{\left(\frac{nc_1^2}{m^2p}\right)}\right)^{-1/3} $$
and $c_1$ is the same constant defined in Assumption \ref{ass:wedge_mult}, $K$ is some constant which does not depend on $n$. 
\end{theorem}

\begin{remark}
\label{rem:multi_low_dim}
%The rate derived above satisfies $r_n \asymp \left(\frac{nc_1^2}{m^2 p \log{n}}\right)^{1/3}$ under the condition $m^2 p = n^{\gamma}$ for some $\gamma < 1$. 
%Here also we mean \emph{approximately} in the same sense as in Theorem \ref{high-dim-rate}, which is explained in Remark \ref{rem:high-dim-rate}. 
It is instructive to relate this theorem with its counterpart for the binary choice model, Theorem \ref{rate-manski-p-less-n}. For the case $m=2$, it is clear from Remark \ref{rem:binary_mult_match} that we can always take $c_1 = 1/2$, and the rate of convergence becomes $\left(n/(p\log{(n/p)})\right)^{1/3}$ which is identical to that from Theorem \ref{rate-manski-p-less-n} when $C_n$ is fixed (See remark \ref{rem:low_dim_rate}). The additional term in the current rate, viz. $(c_1/m)^2$ can be viewed as an adjustment for the number of utilities. Notice that $c_1$ itself depends non-trivially on $m$: as 
$\sum_{j=1}^m p_j(\m{X}) = 1$, we need $c_1 \leq 1/m$ for Part 1 of Assumption \ref{ass:wedge_mult} to make sense. 
%For example, if we take say $c_1 = 1/(m \log{m})$, the rate becomes $$r_ n = \left(\frac{1}{\log{m}}\right)^{\frac{\alpha}{\alpha + 1}}\left(\frac{n}{p \log{n}}\right)^{\frac{\alpha}{\alpha + 2}}\left(\frac{1}{m}\right)^{\frac{\alpha(3\alpha + 4)}{(\alpha + 1)(\alpha + 2)}}$$ which for $\alpha = 1$ becomes: 
%$$r_ n = \left(\frac{1}{\sqrt{\log{m}}}\right)\left(\frac{n}{p \log{n}}\right)^{\frac{1}{3}}\left(\frac{1}{m}\right)^{\frac{7}{6}} \,.$$
\end{remark}

Next, we present our result for the fast growth regime, i.e. when $p \gg n$. As before, we require a sparsity assumption for the identification of the model. Our following assumption encodes the rate at which we can allow the sparsity to grow for our asymptotic analysis:

\begin{assumption}[Sparsity condition for Multinomial model]
\label{ass:sparsity_mult}
Under the fast growth regime, i.e. when $p \gg n$, we assume that there exists $s_0$ with $\|\beta^0\|_0 = s_0$ which satisfies: 
$$\frac{m^2 s_0\log{p}\log{n}}{nc_1^2} \longrightarrow 0$$ as $n \rightarrow \infty$. 
\end{assumption}

This assumption is identical in spirit to Assumption \ref{ass:sparsity_binary}. The only difference is that now both $m$ and $c_1$ play a role in determining the permissible rate of sparsity of the true vector $\beta^0$. As mentioned before, the factor $m^2$ appears due to pairwise comparison of utilities and the factor $c_1$ relates to the curvature condition established in Proposition \ref{multi-alpha-kappa}. 
\begin{theorem}[When $p \gg n$]
\label{multi_high_dim}
Under Assumptions \ref{ass:rank_mult}, \ref{ass:low_noise_mult} and \ref{ass:wedge_mult}, there exists a constant $\Sigma > 0, K_1 > 1, K_2 > 0$ (not depending on $n$) such that for all $y \ge 0$: 
\begin{equation*}
   \P\left(r_n \|\hat \beta_{\hat m} - \beta^0\|_2 \ge K_1 + K_2s_ny \right) \le \Sigma e^{-y^2}\end{equation*}
where,
$$r_n =   \left(\frac{nc_1^2}{m^2s_0\log{(ep/s_0)}}\right)^{1/3}\log{\left(\frac{nc_1^2}{m^2s_0\log{(ep/s_0)}}\right)}^{-1/3} \,.$$
and $s_n \to 0$ as $n \to \infty$.  
% 
%Furthermore, if $p = \exp(n^{a})$, $s_0 = n^{b}$ and $m = n^{c}$ for constants $a,b,c > 0$ satisfying $a + b + 2c < 1$ and Assumption 
%\ref{ass:sparsity_mult} is satisfied, then $$r_n \asymp c_1^{\frac{\alpha}{\alpha + 1}}\left(\frac{n}{m^2 s_0\log{p}\log{n}}\right)^{\frac{\alpha}{\alpha + 2}}$$ with $c_1$ being the same constant as defined in Assumption \ref{ass:wedge_mult}. 
\end{theorem}

\begin{remark}
\label{rem:multi_high_dim}
Recall from Remark \ref{rem:multi_low_dim}, when $m=2$, one can always take $c_1 = 1/2$. The rate of convergence then becomes $\left(n/(s_0\log{(ep/s_0)}\log{(n/s_0\log{(ep/s_0)})})\right)^{1/3}$ (similar to the rate obtained in Theorem \ref{high-dim-rate} when $C_n$ is fixed), which can be further simplified to $(n/\left(s_0\log{p}\log{n}\right))$ under the specific choices of $s_0, p$ taken in Remark \ref{rem:high-dim-rate}. As in the moderate growth regime, the additional term $(c_1/m)^2$ in $r_n$ above is the adjustment for the growing number of utilities. 
\end{remark}

\section{Computational Aspects}\label{computation}
In this section we investigate the performance of a number of procedures employ for estimating $\beta_0$ in the binary choice model and compare their performances. Specifically, we consider the following three methods: 
\begin{enumerate}
\item Logistic regression.
\item Support Vector Machine. 
\item Homotopy path-following framework adapted in \cite{feng2019nonregular}. 
\end{enumerate} 

We divide our simulation studies into two sections: the slowly growing regime, i.e. $p/n \rightarrow 0$ and the fast growing regime i.e. $p \gg n$. The algorithm based on homotopy path-following framework described in \cite{feng2019nonregular} is tailored to the scenario $p \gg n$, hence we only compare SVM and logistic regression under the slowly growing regime, and all three methods (with the $\ell_1$ penalized versions of SVM (see \cite{zhu20041}) and logistic regression) when $p \gg n$. Our primary data generation mechanism is common to both regimes (with slight changes in the $p \gg n$ to accommodate sparsity considerations) and is described below:

\begin{enumerate}
\item {\bf Generation of the true $\beta^0$}: For the regime $p/n \rightarrow 0$, each entry of $\beta^0$ is generated from the $\text{Unif}(1,2)$ distribution and then normalized to make its $\ell_2$ norm 1. For the regime $p \gg n$, each of the $s_0$ active entries is generated randomly from $\text{Unif} (2,3)$ and then normalized to keep $\|\beta_0\|_2 = 1$. This $\beta_0$ remains fixed over all monte-carlo iterations. 
\vspace{0.1in}
 \item Generate $X_1, \dots, X_n$ from $\mathcal{N}(0,\Sigma)$ where the dispersion matrix $\Sigma$ has the following form: 
 $$
\Sigma_{i,j}=
\begin{cases}
1, & \text{if } i = j\\
\rho^{|i-j|} & \text{otherwise} \,,
\end{cases}
$$
where we take $\rho = 0.5$ for our simulations. 
\item We generate the co-variate dependent errors $\epsilon_i's$ as follows: for $1 \le i \le n$, $$\epsilon_i|X_i \sim \mathcal{N}(0, \sigma^2_i)$$ where $\sigma^2_i = 1 \vee |X_i^{\top} \beta_0|$.
\vspace{0.1in}
\item Finally we set $Y_i = \s(X_i^{\top}\beta^0 + \epsilon_i)$ for all $1 \le i \le n$. (or one can set $Y_i = \mathds{1}(X_i^{\top}\beta^0 + \epsilon_i \ge 0)$ depending on how one wants to encode the binary variable). 
\end{enumerate}
The idea behind this model is that, the data close to the boundary are more informative than the data far away. Note also, that when the variability of the error near the boundary is low, estimation is a relatively easy task. Hence, to challenge the existing methods, we assume non-negligible error variance near the boundary. For the simulation setting above, for a point $X$ near the boundary, i.e. $X^{\top}\beta^0 \approx 0$, the (conditional) variance of the error is $\approx 1$, and as one moves to points away from the boundary, the (conditional) variability of the response increases depending on their distance from the true hyperplane. %Based on the above mechanism, we generate data $\{X_i,Y_i\}_{i = 1, \dots, n}$ and fit our models to compare their performances.

\subsection{Estimation error \textbf{$p = o(n)$}}
We explore three different growth patterns of $p$ relative to $n$: 
%\begin{enumerate}
%\item  $p = \lfloor n^{1/4} \rfloor$ \item $p = \lfloor n^{1/2} \rfloor$
%\item $p = \lfloor n^{3/4} \rfloor$
%\end{enumerate}
\begin{align*}
1. \ p = \lfloor n^{1/4} \rfloor \hspace{1in}  2.\  p = \lfloor n^{1/2} \rfloor   \hspace{1in} 3. \  p = \lfloor n^{3/4} \rfloor
\end{align*}
\noindent
where $\lfloor x \rfloor$ is the floor function. The sample size $n$ ranges as: $n = 12000, 14000, 16000, 18000, 20000$. 
\newline
Consider first the performance of SVM on the generated data. Below are three density plots of scaled estimation error $(n/p)^{1/3}\|\hat \beta - \beta_0\|_2$ based on 500 monte-carlo iterations: 

\begin{figure}[H]
\centering
\begin{subfigure}{.5\textwidth}
  \centering
  \includegraphics[width=1\linewidth]{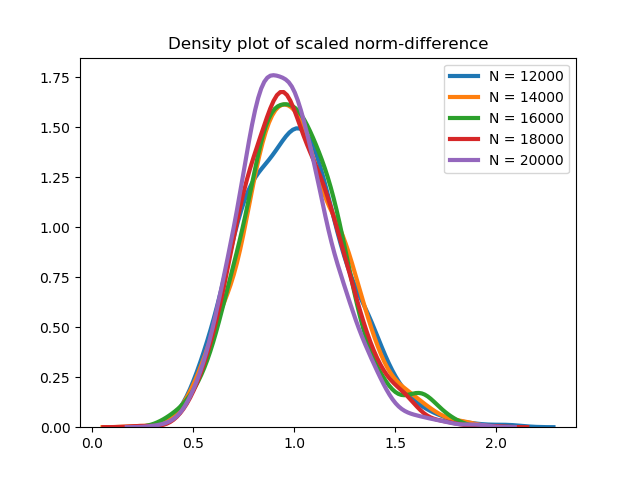}
 \caption{Density plot using SVM when $p = \lfloor n^{1/4} \rfloor$}
  \label{fig:sub1}
\end{subfigure}%
\begin{subfigure}{.5\textwidth}
  \centering
  \includegraphics[width=1\linewidth]{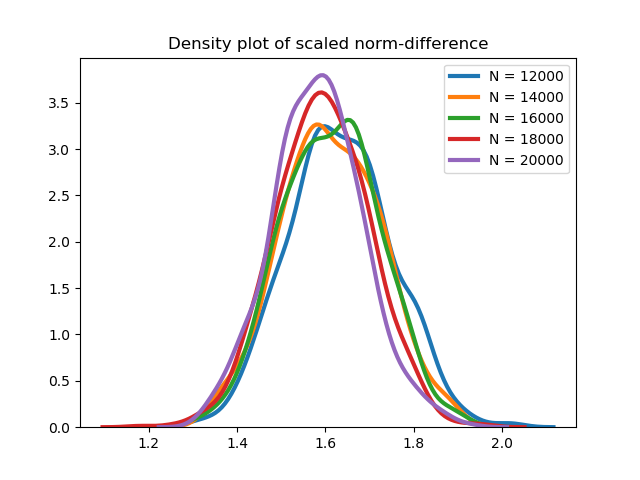}
  \caption{Density plot using SVM when $p = \lfloor n^{1/2} \rfloor$}
  \label{fig:sub2}
\end{subfigure}
\begin{subfigure}{.5\textwidth}
  \centering
  \includegraphics[width=1\linewidth]{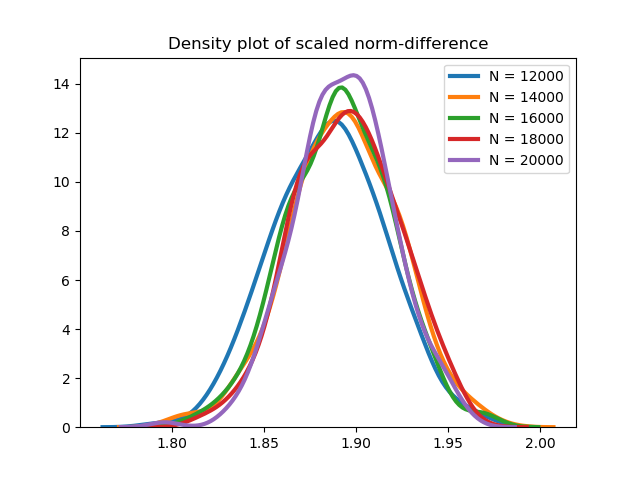}
  \caption{Density plot using SVM when $p = \lfloor n^{3/4} \rfloor$}
  \label{fig:sub2}
\end{subfigure}
\caption{Density plots of scaled estimation error $(n/p)^{1/3}\|\hat \beta - \beta^0\|_2$ using SVM}
\label{fig:svm}
\end{figure}

As is evident from the density plots above, the distribution of the normalized errors is quite stable across the different values of $n$ suggesting that the SVM method is giving a quite decent approximation to the actual score estimator. 
%It seems from the plots that the SVM estimator converges to the truth at optimal rate $(n/p)^{1/3}$ at-least in this simulation model. 
We next apply simple logistic regression for estimating $\beta^0$. As we except, this does not perform as well as SVM owing to model mis-specification. The reason we study logistic regression is because it is typically the bread and butter option for dealing with binary response regression, but as we see below is quite suspect in this situation. Below are the plots from logistic regression:

\begin{figure}[H]
\centering
\begin{subfigure}{.5\textwidth}
  \centering
  \includegraphics[width=1\linewidth]{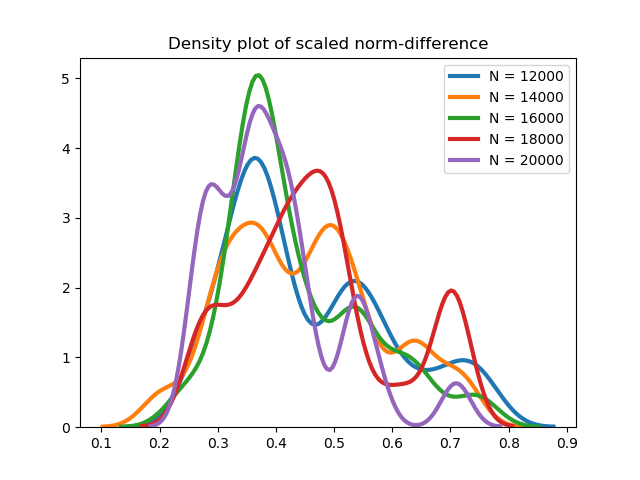}
 \caption{Density plot using LR when $p = \lfloor n^{1/4} \rfloor$}
  \label{fig:sub1}
\end{subfigure}%
\begin{subfigure}{.5\textwidth}
  \centering
  \includegraphics[width=1\linewidth]{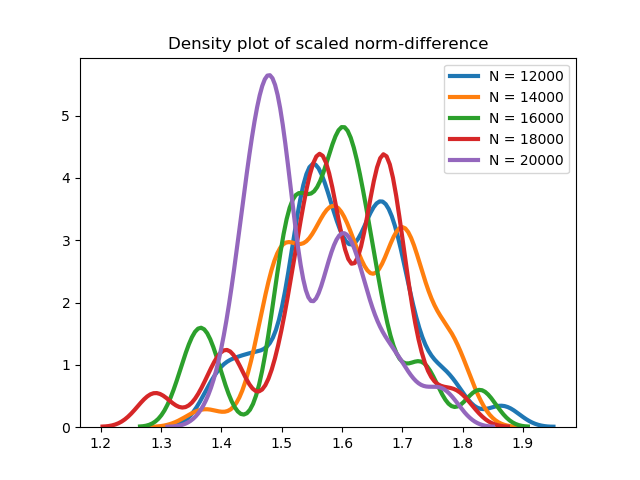}
  \caption{Density plot using LR when $p = \lfloor n^{1/2} \rfloor$}
  \label{fig:sub2}
\end{subfigure}
\begin{subfigure}{.5\textwidth}
  \centering
  \includegraphics[width=1\linewidth]{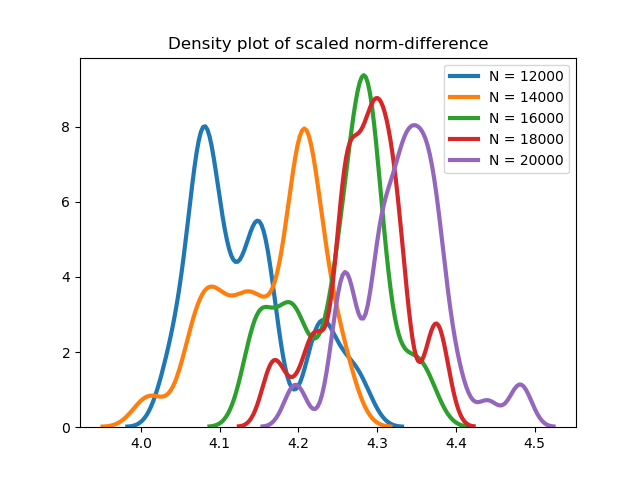}
  \caption{Density plot using LR when $p = \lfloor n^{3/4} \rfloor$}
  \label{fig:sub2}
\end{subfigure}
\caption{Density plots of scaled estimation error $(n/p)^{1/3}\|\hat \beta - \beta^0\|_2$ using Logistic Regression}
\label{fig:svm}
\end{figure}

It is quite clear from the plots that the scaled error is not converging with $n$ and behaves in a rather erratic manner, and the SVM algorithm is markedly superior.  Further investigation into SVM based algorithms in this and similar models can constitute a potentially interesting topic for future research.

\subsection{Model selection and estimation when $p \gg n$}
As mentioned in the Introduction, our problem can be viewed as a binary classification problem with a linear Bayes' classifier. Under the sparsity assumption, only a few covariates contribute to the classification. To identify these covariates, we  resort to a penalized classification approach. We here employ three methods: 
\begin{enumerate}
\item $\ell_1$ penalized SVM.
\item $\ell_1$ penalized logistic regression. 
\item Homotopy path-following framework adapted in \cite{feng2019nonregular}. 
\end{enumerate}
Recall that the data generating mechanism has already been described. We take $n= 2000, p = 10000$ for our simulations and use five different values of $s_0 = 10, 20, 30, 40, 50$. Our goal is to investigate how the performance of the classifier changes as we increase $s_0$ keeping $n$ and $p$ fixed. The penalty parameters for logistic regression and SVM are selected using grid search and two-fold cross-validation. We also implement the algorithm based on homotopy path-following framework adapted in \cite{feng2019nonregular}. We assess the performances of these three approaches based on the following discrepancy measures: 
\begin{enumerate}
\item Misclassification error. 
\item Norm difference between $\hat \beta$ and $\beta^0$. 
\item No. of true active variables not selected (denoted by Type 2 error).
\item No. of true null variables selected (denoted by Type 1 error).
\end{enumerate}
The following plots provide a visual representation of the comparisons for different values of sparsity and across the three methods. 
%Ideally, we would like each of the parameters to be low by our classifier, but the real world is never so kind.  
\begin{figure}[H]
\centering
\begin{subfigure}{.48\textwidth}
  \centering
  \includegraphics[width=1\linewidth]{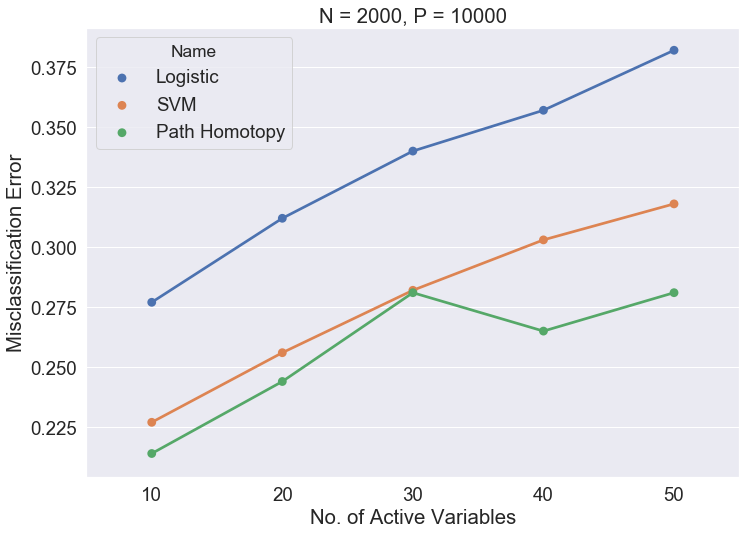}
 \caption{Misclassification error}
  \label{fig:sub1}
\end{subfigure}%
\begin{subfigure}{.48\textwidth}
  \centering
  \includegraphics[width=1\linewidth]{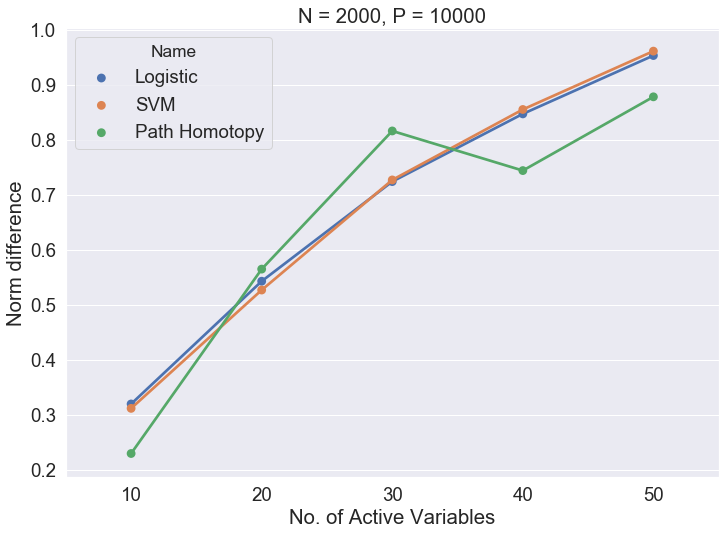}
  \caption{Estimation error using $\ell_2$ norm}
  \label{fig:sub2}
\end{subfigure}
\begin{subfigure}{.48\textwidth}
  \centering
  \includegraphics[width=1\linewidth]{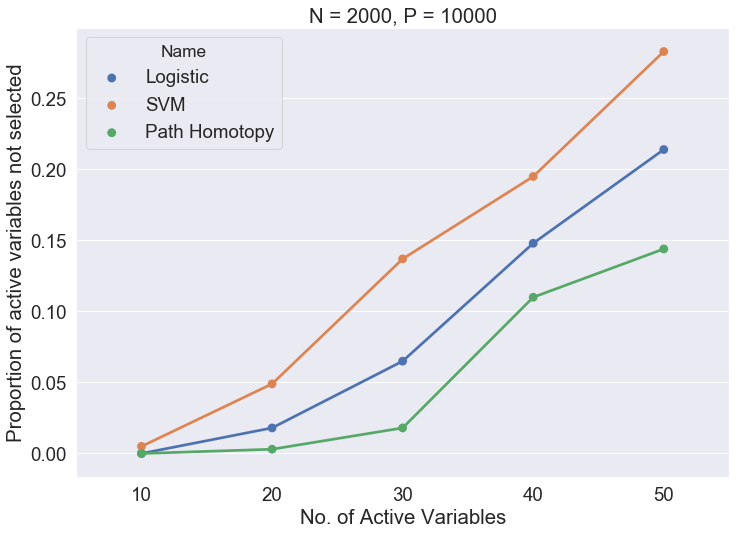}
  \caption{Type 2 error}
  \label{fig:sub2}
\end{subfigure}
\begin{subfigure}{.48\textwidth}
  \centering
  \includegraphics[width=1\linewidth]{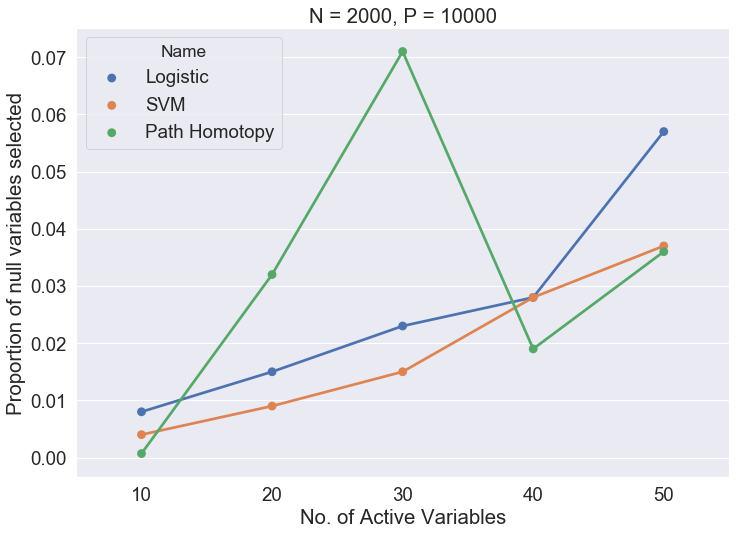}
  \caption{Type 1 error}
  \label{fig:sub2}
\end{subfigure}
\caption{Comparisons among the three approaches}
\label{fig:svm}
\end{figure}

It is clear that logistic regression is generally outperfomed by some other method in this model. The performance of SVM and the algorithm proposed in \cite{feng2019nonregular} are generally at par, though from eye-inspection, the later seems superior. For example, in terms of  mis-classification error and norm-difference, their performance is similar; in some cases algorithm of \cite{feng2019nonregular} performs better than SVM, while in other cases SVM wins marginally. Type 2 error (proportion of true active variables missed by the method) is generally lower for the algorithm in \cite{feng2019nonregular} when compared to SVM, whilst Type 1 error (proportion of null variables declared active) is generally higher: SVM is more conservative in terms of selecting variables. 

Depending on one's priorities, one may weigh Type 1 and Type 2 errors differently to generate a weighted misclassification error. In the absence of any such information, it is natural to assign equal weights, which leads to the sum of thes two errors, as shown in the following plot: 

\begin{figure}[H]
\centering
\includegraphics[scale = 0.5]{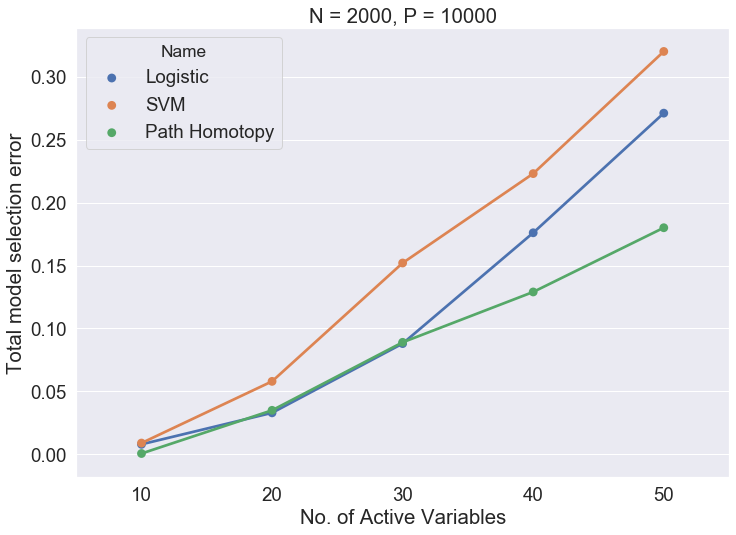}
\end{figure}

We find that the algorithm proposed in \cite{feng2019nonregular} better under this metric especially for large $s_0$ which is explained by tendency of SVM not selecting enough active variables. 

Based on our study, it appears that one is better off with the algorithm proposed in \cite{feng2019nonregular} in the $p \gg n$ scenario, though, of course much larger scale simulations would be necessary to make any general recommendations. As the focus of our paper is largely theoretical, we do not develop these studies any further but note that a thorough investigation of computationally feasible methods in this and related problems involving optimization of discontinuous functions along with analytical assessments of their performance constitutes an open direction of research. 
\section{Concluding Discussion}
\label{discussion}
We close with a discussion of various aspects of the high-dimensional binary choice model and our approach to the problem. 

\subsection{Exploring and relaxing our assumptions:}
\label{Assumption A2}
It is of interest to investigate sufficient conditions under which Assumption \ref{ass:wedge_binary}  and Assumption \ref{ass:A2:upper} hold. We show in Lemma \ref{A2lemma} in the supplement that these two assumptions hold simultaneously when $X$ arises from an elliptically symmetric distribution centered at 0, under some restrictions on the minimum and maximum eigenvalues of its orientation matrix. Assumption \ref{ass:A2:upper} also holds for elliptically symmetric distributions centered at 0 but under some further mild conditions, as demonstrated in Lemma \ref{Lipschitz-lemma}.

\subsection{Model with intercept: }
Our treatment thus far has considered a model of the form $Y^{\star} = X^{\top} \beta_0 + \epsilon$ for a \emph{random} $X$. However, many practical scenarios necessitate the inclusion of an intercept term where the term $X^{\top} \beta_0$ is replaced by $(1, X^{\top}) (\tau_0, (\beta^0)^{\top})^{\top}$. Assumption \ref{ass:wedge_binary} then naturally generalizes to
$$P(\s(\tau^0 + X^{\top} \beta^0) \ne s(\tau + X^{\top} \beta)) \geq c' \sqrt{(\tau - \tau^0)^2 + \|(\beta - \beta^0)\|_2^2} \,.$$
However, we cannot expect this to be satisfied for all $(\tau, \beta)$ with $\|\beta\| = 1$ when $\tau$ varies in an unconstrained manner. Consider for example the case that $X \sim N(0, I_p))$ so that $X^{\top} \beta$ and $X^{\top} \beta^0$ are both standard normal. In this case, when $\tau$ and $\tau^0$ are very large, the signs of $\tau^0 + X^{\top} \beta^0$ and $\tau + X^{\top} \beta$ are primarily driven by the magnitudes of $\tau$ and $\tau^0$, so that if these two parameters have sign, the probability of the signs being different can be made as small as one pleases depending on the magnitudes of the $\tau$'s. This entails controlling the magnitudes of the $\tau$'s relative to the $\beta$'s; in particular, if the absolute magnitudes of the $\tau$'s are kept bounded away from $\infty$, a restricted version of Assumption \ref{ass:wedge_binary}, in the sense that the inequality in Assumption \ref{ass:wedge_binary} is fulfilled for all $\beta$ sufficiently close to $\beta_0$, is, indeed, verifiable for certain families of distributions including elliptically symmetric $X$ centered at the origin, as well as $X$'s with independent components where each component has a symmetric log-concave density with mode at 0. The $\ell_2$ convergence and minimax lower bound results established in this paper  still continue to hold, but to accommodate the restricted version of this assumption, the proofs presented in the supplement need to be slightly modified. An elaborate and rigorous discussion of such models with intercept is available in section \ref{Discussion of intercept} of the supplement. 

\begin{comment} 
Recall that (A2) says that $c'\|\beta - \beta^0\|_2 \le P(\s(X^{\top}\beta) \neq \s(X^{\top}\beta^0))$ for all $\beta \in S^{p-1}$ for some constant $c'> 0$. If $(A2)$ is assumed only locally, in the sense that the inequality in the preceding sentence holds only on the set $\{\beta: \|\beta - \beta^0\|_2 \le \delta\}$, for some fixed $\delta > 0$, Proposition \ref{alpha-kappa} will be true only on this set. Consequently, proving the rate of convergence of Manski's maximum score estimator will require establishing consistency of $\hat{\beta}$ for $\beta^0$ in the $\ell_2$ norm, in order to guarantee that $\hat{\beta}$ will eventually be in a $\delta$ neighborhood of $\beta^0$. 
\begin{theorem}
\label{consistency}
Suppose that Assumption A1 holds and that the inequality in Assumption A2 is only satisfied for $\{\beta: \|\beta - \beta^0\|_2 \le \delta\}$. Then, in the regime that $p/n \rightarrow 0$, $\|\hat{\beta} - \beta^0\|_2 \overset{P} \rightarrow 0$ as $n \rightarrow \infty$. The same conclusion holds in the regime $p > n$ under the assumption that $s\log{p}/n \rightarrow 0$. 
\end{theorem}
Under this localized version of (A2), the minimax upper bound results for $\hat{\beta}$ as presented in Theorems \ref{rate-manski-p-less-n} and \ref{high-dim-rate} are no longer guaranteed (this is easily seen from an inspection of the proof of the minimax upper bound), and whether the grid estimator is $(p/n)^{1/3}$ consistent in the slow regime is also unclear. However, all other results presented in the paper do go through.
\end{comment}  

\subsection{Asymptotic distribution} In their seminal paper, Kim and Pollard \cite{kp90} proved that for fixed $p$, $n^{1/3}(\hat{\beta} - \beta^0)$ converges in distribution to the maximizer of a Gaussian process with quadratic drift. Our treatment of the binary choice model should be contrasted with their approach: while they assumed the continuous differentiability of both the density of $X$ and $\eta(x) = P(Y=1|X=x)$ and a compact support for $X$, we have made no such assumptions. We have tackled those aspects of this problem from the classification point of view, with assumptions on the growth of $P(Y=1|X=x)$ near the Bayes hyperplane and in addition, conditions on the distribution of $X$ to ensure that sufficiently many observations are available around the Bayes hyperplane. 
As far as the asymptotic distribution of the score estimator in growing dimensions (or functionals thereof) is concerned, this is, in itself, a mathematically formidable problem, well outside the scope of this paper. Based on what we know in the fixed $p$ setting, the forms of such distributions are likely to be extremely complicated. 
% One can try to find the asymptotic distribution of $\ell_2$ difference i.e. $a_n \|\hat{\beta} - \beta^0\|_2$ with suitable scaling as both $n,p \rightarrow \infty$ which seems to be theoretically formidable and is an open problem.  A
The question remains whether tractable asymptotic distributions for making inference on components of $\beta^0$ in the growing $p$ setting could be obtained for smoothed versions of the score estimator, in the spirit of Horowitz's paper \cite{horowitz1992}. This is likely to be an interesting but challenging avenue for future research on this subject.

\section{Selected Proofs}
\label{proofs}
\subsection{Proof of Theorem \ref{multistage-estimator}}
We generate $a_n(8n/p)^{\frac{p-1}{3}}$ points uniformly from the surface of the sphere (where $a_n \uparrow \infty$ will be chosen later), maximize the empirical score function $S_n(\beta)$ over these selected points and show that the maximizer achieves the desired rate. Define $T_n = a_n(8n/p)^{(p-1)/3}$ and $E_n$ to be the collection of $T_n$ points generated uniformly. \\
We start with the following technical lemma that plays a key role in the proof. 
\begin{lemma}
\label{distance}
Suppose $D(x,r)$ denotes a spherical cap around $x$ of radius $r$, i.e. $$D(x,r) = \{y \in S^{p-1}: \|x-y\|_2 \le r\}$$ Then we have $$\frac{1}{2}(r/2)^{p-1} \le \sigma(D(x,r)) \le \frac{1}{2\sqrt{2}}r^{p-1} $$ for $0 \le r \le 1$ and $p \ge 8$, where $\sigma$ is the uniform measure on the sphere, i.e. the proportion of the surface of the spherical cap to the surface area of the sphere.
\end{lemma}

\begin{comment}
\begin{lemma}
\label{angle}
For any fixed $x \in S^{p-1}$, define $C(x,\epsilon)$ to be $\epsilon$-angular spherical cap around $x$, i.e. $$C(x, \epsilon) = \{y \in S^{p-1}: \langle x,y \rangle \ge \epsilon\}$$ Then we have $$\sigma(C(x, \epsilon)) \le  \frac{1}{2\epsilon\sqrt{p}}(1-\epsilon^2)^{\frac{p-1}{2}} \le \frac{1}{2\sqrt{2}}(1-\epsilon^2)^{\frac{p-1}{2}}$$ for $\sqrt{\frac{2}{p}} \le \epsilon \le 1$. The last inequality follows from the assumption $\sqrt{\frac{2}{p}} \le \epsilon$.
\end{lemma}
\vspace*{0.2in}

\noindent
One should note the above two Lemmas are in different scale, one of them involves the angle and the other one involves distance. Here we will change Lemma \ref{angle} a little bit to convert in into a Lemma involving distance:
\begin{lemma}
\label{angle-distance}
For $0 \le r \le 1$ and $p \ge 8$, we have: $$\sigma(D(x, r)) \le \frac{1}{2\sqrt{2}}r^{p-1}$$
\end{lemma}
\begin{proof}
The proof of this lemma is quite straight forward. Note that $C(x,\epsilon) = D(x,r)$ where $\epsilon = (1-r^2/2)$. If $r \le 1$ and $p \ge 8$ then $\epsilon \ge \sqrt{\frac{2}{p}}$. Hence we have:
\allowdisplaybreaks
\begin{align*}
\sigma( D(x,r)) & \le \frac{1}{2\sqrt{2}}\left(1-\left(1-\frac{r^2}{2}\right)^2\right)^{\frac{p-1}{2}} \\
& \le \frac{1}{2\sqrt{2}}r^{p-1}\left(1 - \frac{r^2}{4}\right)^{\frac{p-1}{2}} \\
& \le \frac{1}{2\sqrt{2}}r^{p-1}
\end{align*}
which completes the proof.
\end{proof}
\end{comment}
\noindent
For a brief discussion on this Lemma, see section \ref{Convex geometry}. The next lemma shows that we can find at least one point in our collection which is within a distance of $(p/n)^{1/3}$  of $\beta_0$ with probability $\uparrow 1$.  
\begin{lemma}
\label{cover}
Let $\Omega_{-1,n}$ denote the event that there exists at least one $\beta' \in E_n$ such that $\|\beta' - \beta^0\|_2 \le (p/n)^{1/3}$. Then $P(\Omega_{-1,n}) 
\rightarrow 1$. 
\end{lemma}
\begin{proof}
Using Lemma \ref{distance} we have the following bound:
\allowdisplaybreaks
\begin{align*}
\P(\exists \beta' \in E_n \ \text{such that} \,\,\, \|\beta' - \beta^0\|_2 \le (p/n)^{1/3}) & \ge 1 - \left(1- \frac{1}{2}(p/8n)^{\frac{p-1}{3}}\right)^{a_n(8n/p)^{\frac{p-1}{3}}} \\
& \rightarrow 1 \,\,\, \text{as} \,\,\, n \rightarrow \infty \,.
\end{align*}
\end{proof}
\noindent
Let $\tilde{\beta}$ denote the point closest to $\beta^0$. On $\Omega_{-1,n}$, $\|\tilde{\beta} - \beta^0\| \leq (p/n)^{1/3}$. 
To establish the convergence rate, we will use a specific version of the shelling argument. Fix $T > 0$, sufficiently large. (In fact, as we work our way through the proof we will keep enhancing the value of $T$ as and when necessary, but as this will be done finitely many times, it won't have a bearing on the rate of convergence.) Consider shells $C_i$ around the true parameter $\beta^0$, where $$C_i = \{\beta \in S^{p-1}: \|\beta - \beta^0\|_2 \le T(p/n)^{1/3}2^i\} = D(\beta^0, r_i)$$ with $r_i = T(p/n)^{1/3}2^i$, for $i = 0,1, .. A_n$ and $A_n \overset{\Delta} = \frac{1}{3}\log_2{(n/p)} - \log_2{T}$. We will compute an upper bound on the number of elements of $B_i = E_n \cap C_i$ for all $i \in \{0,1, \cdots, A_n\}$.
\begin{lemma}
\label{upper-bound}
For all $i \in \{0,1, \cdots, A_n\}$, $$|B_i| = |E_n \cap C_i| \le 2T_np_i \le \frac{1}{2\sqrt{2}}a_n(T2^{i+1})^{p-1} \leq a_n (T 2^{i+1})^p$$ with exponentially high probability where $p_i = \sigma(D(\beta^0, r_i))$.
\end{lemma}
\begin{proof}
Let $N_i$ denote the number of points in $E_n \cap B_i$. Then $N_i \sim \text{Bin}(T_n , p_i)$ where $p_i = \sigma(D(\beta^0, r_i))$. For $i = A_n$, $p_i = 1$. So $P(N_i > 2T_np_i) = 0$. Hence we will only confine ourselves to the case $i \in \{0,1, \cdots, A_n -1\}$. In this case, $r_i \le 1$ and hence from Lemma \ref{distance} we have $p_i \le \frac{1}{2\sqrt{2}} < \frac{1}{2}$. From the Chernoff tail bound for the Binomial 
distribution we have, for each $i$: $P(N_i > 2T_np_i) \le \exp(-T_nD(2p_i || p_i))$, where 
$$D(2p_i || p_i) = 2p_i\log{\frac{2p_i}{p_i}} + (1-2p_i)\log{\frac{1-2p_i}{1-p_i}} = p_i\log{4} + (1-2p_i)\log{\frac{1-2p_i}{1-p_i}}$$ 
is the Kullback-Liebler divergence between Bernoulli$(p_i)$ and Bernoulli$(2p_i)$. This can be lower bounded thus:
\allowdisplaybreaks
\begin{align*}
D(2p_i || p_i) & = p_i\log{4} + (1-2p_i)\log{\frac{1-2p_i}{1-p_i}} \\
& = p_i\log{4} + (1-2p_i)\log{\frac{1-p_i - p_i}{1-p_i}} \\
& = p_i\log{4} + (1-2p_i)\log{\left(1-\frac{p_i}{1-p_i}\right)} \\
& \ge p_i\log{4} - (1-2p_i)\frac{\frac{p_i}{1-p_i}}{1-\frac{p_i}{1-p_i}} \hspace*{0.2in} [\because \log{1-x} \ge \frac{-x}{1-x}, \ p_i \le \frac{1}{2}] \\
& = p_i\left(\log{4} - 1\right) \\
& \ge \frac{1}{2}(T2^{i})^{p-1}\left(\frac{p}{8n}\right)^{p-1}(\log{4} - 1) \,.\hspace*{0.2in} [\text{Lemma} \ \ref{distance}]
\end{align*}
Using this upper bound we have:
\allowdisplaybreaks
\begin{align*}
P(N_i > 2T_np_i) & \le e^{\left(-T_n \frac{1}{2}(T2^{i})^{p-1}\left(\frac{p}{8n}\right)^{p-1}(\log{4} - 1)\right)} \\
& = e^{\left(- a_n(8n/p)^{p-1}\frac{1}{2}(T2^{i})^{p-1}\left(p/8n\right)^{p-1}(\log{4} - 1)\right)} \\
& = e^{\left(- a_n\frac{1}{2}(T2^{i})^{p-1}(\log{4} - 1)\right)}
\end{align*}
\end{proof}
\noindent
Define $\Omega_{i,n} = {N_i \le 2T_np_i}$ for $i = \{0,1,\cdots, A_n-1\}$ and let $\Omega_{n} = \cap_{i=-1}^{A_n-1}\Omega_{i,n}$. The following lemma says that the event $\Omega_n$ happens with high probability:
\begin{lemma}
For any $T>1$, $\P(\Omega_n) \rightarrow 1$ as $n \rightarrow \infty$.
\end{lemma}
\begin{proof}
It is enough to show that $\sum_{i=-1}^{A_n -1}\P(\Omega_{i,n}^c) \rightarrow 0$ as $n \rightarrow \infty$. We have already established in Lemma \ref{cover} that $\P(\Omega_{-1,n}^c) \rightarrow 0$ as $n \rightarrow \infty$. Using Lemma \ref{upper-bound}:
\allowdisplaybreaks
\begin{align*}
\sum_{i=0}^{A_n -1}\P(\Omega_{i,n}^c) & \le \sum_{i=0}^{A_n -1} e^{\left(- a_n\frac{1}{2}(T2^{i})^{p-1}(\log{4} - 1)\right)} \hspace{0.2in} \,.
%[\text{From Lemma} \ \ref{upper-bound}]
\end{align*}
Now for any fixed $n$, the maximum term obtains when $i=0$ i.e. $e^{\left(- a_n\frac{1}{2}(T)^{p-1}(\log{4} - 1)\right)}$ which goes to 0 for $T > 1$. 
Furthermore, the series under consideration is easily dominated by $\sum_{i=1}^{\infty}\,e^{-k 2^{i}}$ for some constant $k > 0$, which is clearly finite. 
Hence the series on the right-side of the above display goes to 0 with increasing $n$. 
\end{proof}

\noindent
The rest of the analysis will be done conditioning on the event $\Omega_n$. Define $\P_n(A) = \P(A \ \vert \ \Omega_n)$. Then we have:
\allowdisplaybreaks
\begin{align*}
& \hspace{-0.2in}\P\left(\|\hat{\beta}_{new} - \beta^0\|_2 > 3T(p/n)^{1/3} \right) \\
& \le \P\left(\|\hat{\beta}_{new} - \beta^0\|_2 > 3T(p/n)^{1/3} \ \vert \ \Omega_n\right)P(\Omega_n) + \P(\Omega_n^c) \\
& \le \P_n \left(\|\hat{\beta}_{new} - \beta^0\|_2 > 3T(p/n)^{1/3}\right) + \P(\Omega_n^c) \\
& \le \P_n \left(\|\hat{\beta}_{new} - \tilde{\beta}\|_2 > 2T(p/n)^{1/3}\right) + \P(\Omega_n^c) \\
& \le \P_n\left(\sup_{\beta \in B_0^c} S_n(\beta) - S_n(\tilde{\beta}) \ge 0\right) + \P(\Omega_n^c)  \\
& = \P_n\left(\cup_{i=1}^{A_n} \left\{\sup_{\beta \in B_i \cap B_{i-1}^c} S_n(\beta) - S_n(\tilde{\beta}) \ge 0\right\} \right) + \P(\Omega_n^c) \\
& \le \sum_{i=1}^{A_n} \P_n \left(\sup_{\beta \in B_i \cap B_{i-1}^c} S_n(\beta) - S_n(\tilde{\beta}) \ge 0\right) + \P(\Omega_n^c) \,.
\end{align*}
Since $\P(\Omega_n^c) \rightarrow 0$ as $n \rightarrow \infty$, we omit this term henceforth. Next, we analyze a general summand. Define $Z_i(\beta) = Y_i\s(X_i^{\top}\beta) - Y_i\s(X_i^{\top}\tilde{\beta})$. Then  $\frac{1}{n}\sum_{i=1}^n Z_i(\beta) = S_n(\beta) - S_n(\tilde{\beta})$ and $Z_i(\beta)$ assumes values $\{-2,0,2\}$. Also, $\E(Z_i(\beta)) = S(\beta) - S(\tilde{\beta})$. Using Proposition \ref{alpha-kappa} and Assumption \ref{ass:A2:upper} we have: 
\allowdisplaybreaks
\begin{align*}
S(\beta) - S(\tilde{\beta}) & =  S(\beta) - S(\beta^0) + S(\beta^0) - S(\tilde{\beta}) \\
& \le -u_- \|\beta - \beta^0\|^2_2 + u^+\|\tilde{\beta} - \beta^0\|^2_2 \\
& \le - (u_-/2) \|\beta - \beta^0\|^2_2
\end{align*}
for $T> \sqrt{(2u_+)/u_-}$. This implies $Z_i(\beta)$ has high probability of being negative. We exploit this to prove the concentration. To simplify the calculations,  define $\{Y_i(\beta)\}_{i=1}^n$ be to be a collection of independent random variables with $$Y_i(\beta) =
\begin{cases}
 2, & \text{with prob.}\, P(Z_i(\beta) = 2\ \vert \ Z_i(\beta) \neq 0)\\
 -2, & \text{with prob.}\, P(Z_i(\beta) = -2\ \vert \ Z_i(\beta) \neq 0)\\
\end{cases}$$
Hence the expectation of $Y_i(\beta)$ is:
\allowdisplaybreaks
\begin{align}
E(Y_i(\beta)) &= E(Z_i(\beta)|Z_i(\beta) \neq 0) \notag \\
& = \frac{ E(Z_i(\beta)|Z_i(\beta) \neq 0)P(Z_i(\beta) \neq 0)}{P(Z_i(\beta) \neq 0)} \notag  \\
& =  \frac{E(Z_i(\beta))}{P(\s(X^{\top}\beta) \neq \s(X^{\top}\tilde{\beta}))} \notag  \\
& \le \label{expectation} \frac{-(u_-/2)\|\beta - \beta^0\|^2_2}{P(\s(X^{\top}\beta) \neq \s(X^{\top}\tilde{\beta}))} 
\end{align}
For the rest of the calculations we need to bound $\P(\s(X^{\top}\beta) \neq \s(X^{\top}\tilde{\beta}))$. Towards that direction we have the following: 
\allowdisplaybreaks
\begin{align*}
& \P(\s(X^{\top}\beta) \neq \s(X^{\top}\tilde{\beta})) \\
& = \P(\s(X^{\top}\beta) \neq \s(X^{\top}\tilde{\beta}), \s(X^{\top}\tilde{\beta}) \neq \s(X^{\top}\beta^0)) \\ & \ \ \ \ \  +  \P(\s(X^{\top}\beta) \neq \s(X^{\top}\tilde{\beta}), \s(X^{\top}\tilde{\beta}) = \s(X^{\top}\beta^0)) \\
& \le \P(\s(X^{\top}\beta) \neq \s(X^{\top}\beta^0), \s(X^{\top}\tilde{\beta}) = \s(X^{\top}\beta^0)) +  \P(\s(X^{\top}\tilde{\beta}) \neq \s(X^{\top}\beta^0)) \\
& \le \P(\s(X^{\top}\beta) \neq \s(X^{\top}\beta^0)) +  \P(\s(X^{\top}\tilde{\beta}) \neq \s(X^{\top}\beta^0)) \\
& \le 2C' \|\beta - \beta^0\|_2
\end{align*}
for $T > 1$. For the lower bound we have: 
\allowdisplaybreaks
\begin{align*}
& \P(\s(X^{\top}\beta) \neq \s(X^{\top}\tilde{\beta})) \\
& \ge \P(\s(X^{\top}\beta) \neq \s(X^{\top}\tilde{\beta}), \s(X^{\top}\tilde{\beta}) = \s(X^{\top}\beta^0)) \\
& = \P(\s(X^{\top}\beta) \neq \s(X^{\top}\beta^0), \s(X^{\top}\tilde{\beta}) = \s(X^{\top}\beta^0)) \\
& \ge  \P(\s(X^{\top}\beta) \neq \s(X^{\top}\beta^0)) - \P(\s(X^{\top}\tilde{\beta}) \neq \s(X^{\top}\beta^0)) \\
& \ge c' \|\beta - \beta^0\|_2 - C'(p/n)^{1/3} \\
& \ge (c'/2) \|\beta - \beta^0\|_2 
\end{align*}
when $T>2C'/c'$, where we are also using the fact that $\beta \in B_0^c$. Putting the upper bound in equation \eqref{expectation} we have: 
$$E(Y_i(\beta)) \le - \frac{u_-}{4C'}\|\beta - \beta^0\|_2$$
So, if $P(Y_i(\beta) = 2) = P(Z_i(\beta) = 2\ \vert \ Z_i(\beta) \neq 0) = p_2$ (say) then $4p_2-2 \le (-u_-/4C')\|\beta - \beta^0\|_2$ which implies $p_2 \le \frac{1}{2} - (-u_-/16C')\|\beta - \beta^0\|_2$. Define $W_i(\beta) = \frac{Y_i(\beta) + 2}{4}$. Then $W_i(\beta) \sim \text{Ber}(p_2)$. Let $N$ denote the number of non-zero $Z_i(\beta)$'s. Then $N \sim \text{Bin}(n,p_1)$where $p_1 = \P(\s(X^{\top}\beta) \neq \s(X^{\top}\tilde{\beta}))$.
\allowdisplaybreaks
\begin{align*}
& \hspace*{-0.3in}\P_n\left(\sup_{\beta \in B_i \cap B_{i-1}^c} S_n(\beta) - S_n(\tilde{\beta}) \ge 0\right) \\
& \le \sum_{j:\beta_j \in B_i\cap B_{i-1}^c}\P_n\left(S_n(\beta_j) - S_n(\tilde{\beta}) \ge 0\right) \\
& \le \sum_{j:\beta_j \in B_i\cap B_{i-1}^c}\P_n\left(\sum_{i=1}^n Z_i(\beta_j) \ge 0\right) \\
& \le \sum_{j:\beta_j \in B_i\cap B_{i-1}^c}\sum_{m=1}^n \P_n\left(\sum_{i=1}^n Z_i(\beta_j) \ge 0 | N=m\right)\P(N = m) \\
& \le \sum_{j:\beta_j \in B_i\cap B_{i-1}^c}\sum_{m=1}^n \P_n\left(\sum_{i=1}^m Y_i(\beta_j) \ge 0 \right)\P(N = m) \\
& \le \sum_{j:\beta_j \in B_i\cap B_{i-1}^c}\sum_{m=1}^n \P_n\left(\sum_{i=1}^m W_i(\beta_j) \ge \frac{m}{2} \right)\P(N = m) \\
& \le  \sum_{j:\beta_j \in B_i\cap B_{i-1}^c}\sum_{m=1}^n (4p_2q_2)^{m/2}\dbinom{n}{m}p_1^m(1-p_1)^{n-m} \hspace{0.2in} [\text{Chernoff bound for Binomial tail probability}]\\
& \le  \sum_{j:\beta_j \in B_i\cap B_{i-1}^c} (1-p_1 + 2p_1\sqrt{p_2q_2})^n \\
& \le \sum_{j:\beta_j \in B_i\cap B_{i-1}^c} e^{n\log{(1-p_1(1 - \sqrt{4p_2q_2}))}} \\
& \le \sum_{j:\beta_j \in B_i\cap B_{i-1}^c} e^{n\log{\left(1-p_1\left(1 - \sqrt{1-(u_-^2/64C'^2)\|\beta - \beta^0\|^2}\right)\right)}} \\
& \le \sum_{j:\beta_j \in B_i\cap B_{i-1}^c} e^{n\log{\left(1-p_1\left((u_-^2/128C'^2)\|\beta - \beta^0\|^2_2\right)\right)}} \\
& \le \sum_{j:\beta_j \in B_i\cap B_{i-1}^c} e^{-np_1\left((u_-^2/128C'^2)\|\beta - \beta^0\|^2_2\right)} \\
& \le \sum_{j:\beta_j \in B_i\cap B_{i-1}^c} e^{-n\left(u_-^2c'/256C'^2)\|\beta - \beta^0\|^3_2\right)} \\
& \le \sum_{j:\beta_j \in B_i\cap B_{i-1}^c} e^{- p(u_-^2c'/256C'^2)T^32^{3(i-1)}} \\
& \le e^{\log{2T_np_i} - p(u_-^2c'/256C'^2)T^32^{3(i-1)}} \\
& \le e^{p\left[(\log{a_n}/p) + \log{T}+(i+1)\log{2} - (u_-^2c'/256C'^2)T^32^{3(i-1)}\right]} \hspace*{0.2in} [\text{by Lemma} \ \ref{upper-bound}]
\end{align*}
Thus we can take $a_n = p$ to ignore the effect of $(\log{a_n}/p)$. Putting this back in equation (1) we get:
\allowdisplaybreaks
\begin{align*}
& \hspace{-0.5in} \sum_{i=1}^{A_n} \P\left(\sup_{\beta \in B_i \cap B_{i-1}^c} S_n(\beta) - S_n(\tilde{\beta}) \ge 0\right) \\
& \le 2 \sum_{i=1}^{A_n} e^{p\left[\log{T}+ i\log{2} - (u_-^2c'/256C'^2)T^32^{3(i-1)}\right]} \\
& \rightarrow 0 \,\,\, \text{as} \,\,\, p \rightarrow \infty \,, 
\end{align*}
for large enough $T$ (by using similar arguments to the one used for handling the earlier series) which proves the theorem.

\subsection{Proof of Theorem \ref{minimax-high-lower}}
We use Fano's inequality along with the Gilbert-Varshamov Lemma to prove the minimax lower bound. Fano's inequality (or Local-Fano's inequality) gives us a lower bound on the minimax risk as follows:  
If $\Theta' \subseteq \Theta$ is a finite $2\epsilon$ packing set, i.e. for any two $\theta_i, \theta_j \in \Theta'$, $\|\theta_i - \theta_j\|_2 \ge 2\epsilon$ with $|\Theta'| = M$, then, based on $n$ i.i.d. samples $z_1, z_2, \dots, z_n \sim P_{\theta}$ we have the following minimax lower bound: $$\inf_{\hat{\theta}}\sup_{\theta \in \Theta} \E\left(\|\hat{\theta} - \theta\|^2\right) \ge \epsilon^2\left(1-\frac{\frac{n}{M^2}\sum_{i,j}KL(\P_{\theta_i} || \P_{\theta_j}) + \log{2}}{\log{(M-1)}}\right)$$
The crux of the proof relies on constructing competing models that approach each other at an optimal rate, as $n$ increases. We start with a 
preliminary lemma. 
\begin{lemma}
\label{KL}
If $P \sim Ber(p_1)$ and $Q \sim Ber(q_1)$ and if $\frac{1}{4} \le q_1 \le \frac{3}{4}$, then $KL(P||Q) \le \frac{16}{3}(p-q)^2$
\end{lemma}
\noindent
The proof of the Lemma appears in supplement in section \ref{Proof of KL}. We next state the \emph{Gilbert-Varshamov Lemma} for convenience (see \cite{raskutti2011minimax} and references therein), that guides the construction of $\Theta^{'}$ in our problem.

\begin{lemma}[{\bf Gilbert-Varshamov}]
Define $d_H$ to be the Hamming distance, i.e. $d_H(x,y) = \sum_{i=1}^d\mathds{1}_{x_i \neq y_i}$ with $d$ being the underlying dimension. Given any $s$ with $1 \le s \le \frac{d}{8}$, we can find $w_1, \cdots, w_M \in \{0,1\}^d$ such that:
\begin{enumerate}[a)]
\item $d_H(w_i, w_j) \ge \frac{s}{2} \,\, \forall \,\, i \neq j \in \{1,2, \cdots, M\}$.
\item $\log{M} \ge \frac{s}{8}\log{\left(1+\frac{d}{2s}\right)}$
\item $\|w_j\|_0 = s \ \forall j \ \in\{1,2, \cdots, M\}$.
\end{enumerate}
\end{lemma}
\noindent
Fix $0 < \delta < 1/4$. To construct a $2\epsilon$ packing set $\left(\epsilon = \frac{\delta}{4}\right)$ of $S^{p-1}$, consider the following vectors: $$\beta_J = \frac{\left(1, \frac{\delta}{\sqrt{s}}w_J\right)}{\sqrt{1+\delta^2}} \,.$$
where $w_J \in W$, a subset of $\{0,1\}^{p-1}$ constructed using GV lemma. Let $\Theta' = \{\beta_J: J \in \{1,2, \cdots, M\}\} \subseteq \Theta = S^{p-1}$. For $I \neq J$, $$\|\beta_I- \beta_J\|_2^2 = \frac{\delta^2}{s(1+\delta^2)}d_H(w_I, w_J) \ge \frac{\delta^2}{4} \,.$$
For notational simplicity define $m(\delta) = \sqrt{1+\delta^2}$ and $C = C_n$. Fix $\alpha \ge 1$. Denote $\P_{\beta_J}(X,Y)$ as the joint distribution of $(X,Y)$ where $X \sim \mathcal{N}(0,I_p)$ and   $$P_{\beta_J}(Y = 1|X)=
    \begin{cases}
     \frac{1}{2}+ (\beta_J^\mathsf{T}X)/C, & \text{if} \,\, |\beta_J^\mathsf{T}X| \le \left(C\delta \vee \frac{|X_1|}{2m(\delta)}\right)\wedge \frac{1}{4}  \\
      \frac{1}{2}+\left[\left(\delta \vee \frac{|X_1|}{2Cm(\delta)}\right)\wedge \frac{1}{4}\right]\s(\beta_J^\mathsf{T}X), & \text{if} \,\, |\beta_J^\mathsf{T}X| > \left(C\delta \vee \frac{|X_1|}{2m(\delta)}\right)\wedge \frac{1}{4}
    \end{cases}$$
for all $J \in \{1,2, \cdots, M\}$. Now, for any $\beta_J \in \Theta'$, we have $\beta_J^\mathsf{T}X = \frac{X_1}{m(\delta)} + \tilde{\beta}_J^\mathsf{T}\tilde{X}$ where $\tilde{a} = (a_2, a_3, \cdots. a_p)$. As $\|\tilde{\beta}_J\|_0 = s$ by construction, we know $\tilde{\beta}_J^\mathsf{T}\tilde{X} = \frac{\delta Z_J}{m(\delta)}$ where $Z_J \sim \mathcal{N}(0,1)$ and independent of $X_1$. Thus, we have $\beta_J^\mathsf{T}X = \frac{X_1 + \delta Z_J}{m(\delta)}$. %Next we will show that, this construction obeys the margin assumption for $\alpha =1$ and $t^* = \frac{1}{4}$.

\begin{lemma} 
\label{low_noise_minimax_high}
The above family of distributions satisfy the margin assumption (Assumption \ref{ass:low_noise_binary}) for all $0 < t < 1/4$. 
\end{lemma}

\begin{comment} 
\begin{remark}
From the statement of the above Lemma, it is clear that the competing distributions constructed in this proof satisfy the margin condition on a sub-interval $(c, t^*)$ of $(0,t^*)$. Hence, we are really establishing a minimax lower bound for a class of models containing $\mathcal{P}$ (e.g. models satisfying a relaxed version of Assumption \ref{} with the inequality holding on a sub-interval of $(0,t^{\star})$), which then is obviously 
\end{remark}
\end{comment} 

\begin{proof}
Fix any $0 <  t < \frac{1}{4}$.
%\begin{equation*}
\begin{align*}
\allowdisplaybreaks
 \P_X(|\eta(X) - 0.5| \le t)
& =   \P_X\left(|\eta(X) - 0.5| \le t, |\beta_J^\mathsf{T}X| \le \left[C\delta \vee \frac{|X_1|}{2m(\delta)} \right ]\wedge 1/4\right) \\ + & \P_X\left(|\eta(X) - 0.5| \le t, |\beta_J^\mathsf{T}X| > \left[C\delta \vee \frac{|X_1|}{2m(\delta)} \right ]\wedge 1/4\right) \\
& =   \P_X\left( |\beta_J^\mathsf{T}X| \le Ct, |\beta_J^\mathsf{T}X| \le \left[C\delta \vee \frac{|X_1|}{2m(\delta)} \right ]\wedge 1/4\right) \\ +& \P_X\left(\left[\delta \vee \frac{|X_1|}{2Cm(\delta)} \right ]\wedge 1/4 \le t, |\beta_J^\mathsf{T}X| > \left[C\delta  \vee \frac{|X_1|}{2m(\delta)} \right ]\wedge 1/4\right) \\
& \le  \, \P_X(|\beta_J^\mathsf{T}X| \le t/\xi) + \P_X\left(\left[\delta \vee \frac{|X_1|}{2Cm(\delta)} \right ]\wedge 1/4 \le t\right) \\
& \le \, \P_X(|\beta_J^\mathsf{T}X| \le Ct) + \P_X(|X_1| \le 2Cm(\delta)t) \\
& \le  \, \sqrt{\frac{2}{\pi}}\left[Ct + 2Cm(\delta)t\right] \\
& \le  \, 5\sqrt{\frac{2}{\pi}} Ct \, .
\end{align*}
%\end{equation*}
\end{proof}
\noindent
Define the event $A_I = \left\{|\beta_I^\mathsf{T}X| \le \left(C\delta \vee \frac{|X_1|}{2m(\delta)}\right)\wedge \frac{1}{4}\right\}$. Then we have the following lemma:
\begin{lemma}
If $X \in A_i \cup A_j$, then $$\left|P_{\beta_I}(Y=1|X) - P_{\beta_J}(Y=1|X)\right| \le \frac{1}{C} \left|\beta_I^\mathsf{T}X - \beta_J^\mathsf{T}X\right| =  \frac{1}{C} \left|\tilde{\beta}_I^\mathsf{T}\tilde{X} - \tilde{\beta}_J^\mathsf{T}\tilde{X}\right|$$
\end{lemma}
The proof follows the same arguments as that of Lemma \ref{intermediate-lemma} and is skipped. 
\begin{comment}
\begin{proof}
The last equality in the lemma is trivial as $\beta_I^\mathsf{T}X - \beta_J^\mathsf{T}X = \tilde{\beta}_I^\mathsf{T}\tilde{X} - \tilde{\beta}_J^\mathsf{T}\tilde{X}$ by construction of $\beta_I, \beta_J$. To prove the first inequality, at first assume that, $X \in A_I \cap A_J$. Then, clearly $$\left|P_{\beta_I}(Y=1|X) - P_{\beta_J}(Y=1|X)\right| = |\beta_1'X - \beta_2'X| $$
On the other hand, there may be another situation where say, $X \in A_I \cap A_J^c$. Then, $|\beta_J^\mathsf{T}X| > \left(\delta \vee \frac{|X_1|}{2m(\delta)}\right) \wedge 1/4$ but $|\beta_I^\mathsf{T}X| < \left(\delta \vee \frac{|X_1|}{2m(\delta)}\right) \wedge 1/4$. Hence, then $$\left|P_{\beta_1}(Y=1|X) - P_{\beta_2}(Y=1|X)\right| \le  |\beta_I^\mathsf{T}X - \beta_J^\mathsf{T}X| $$ which concludes the lemma.
\end{proof}
\end{comment} 
Next, we upper-bound the KL divergence:
\begin{lemma}
For any $I \neq J \in \{1,2, \cdots, M\}$, we have $$KL(\P_{\beta_I} || \P_{\beta_J}) \le \frac{128}{3}\sqrt{\frac{2}{\pi}}\left[6 + \sqrt{6} + \frac{4\phi(3)}{27}\right]\frac{\delta^3}{C^2}\,,$$
where $\phi$ is the standard normal density. 
\end{lemma}
\begin{proof}
\begin{align}
KL(\P_{\beta_I} || P_{\beta_J}) \notag
 & =  \E_X\left(KL(P_{\beta_I}(Y|X) || P_{\beta_J}(Y|X))\right) \notag \\
& \le \frac{16}{3}\E_X(P_{\beta_I}(Y|X)- P_{\beta_J}(Y|X))^2 \notag \\
& \le   \frac{16}{3}\left[\frac{1}{C^2}\E_X\left((\beta_I^\mathsf{T}X - \beta_J^\mathsf{T}X)^2\mathds{1}_{X \in A_i \cup A_j}\right) \right. \notag \\ & \ \ \ \ \ \ \ \ \  \left. + 4\E_X\left(\left[\left(\delta \vee \frac{|X_1|}{2Cm(\delta)}\right)\wedge \frac{1}{4}\right]^2\mathds{1}_{X \in A_i^c \cap A_j^c, \s(\beta_I^{\top}X) \neq \s(\beta_J^{\top}X)}\right)\right] \notag \\
\label{eq1-mlbh}& = \frac{16}{3}(S_1 + 4S_2) \hspace*{0.2in} [\text{say}] 
\end{align}
We analyze each summand separately, starting with $S_1$. 
\allowdisplaybreaks
\begin{align}
& \E_X\left((\beta_I^\mathsf{T}X - \beta_J^\mathsf{T}X)^2\mathds{1}_{X \in A_i \cup A_j}\right) \notag \\
& \le \, 2 \E_X\left((\beta_I^\mathsf{T}X - \beta_J^\mathsf{T}X)^2\mathds{1}_{X \in A_i}\right) \notag \\
& \le \, 2\E_X\left((\beta_I^\mathsf{T}X - \beta_J^\mathsf{T}X)^2\mathds{1}_{|\beta_I^\mathsf{T}X| \le \left(C\delta  \vee \frac{|X_1|}{2m(\delta)}\right)}\right) \notag \\
& \le \, 2\left[\E_X\left((\beta_I^\mathsf{T}X - \beta_J^\mathsf{T}X)^2\mathds{1}_{|X_1| \le 2m(\delta)|\tilde{\beta}_I^\mathsf{T}\tilde{X}|}\right)+\E_X\left((\beta_I^\mathsf{T}X - \beta_J^\mathsf{T}X)^2\mathds{1}_{|X_1| \le 2Cm(\delta)\delta}\right)\right] \notag \\
& \le \, 2\left[\E_{\tilde{X}}\left((\tilde{\beta}_I^\mathsf{T}\tilde{X} - \tilde{\beta}_J^\mathsf{T}\tilde{X})^2 P_{X_1|\tilde{X}}\left(|X_1| \le 2m(\delta)|\tilde{\beta}_I^\mathsf{T}\tilde{X}|\right)\right)+\E_{\tilde{X}}\left((\tilde{\beta}_I^\mathsf{T}\tilde{X} - \tilde{\beta}_J^\mathsf{T}\tilde{X})^2\right)P_{X_1}\left(|X_1| \le 2Cm(\delta)\delta\right)\right] \notag  \\
& \le \, \sqrt{\frac{8}{\pi}} \left[2m(\delta)\E_{\tilde{X}}\left((\tilde{\beta}_I^\mathsf{T}\tilde{X} - \tilde{\beta}_J^\mathsf{T}\tilde{X})^2|\tilde{\beta}_I^\mathsf{T}\tilde{X}|\right) + \frac{8}{m(\delta)}C\delta^3\right] \notag \\
& \le \, 4\sqrt{\frac{2}{\pi}} \left[m(\delta)\left(\E_{\tilde{X}}(\tilde{\beta}_I^\mathsf{T}\tilde{X} - \tilde{\beta}_J^\mathsf{T}\tilde{X})^4\right)^{\frac{1}{2}}\left(\E_{\tilde{X}}|\tilde{\beta}_I^\mathsf{T}\tilde{X}|^2\right)^{\frac{1}{2}} + \frac{4}{m(\delta)}\delta^3\right] \hspace{0.2in} [\because C < 1]\notag \\
& \le \, 4\sqrt{\frac{2}{\pi}} \left[\frac{\sqrt{12}}{m(\delta)^2} + \frac{4}{m(\delta)}\right]\delta^3 \notag \\
& \le \, 8\sqrt{\frac{2}{\pi}} \left[2 + \sqrt{3}\right]\delta^3 \,.\label{eq2-mlbh}
\end{align}
Now, on to $S_2$:
\allowdisplaybreaks
\begin{align}
& \E_X\left(\left(\left(\delta \vee \frac{|X_1|}{2C m(\delta)}\right)\wedge 1/4\right)^2\mathds{1}_{X \in A_i^c \cap A_j^c, \s(X^{\top}\beta_I) \neq \s(X^{\top}\beta_J)}\right) \notag \\
& = 2\E_X\left(\left(\left(\delta \vee \frac{|X_1|}{2Cm(\delta)}\right)\wedge 1/4\right)^2\mathds{1}_{\frac{X_1 + \delta Z_I}{m(\delta)} \ge \left(C\delta \vee \frac{|X_1|}{2m(\delta)}\right) \wedge 1/4} \mathds{1}_{\frac{X_1 + \delta Z_J}{m(\delta)} \le -\left(\left(C\delta  \vee \frac{|X_1|}{2m(\delta)}\right)\wedge 1/4\right)}\right) \notag \\
& = 2\E_X\left(\delta^2 \mathds{1}_{\frac{X_1 + \delta Z_I}{m(\delta)} \ge \delta} \mathds{1}_{\frac{X_1 + \delta Z_J}{m(\delta)} \le -\delta}\mathds{1}_{|X_1| \le 2Cm(\delta)\delta}\right) \notag \\ 
&  \qquad + 2\E_X\left(\frac{1}{C^2}\frac{|X_1|^2}{4m(\delta)^2}\mathds{1}_{X_1 + \delta Z_I \ge \frac{|X_1|}{2}} \mathds{1}_{X_1 + \delta Z_J \le -\frac{|X_1|}{2}}\mathds{1}_{\delta \le \frac{|X_1|}{2Cm(\delta)} \le 1/4}\right)  \notag \\ 
& \qquad \qquad + 2\E_X\left(\frac{1}{16}\mathds{1}_{X_1 + \delta Z_I \ge \frac{m(\delta)}{4}} \mathds{1}_{X_1 + \delta Z_J \le -\frac{m(\delta)}{4}}\mathds{1}_{\frac{|X_1|}{2Cm(\delta)} \ge 1/4}\right) \notag \\
& \le 2\E_X\left(\delta^2 \mathds{1}_{\frac{X_1 + \delta Z_I}{m(\delta)} \ge \delta} \mathds{1}_{\frac{X_1 + \delta Z_J}{m(\delta)} \le -\delta}\mathds{1}_{|X_1| \le 2Cm(\delta)\delta}\right) \notag \\ 
& \qquad + 2\E_X\left(\frac{1}{C^2}\frac{|X_1|^2}{4m(\delta)^2}\mathds{1}_{X_1 + \delta Z_I \ge \frac{|X_1|}{2}} \mathds{1}_{X_1 + \delta Z_J \le -\frac{|X_1|}{2}}\mathds{1}_{\delta \le \frac{|X_1|}{2Cm(\delta)}}\right)  \notag \\
& \qquad \qquad + 2\E_X\left(\frac{1}{16}\mathds{1}_{X_1 + \delta Z_I \ge \frac{1}{4}} \mathds{1}_{X_1 + \delta Z_J \le -\frac{1}{4}}\mathds{1}_{|X_1| \ge C/2}\right) \hspace{0.2in} [\because m(\delta) \ge 1]\notag \\
& \le 2\left[2\sqrt{\frac{2}{\pi}}Cm(\delta)\delta^3 + 2\E_X\left(\frac{1}{C^2}\frac{|X_1|^2}{4m(\delta)^2}\mathds{1}_{X_1 + \delta Z_I \ge \frac{X_1}{2}} \mathds{1}_{X_1 + \delta Z_J \le -\frac{X_1}{2}}\mathds{1}_{X_1 \ge 2Cm(\delta)\delta}\right) \right. \notag \\ 
& \qquad \qquad \left. + 2\E_X\left(\frac{1}{16}\mathds{1}_{X_1 + \delta Z_I \ge \frac{1}{4}} \mathds{1}_{X_1 + \delta Z_J \le -\frac{1}{4}}\mathds{1}_{X_1 \ge C/2}\right)\right] \notag \\
& \le 2\left[2\sqrt{\frac{2}{\pi}}Cm(\delta)\delta^3 + 2\E_X\left(\frac{1}{C^2}\frac{|X_1|^2}{4m(\delta)^2}\mathds{1}_{X_1 \ge - 2\delta Z_I} \mathds{1}_{X_1  \le -\frac{2}{3}\delta Z_J}\mathds{1}_{X_1 \ge 2Cm(\delta)\delta}\right)\right. \notag\\ 
& \qquad \qquad + \left. 2\E_X\left(\frac{1}{16}\mathds{1}_{X_1 + \delta Z_J \le -\frac{1}{4}}\mathds{1}_{X_1 \ge C/2}\right)\right] \notag \\
& \le 2\left[2\sqrt{\frac{2}{\pi}}Cm(\delta)\delta^3 + 2\E_X\left(\frac{1}{C^2}\frac{|X_1|^2}{4m(\delta)^2}\mathds{1}_{(- 2\delta Z_I \vee 2Cm(\delta)\delta) \le X_1 \le -\frac{2}{3}\delta Z_J}\right)\right. \notag\\ 
& \qquad \qquad + \left. 2\E_X\left(\frac{1}{16}\mathds{1}_{Z_J \le -\frac{3}{4\delta}}\mathds{1}_{X_1 \ge C/2}\right)\right] \notag \\
& \le 2\left[2\sqrt{\frac{2}{\pi}}m(\delta)\delta^3 + 2\E_X\left(\frac{1}{C^2}\frac{|X_1|^2}{4m(\delta)^2}\mathds{1}_{ 2Cm(\delta)\delta \le X_1 \le -\frac{2}{3}\delta Z_J}\mathds{1}_{Z_J \le -3m(\delta)}\right) + \frac{1}{8}e^{-\frac{9}{32 \delta^2}}\right] \notag \\
& \le 2\left[2\sqrt{\frac{2}{\pi}}m(\delta)\delta^3 + \frac{2\delta^2}{9C^2m(\delta)^2}\E_X\left(Z_J^2\mathds{1}_{ 2Cm(\delta)\delta \le X_1 \le -\frac{2}{3}\delta Z_J}\mathds{1}_{Z_J \le -3m(\delta)}\right)+ \frac{1}{8}e^{-\frac{9}{32 \delta^2}}\right] \notag \\
& \le 4\left[\sqrt{\frac{2}{\pi}}m(\delta)\delta^3 + \frac{\delta^2}{9C^2m(\delta)^2}\E_{Z_J}\left(Z_J^2\left[\Phi\left(-\frac{2}{3}\delta Z_J\right) - \Phi \left(2Cm(\delta)\delta \right)\right]\mathds{1}_{Z_J \le -3m(\delta)}\right) + \frac{1}{16}e^{-\frac{9}{32 \delta^2}}\right] \notag \\
& \le 4\sqrt{\frac{2}{\pi}}\left[m(\delta)\delta^3 + \frac{\delta^3}{27C^2m(\delta)^2}\E_{Z_J}\left(Z_J^2\left(-Z_J - 3Cm(\delta)\right)\mathds{1}_{Z_J \le -3m(\delta)}\right)+\frac{\sqrt{\pi}}{16\sqrt{2}}e^{-\frac{9}{32 \delta^2}}\right] \notag \\
& \le 4\sqrt{\frac{2}{\pi}}\left[m(\delta)\delta^3 + \frac{\delta^3}{27C^2m(\delta)^2}\E_{Z_J}\left(Z_J^2|Z_J + 3Cm(\delta)|\right)+\frac{\sqrt{\pi}}{16\sqrt{2}}e^{-\frac{9}{32 \delta^2}}\right] \notag \\
& \le 4\sqrt{\frac{2}{\pi}}\left[m(\delta) + \frac{1}{27m(\delta)^2}(\sqrt{3} +6) + \frac{\sqrt{\pi}}{16\sqrt{2}}\right]\frac{\delta^3}{C^2} \hspace{0.2in} [ \because e^{-\frac{9}{32 \delta^2}} \le \delta^3, \ m(\delta) \le 2, \ C < 1] \notag \\
\label{eq3-mlbh} & \le 8\sqrt{\frac{2}{\pi}}\left[1 + \frac{\sqrt{3} + 6}{54} +  \frac{\sqrt{\pi}}{32\sqrt{2}}\right]\frac{\delta^3}{C^2} 
\end{align}
Combining equations \ref{eq1-mlbh}, \ref{eq2-mlbh} and \ref{eq3-mlbh}  we conclude that: $$KL(\P_{\beta_I} || \P_{\beta_J}) \le \frac{128}{3}\sqrt{\frac{2}{\pi}}\left[6 + \sqrt{3} + \frac{2(\sqrt{3}+ 6)}{27} + \frac{\sqrt{\pi}}{8\sqrt{2}}\right]\frac{\delta^3}{C^2} \,.$$
\end{proof}
The final step is a direct application of Fano's inequality. According to our construction, $\Theta'$ is a $2\epsilon$ packing set with $\epsilon = \frac{\delta}{4}$. For notational simplicity, set $$U_c \triangleq \frac{128}{3}\sqrt{\frac{2}{\pi}}\left[6 + \sqrt{3} + \frac{2(\sqrt{3}+ 6)}{27} + \frac{\sqrt{\pi}}{8\sqrt{2}}\right]\,.$$ 
The upper bound on the KL divergences, in conjunction with Fano's inequality, gives: 
$$\inf_{\hat{\beta}}\sup_{\P_{\beta}} \E\left(\|\hat{\beta} - \beta\|^2\right) \ge \frac{\delta^2}{16}\left(1-\frac{(nU_c\delta^3)/C^2+ \log{2}}{\frac{s}{32}\log{\frac{p}{s}}}\right) \,.$$ 
Taking $\delta = \left(\frac{\frac{s}{64}\log{\frac{p}{s}}}{nU_c}\right)^{\frac{1}{3}}C^{\frac{2}{3}}$, then we have: $$\inf_{\hat{\beta}}\sup_{\P_{\beta}} \E\left(\|\hat{\beta} - \beta\|^2\right) \ge \frac{1}{256U_c^{\frac{2}{3}}}\left(\frac{s\log{\frac{p}{s}}}{n}\right)^{\frac{2}{3}}C^{\frac{4}{3}}\left(1-\frac{\frac{s}{64}\log{\frac{p}{s}}+ \log{2}}{\frac{s}{32}\log{\frac{p}{s}}}\right) \ge \frac{1}{2^{10}U_c^{\frac{2}{3}}}\left(\frac{s\log{\frac{p}{s}}C^2}{n}\right)^{\frac{2}{3}}\;,$$ the last inequality holding true when $\log{2}\le \frac{s}{128}\log{\frac{p}{s}}$, which is true for all large $s,p$ as $s\log{\frac{p}{s}} \rightarrow \infty$. 
\\\\
The other inequality (i.e. we cannot estimate at a better rate than $(s_0\log{p/s_0}/n)$), essentially follows from the same argument with taking $C_n = 0$. We skip the details here for the sake of brevity.  $\Box$

\begin{appendix}
\section{Some important results}
In this section we state some results from the existing literature for the convenience of the readers which we use in our proofs. Theorem \ref{massart} is Theorem 2 of \cite{massart2006risk} which provides some exponential concentration bound on the ERM estimators for bounded loss functions. Lemma \ref{miwep} is a classical maximal inequality, which is used to bound the fluctuations of an empirical process. A simple proof of this Lemma can be found in \cite{massart2006risk}. Theorem \ref{our_model_selection} is a modified version of Theorem 8.5 of \cite{massart2007concentration}, which we use for our model selection consistency results in case of $p \gg n$. We provide the proof of Theorem \ref{our_model_selection} in this supplement. Theorem \ref{thm:bousquet_talagrand} is a version of Talagrand's inequality (also known as Bousquet's version of Talagrand inequality, see \cite{bousquet2002bennett}) which we use to prove Theorem \ref{our_model_selection}.

\begin{theorem}
\label{massart}
Let $\{Z_i = (X_i, Y_i)\}_{i=1}^n$ be i.i.d. observations taking values in the sample space $\mathcal{Z} : \mathcal{X} \times \mathcal{Y}$ and let $\mathcal{F}$ be a class of real-valued functions defined on $\mathcal{X}$. Let $\gamma: \mathcal{F} \times \mathcal{Z} \rightarrow [0,1]$ be a loss function, and suppose that $f^* \in \mathcal{F}$ uniquely minimizes the expected loss function $P(\gamma(f, .))$ over  $\mathcal{F}$. Define the empirical risk as $\gamma_n(f) = (1/n) \sum_{i=1}^n \gamma(f, Z_i)$, and $\bar{\gamma}_n(f) = \gamma_n(f) - P(\gamma(f, .))$. Let $l(f^*, f) = P(\gamma(f, .)) - P(\gamma(f^*, .))$ be the excess risk. Consider a pseudo-distance $d$ on $\mathcal{F} \times \mathcal{F}$ satisfying $Var_P[\gamma(f, .) - \gamma(g, .)] \le d^2(f,g)$. Finally, let $C_1$ be the collection of all functions $\{h: \mathbb{R}^+ \rightarrow \mathbb{R}^+\}$ such that, $h$ is non-decreasing, continuous with $h(x)/x$ is non-increasing on $[0, \infty)$ and $h(1) \ge 1$. Assume that: 
\begin{enumerate}[(1)]
\item There exists $F \subseteq \mathcal{F}$ and a countable subset $F' \subseteq F$, such that for each $f \in F$, there is a sequence $\{f_k\}$ of elements of $F'$ satisfying $\gamma(f_k, z) \rightarrow \gamma(f, z)$ as $k \rightarrow \infty$, for every $z \in \mathcal{Z}$. 
\item $d(f,f^*) \le \omega\left(\sqrt{l(f^*, f)}\right) \ \forall \ f \in \mathcal{F}$, for some function $\omega \in C_1$.
\item For every $f \in F'$ $$\sqrt{n}\E\left[\sup_{g \in F': d(f,g) \le \sigma}\left[\bar{\gamma}_n(f)- \bar{\gamma}_n(g)\right]\right] \le \phi(\sigma)$$ for every $\sigma > 0$ such that $\phi(\sigma) \le \sqrt{n} \sigma^2$, where $\phi \in C_1$. 
\end{enumerate}
Let $\epsilon_*$ be the unique positive solution of $\sqrt{n}\epsilon_*^2 = \phi(\omega(\epsilon_*))$. Let $\hat{f}$ be the (empirical) minimizer of $\gamma_n$ over $F$ and $l(f^*, F) = \inf_{f \in F}l(f^*, f)$.Then, there exists an absolute constant $K$ such that for all $y \ge 1$, the following inequality holds: $$\P\left(l(f^*, \hat{f}) > 2l(f^*, F) + Ky\epsilon_*^2\right) \le e^{-y} \,.$$ 
\end{theorem}

\begin{lemma}[A maximal inequality for weighted empirical process]
\label{miwep}
Let $S$ be a countable set, $u \in S$ and $a: S \rightarrow \mathbb{R}_+$ such that $a(u) = \inf_{t \in S} a(t)$. Let $Z$ be a process indexed by $S$ and assume that the non-negative random variable $\sup_{t \in \mathcal{B}(\epsilon)}[Z(u) - Z(t)]$ has finite expectation for any positive number $\epsilon$, where $\mathcal{B} (\epsilon) = \{t \in S, \  a(t) \le \epsilon\}$. Let $\psi$ be a non-negative function on $\mathbf{R}_+$ such that $\psi(x)/x$ is non-increasing on $\mathbf{R}_+$ and satisfies for some positive number $\epsilon_*$: $$\E\left[\sup_{t \in \mathcal{B}(\epsilon)}[Z(u) - Z(t)]\right] \le \psi(\epsilon) \ \forall \ \epsilon \ge \epsilon_* \,.$$ Then, one has, for any positive number $x \ge \epsilon_*$, $$\E\left[\sup_{t \in S}\frac{[Z(u) - Z(t)]}{a^2(t) + x^2}\right] \le \frac{4\psi(x)}{x^2} \,.$$
\end{lemma}

\begin{theorem}[Model selection consistency]
\label{our_model_selection}
Let $\xi_1, \dots, \xi_n$ be independent observations taking their values in the measurable space $\Xi$ with common distribution $P$. Let $\mathcal{S}$ be some set, $\gamma: \mathcal{S} \times \Xi \rightarrow [0,1]$, be a measurable function such that for every $t\in \mathcal{S}$, $x \rightarrow \gamma(t,x)$ is measurable. Assume that there exists some minimizer $s$ of $P(\gamma(t,\cdot))$ over $\mathcal{S}$ and define $\ell(s,t)$ as the excess risk:  
$$\ell(s, t) = P(\gamma(t,\cdot)) - P(\gamma(s,\cdot))$$ for every $t \in \mathcal{S}$. Let $\gamma_n$ be the empirical risk: 
$$\gamma_n(t) = P_n(\gamma(t,\cdot)) = \frac1n\sum_{i=1}^n \gamma(t,\xi_i), \ \text{for every} \ t \in \mathcal{S}$$ 
and $\bar \gamma_n$ be corresponding centered empirical process defined by $$\bar \gamma_n(t) = P_n(\gamma(t,\cdot)) - P(\gamma(t, \cdot)) \ \text{for every} \ t \in \mathcal{S}$$ Let $d$ be some psuedo-distance on $\mathcal{S} \times \mathcal{S}$ such that $$P\left((\gamma(t,\cdot) - \gamma(s,\cdot))^2\right) \le d^2(s,t) \ \text{for every} \ t \in \mathcal{S}\,.$$
Let $\{S_m\}_{m \in \mathcal{M}}$ be some, at most, countable collection of subsets of $\mathcal{S}$, each model $S_m$ admitting some countable subset $S'_m$ such that for every $t \in S_m$, there exists some sequence $\{t_k\}_{k\ge1}$ of elements of $S'_m$ satisfying $\gamma(t_k, \xi) \to \gamma(t, \xi)$ as $k \to \infty$, for every $\xi \in \Xi$. Let $\omega$ and $\phi_m$ belong to class of functions $C_1$ (defined in Theorem \ref{massart})  for all $m \in \mathcal{M}$. Assume one hand $$d(s,t) \le \omega(\sqrt{\ell(s,t)}) \ \text{for every} \ t \in \mathcal{S} \,,$$ and on the other hand one has for every $m \in \mathcal{M}$ and $u \in S'_m$: 
$$\sqrt{n}\E\left[\sup_{t \in S'_m, d(u,t) \le \sigma}|\bar{\gamma_n}(t) - \bar{\gamma_n(s)}|\right] \le \phi_m(\sigma)$$ for every positive $\sigma$ such that $\phi_m(\sigma) \le \sqrt{n}\sigma^2$. Let $\epsilon_m$ be the unique solution of the equation:
$$\sqrt{n}\epsilon_m^2 = \phi_m(\omega(\epsilon_m)) \,,$$ 
with $\epsilon_m \le 1 \ \forall \ m \in \mathcal{M}$. Let $\hat s_m \in S_m$ be the empirical minimizer: 
$$\gamma_n(\hat s_m) = \inf_{t \in S_m} \gamma_n(t) \,,$$ and $\{x_m\}_{m \in \mathcal{M}}$ be some family of nonnegative weights such that $$\sum_{m \in M}e^{-x_m} \le \Sigma < \infty.$$ 
Consider a penalty function \pen: $\mathcal{M} \rightarrow R_+$ such that for every $m \in \mathcal{M}$, $$\pen(m) \ge K\left(\epsilon_m^2 + \frac{\omega^2(\epsilon_m)}{n\epsilon_m^2}x_m\right)$$
for some judiciously chosen constant $K$. Define the chosen model as $\hat m$, i.e.: 
$$\hat m = \argmin_{m \in \mathcal{M}} \left[\gamma_n(\hat s_m) + \pen(m)\right] \,.$$ 
Also, define $m_{(1)} = \arg\min_{m \in \mathcal{M}}\epsilon_m$ and $b(n) = \omega^2\left(\epsilon_{m_{(1)}}\right)/\epsilon^2_{m_{(1)}}$. Then the penalized estimator $\tilde{s} = \hat s_{\hat m}$ satisfies the following inequality: 
$$\P\left(\ell(\tilde{s}, s)) > C\left[\inf_{m \in \mathcal{M}}\left(\ell(s, S_m) + \pen(m)\right)\right] + C_1\frac{tb(n)}{n}\right) \le \Sigma e^{-t}$$
where the constants $C, C_1$ depend on $K$. This immediately implies: 
$$\E(\ell(\tilde{s}, s))\le C_2\left[\inf_{m \in \mathcal{M}}\left(\ell(s, S_m) + \pen(m)\right) + \frac{\Sigma b(n)}{n}\right]$$
for some constant $C_2$ depending on $C, C_1$. 
\end{theorem}

\begin{theorem}[Bousquet's version of Talagrand inequality]
\label{thm:bousquet_talagrand}
Let $\mathcal{F}$ be a countable family of measurable functions such that for some positive constants $v,b$ one has for all $f \in \mathcal{F}$, $\mbox{Var}_P(f) \le v$ and $\|f\|_{\infty} \le \mathsf{b}$. Then for all $y \ge 0$: 
$$\P\left(Z - \E(Z) \ge \sqrt{2\frac{\left(v + 4b\E(Z)\right)y}{n}} + \frac{2\mathsf{b}y}{3n}\right) \le e^{-y} \,,$$
where $Z = \sup_{f \in \mathcal{F}}\left(\P_n - P\right)f$. 
\end{theorem}

\begin{remark} 
The result above extends to an uncountable family $\mathcal{F}$ if there exists a countable $\mathcal{F}' \subset \mathcal{F}$ with the property that 
for every $\tilde{f} \in \mathcal{F}$, there is a sequence $\{\tilde{f}_j\}$ belonging to $\mathcal{F}'$ such that $\tilde f_j(\cdot) \rightarrow \tilde f(\cdot)$ 
pointwise. This is indeed the case for all applications of this result in our paper. 
\end{remark}

\section{Proofs of Theorems and Lemmas}\label{proofs}
\subsubsection{Proof of Proposition \ref{alpha-kappa}}
To prove Proposition \ref{alpha-kappa} at first we relate the excess risk $S(\beta^0) - S(\beta)$ to $d_{\Delta}(\beta, \beta^0)$. Define for notational simplicity: 
$$W(\beta, \beta^0) =  \left\{x: \s(x^{\top}\beta) \neq \s(x^{\top}{\beta^0})\right\}$$ 
for $\beta \in S^{p-1}$. We have, 
\allowdisplaybreaks
\begin{align*}
S(\beta^0) - S(\beta) & =  2\int_{X_{\beta}} |E(Y|X)| f(x) \,\,dx \\
&=  4\int_{W(\beta, \beta^0)} |1-F_{\epsilon|X}(-x^\mathsf{T}\beta^0) - 0.5| f(x) \,\, dx \\
& =  4\int_{W(\beta, \beta^0)}|\eta(x) - 0.5| f(x) \,\, dx \\
& \ge  \,\,4 \sup_{0 \le t \le t^*} \left[t\,\P(|\eta(x) - 0.5| \ge t, W(\beta, \beta^0))\right] \\
& \ge \,\, 4 \sup_{0 \le t \le t^*} \left[t \left(d_{\Delta}(\beta, \beta^0) - \P_X\left(|\eta(x) - 0.5| \le t\right)\right)\right] \\
& \ge  \,\,4 \sup_{0 \le t \le t^*} \left[t \left(d_{\Delta}(\beta, \beta^0) - C_nt\right)\right] 
\end{align*}
A straightforward derivative calculation implies that the suprema is attained at $d_{\Delta}(\beta, \beta^0)/2C_n$ if $d_{\Delta}(\beta, \beta^0) < 2t^*C_n$ and at $t^*\left(d_{\Delta}(\beta, \beta^0) - C_nt^*\right)$ otherwise. Hence we conclude: 
\begin{align*}
S(\beta^0) - S(\beta) & \ge \left[\frac{d^2_{\Delta}(\beta, \beta^0)}{C_n}\mathds{1}_{\left(d_{\Delta}(\beta, \beta^0) \leq 2t^*C_n\right)} + 2t^*d_{\Delta}(\beta, \beta^0)\mathds{1}_{\left(d_{\Delta}(\beta, \beta^0) > 2t^*C_n\right)}\right] 
%& = \frac{d^2_{\Delta}(\beta, \beta^0)}{C_n} \wedge 2t^*d_{\Delta}(\beta, \beta^0) \\
%& \overset{\Delta} = d^2(\beta, \beta^0) \,.
\end{align*}
Combining this with Assumption \ref{ass:wedge_binary} we conclude: $$d^2(\beta, \beta^0) \ge \left[c_1^2\frac{\left\|\beta - \beta^0\right\|^2_2}{C_n}\mathds{1}_{\left(d_{\Delta}(\beta, \beta^0) \le 2t^*C_n\right)} + 2t^*c_1\left\|\beta - \beta^0\right\|_2\mathds{1}_{\left(d_{\Delta}(\beta, \beta^0) > 2t^*C_n\right)}\right]$$
which completes the proof. 

%If we maximize RHS with respect to $t$, the maximum is attained at $\hat{t} = \left(\frac{\P_X(X_{\beta})}{(\alpha + 1)C}\right)^{\frac{1}{\alpha}}$. This is a feasible value of $t$, i.e. $0 \le \hat{t} \le t^*$ if $P_X(X_{\beta}) \le C(\alpha + 1)(t^*)^{\alpha}$.  If the Assumption (A1) is valid for some $C$, then it true for all $C_1 > C$. Hence, without loss of generality we can assume $C \ge (\alpha + 1)(t^*)^{\alpha}$, i.e. $0 \le \hat{t} \le t^*$ is true always. Under this condition, putting the value of $\hat{t}$ we have for all $\beta \in S^{p-1}$: $$S(\beta^0) - S(\beta) \ge \frac{4\alpha}{C^{\frac{1}{\alpha}}(\alpha+1)^{\kappa}}(\P_X(X_{\beta}))^{\kappa} $$
%Using the relation between $\P_X(X_{\beta})$ and $\|\beta - \beta^0\|_2$ via Assumption (A2), we conclude the proof of the Proposition.

\subsection{Some sufficient conditions for Assumptions \ref{ass:wedge_binary} and \ref{ass:A2:upper}} 
In this subsection we provide some sufficient conditions for Assumption \ref{ass:wedge_binary} and \ref{ass:A2:upper}. We break the analysis into two lemmas. Lemma \ref{A2lemma} below exerts some sufficient conditions for Assumption \ref{ass:wedge_binary} and part (i) of Assumption \ref{ass:A2:upper}. Lemma \ref{Lipschitz-lemma} yields sufficient conditions for part (ii) of Assumption \ref{ass:A2:upper}.

\begin{lemma}
\label{A2lemma} 
Suppose that $X_{p \times 1}$ follows an elliptically symmetric distribution centered at 0, with density $f_X(x) = |\Sigma_p|^{-1/2} g(x^T \Sigma_p^{-1} x)$, where $g$ is a non-negative function. Assume that:
$$\inf_{p} \frac{\lambda_{min}(\Sigma_p)}{\lambda_{max}(\Sigma_p)} \ge c_{\lambda} > 0\,.$$
where $c_{\lambda}$ does not depend on $n,p$. Then $X$ satisfies Assumption \ref{ass:wedge_binary} and part (i) of Assumption \ref{ass:A2:upper}. 
\end{lemma}

\begin{proof} 
First, we prove that for $X \sim \mathcal{N}(0,\Sigma_p)$ with the above displayed condition holding. Observe that $\P_X(\s(X^\mathsf{T}\beta) \neq \s(X^\mathsf{T}\beta^0))$ depends on the two-dimensional geometry of $X$, i.e. only on the distribution of $(X^\mathsf{T}\beta, X^\mathsf{T}\beta^0)$. To make the calculations easier, we transform $X$ into $Y$ where the first two-coordinates of $Y$ corresponds to $(X^\mathsf{T}\beta, X^\mathsf{T}\beta^0)$. Consider the following orthogonal matrix:
\allowdisplaybreaks
\[
P = \begin{bmatrix}
\beta^{0'} \\
\frac{\beta' - \langle \beta^0,\beta \rangle \beta^{0'}} {\sqrt{1-\langle \beta^0,\beta \rangle^2}} \\
v_3 \\
\vdots \\
v_{p}
\end{bmatrix}
\]
where $\beta^{0'}, \frac{\beta' - \langle \beta^0,\beta \rangle \beta^{0'}} {\sqrt{1-\langle \beta^0,\beta \rangle^2}}, v_3, \dots , v_p$ forms an orthonormal basis of $\mathbb{R}^p$ (For example the vectors $v_3, \dots , v_p$ can be constructed using the Gram-Schimdt algorithm). If we define $Y = PX$, then $Y_1 = X^\mathsf{T}\beta^0$ and $X^\mathsf{T}\beta = a_1Y_1 + a_2Y_2$ where $a_1 = \langle \beta^0, \beta \rangle, a_2 = \sqrt{1-\langle \beta, \beta^0 \rangle^2}$. Then the probability of the wedge shaped region becomes:
\allowdisplaybreaks
\begin{align}
& \P(\s(X^\mathsf{T}\beta) \neq \s(X^\mathsf{T}\beta^0)) \notag\\
= & \,  \int_{\substack{x'\beta^0 \ge 0 \\ \beta'x < 0}} f_{X}(x) \,\, dx + \int_{\substack{\beta'x \ge 0 \\ x'\beta^0 < 0}} f_{X}(x) \,\, dx \notag \\
= & \,  \int_{\substack{y_1 \ge 0 \\ a_1y_1 + a_2y_2 < 0}} f_{X} (P^{-1}y) \,\, dy + \int_{\substack{y_1 < 0 \\ a_1y_1 + a_2y_2 \ge 0}} f_{X} (P^{-1}y) \,\, dy \notag \\
= & \,  \int_{\substack{y_1 \ge 0 \\ a_1y_1 + a_2y_2 < 0}} \frac{1}{\sqrt{|2\pi (P\Sigma P^\mathsf{T})|}}e^{-\frac{1}{2}y^\mathsf{T}(P\Sigma P^\mathsf{T})^{-1}y} \, dy+  \int_{\substack{y_1 < 0 \\ a_1y_1 + a_2y_2 <\ge 0}} \frac{1}{\sqrt{|2\pi (P\Sigma P^\mathsf{T})|}}e^{-\frac{1}{2}y^\mathsf{T}(P\Sigma P^\mathsf{T})^{-1}y} \, dy \notag \\
= & \,  \int_{\substack{y_1 \ge 0 \\ a_1y_1 + a_2y_2 < 0}} \frac{1}{\sqrt{|2\pi \tilde{\Sigma}|}}e^{-\frac{1}{2}y^\mathsf{T}\tilde{\Sigma}^{-1}y} \, dy_1 dy_2+  \int_{\substack{y_1 < 0 \\ a_1y_1 + a_2y_2 <\ge 0}} \frac{1}{\sqrt{|2\pi \tilde{\Sigma}|}}e^{-\frac{1}{2}y^\mathsf{T}\tilde{\Sigma}^{-1}y} \, dy_1 dy_2 \notag \\
 & \, \hspace*{1in} [\text{Mariginalise over} \,\, y_3, \dots, y_p, \text{with} \,\, \tilde{\Sigma} \,\, \text{being the leading} \,\,  2 \times 2 \,\, \text{block of} \,\, (P\Sigma P^\mathsf{T})] \notag \\
 \ge & \,\, \frac{1}{2\pi \sqrt{\lambda_1 \lambda_2}}\left[\int_{\substack{y_1 \ge 0 \\ a_1y_1 + a_2y_2 < 0}}e^{-\frac{1}{2\lambda_2}(y_1^2 + y_2^2)} \, dy_1 dy_2+  \int_{\substack{y_1 < 0 \\ a_1y_1 + a_2y_2 <\ge 0}} e^{-\frac{1}{2\lambda_2}(y_1^2 + y_2^2)} \, dy_1 dy_2\right] \notag \\
& \, \hspace*{3in} [\lambda_1 \ge \lambda_2, \text{are two eigenvalues of} \,\, \tilde{\Sigma}] \notag \\
\ge & \,\, \frac{1}{2\pi \sqrt{\lambda_1 \lambda_2}} \left[\int_{\frac{3\pi}{2}}^{\tan^{-1}(-a_1 / a_2)+2\pi} \int_0^{\infty} r e^{-\frac{1}{2\lambda_2}r^2} \,\, dr \,\, d\theta + \int_{\frac{\pi}{2}}^{\tan^{-1}(-a_1 / a_2)+\pi} \int_0^{\infty} r e^{-\frac{1}{2\lambda_2}r^2} \,\, dr \,\, d\theta\right] \hspace{0.2in} [\text{Polar}] \notag \\
\label{eqn1}
= & \,\,  \frac{1}{\pi} \sqrt{\frac{\lambda_2}{\lambda_1}} \left[\tan^{-1}(-a_1 / a_2)+ \pi/2\right]
\end{align}
\noindent
Now, for $||\beta - \beta^0|| = \delta, a_1 = \langle \beta, \beta^0 \rangle = 1 - \frac{\delta^2}{2}$. Hence we get, $$\frac{\lambda_1}{\lambda_2} = \frac{a_1}{\sqrt{1 - a_1^2}} = \frac{1 - \frac{\delta^2}{2}}{\sqrt{1-\bigg(1 - \frac{\delta^2}{2}\bigg)^2}} = \frac{1 - \frac{\delta^2}{2}}{\delta \sqrt{1 - \frac{\delta^2}{4}}}$$Using this we obtain,
\begin{equation*}
\begin{split}
&(\tan^{-1}(-a_1 / a_2)+\pi/2)
=   \left(\tan^{-1}\left[-\frac{1 - \frac{\delta^2}{2}}{\delta \sqrt{1 - \frac{\delta^2}{4}}}\right]+\frac{\pi}{2}\right)
\end{split}
\end{equation*}
It can be easily seen (i.e. by differentiating) that the function $\frac{\tan^{-1}\left[-\frac{1 - \frac{\delta^2}{2}}{\delta \sqrt{1 - \frac{\delta^2}{4}}}\right] + \frac{\pi}{2}}{\delta}$ is an increasing function of ${\delta}$ for $0 \le \delta \le 2$. More precisely, observing that 
$$\lim_{\delta \downarrow 0}\frac{\tan^{-1}\left[-\frac{1 - \frac{\delta^2}{2}}{\delta \sqrt{1 - \frac{\delta^2}{4}}}\right] + \frac{\pi}{2}}{\delta} = 1 $$ 
we conclude, for $0 \le \delta \le 2$: 
$$1 \le \frac{\tan^{-1}\left[-\frac{1 - \frac{\delta^2}{2}}{\delta \sqrt{1 - \frac{\delta^2}{4}}}\right] + \frac{\pi}{2}}{\delta} \le \frac{\pi}{2}\,.$$ 
In conjunction with $\ref{eqn1}$, this gives: 
$$\P_X(\s(X^\mathsf{T}\beta) \neq \s(X^\mathsf{T}\beta^0)) \ge \frac{1}{\pi} \sqrt{\frac{\lambda_2}{\lambda_1}} \, \|\beta - \beta^0\|_2 \,.$$ 
Finally using the fact that 
$$\lambda_{min}(P\Sigma P^\mathsf{T}) = \lambda_{min}(\Sigma) \le \lambda_2 \le \lambda_1 \le \lambda_{max}(P\Sigma P^\mathsf{T}) = \lambda_{max} (\Sigma) $$ 
we have $\sqrt{\frac{\lambda_2}{\lambda_1}} \ge \sqrt{c_{\lambda}}$. Combining these, we conclude: 
$$\P_X(\s(X^\mathsf{T}\beta) \neq \s(X^\mathsf{T}\beta^0)) \ge \sqrt{c_{\lambda}} \,  \|\beta - \beta^0\|_2 \, .$$ 
Now, on to general $X$. By our assumption on $X$ in the statement of the lemma, $X \sim \mathcal{E}(0, \Sigma_p)$ i.e. $X = \Sigma_p^{1/2} Y$ for some spherically symmetric random variable $Y$. We know $$Y \overset{d}= \frac{Z}{\|Z\|} g(\|Z\|)$$ for some $g: \mathbb{R}^+ \rightarrow \mathbb{R}^+$ with $Z \sim \mathcal{N}(0, I_p)$. Using the relation we have:
\allowdisplaybreaks
\begin{align*}
& \P_X(\s(X^\mathsf{T}\beta) \neq \s(X^\mathsf{T}\beta^0)) \\
= & \,  \P_X\left(\s\left(\frac{(\Sigma_p^{1/2}Z)^\mathsf{T}\beta}{\|Z\|} g(\|Z\|)\right) \neq \s\left(\frac{(\Sigma_p^{1/2}Z)^\mathsf{T}\beta^0}{\|Z\|} g(\|Z\|)\right)\right) \\
= & \,  \P_X\left(\s\left((\Sigma_p^{1/2}Z)^\mathsf{T}\beta\right) \neq \s\left((\Sigma_p^{1/2}Z)^\mathsf{T}\beta^0\right)\right)
\end{align*}
which again falls back to $\mathcal{N}(0, \Sigma_p)$ situation. The upper bound can be established via a similar calculation, where we need a finite upper bound on $\sup_{p} \frac{\lambda_{max}(\Sigma_p)}{\lambda_{min}(\Sigma_p)}$: this is given by $\frac{1}{c_{\lambda}}$.
\end{proof} 

\begin{lemma}
\label{Lipschitz-lemma} 
Assume that the function $\eta(x)$ satisfies that $$|\eta(x) - 1/2| = |F_{\epsilon|X=x}(-x^T\beta^0) - F_{\epsilon|X=x}(0)| \le k|x^T\beta^0|\,,$$ for some constant $k$ a.e. with respect to the measure of $X$ and the distribution of $X$ follows a consistent family of elliptical distribution with $f_X(x) = |\Sigma_p|^{-1/2}g_p(x^T\Sigma_p^{-1}x)$. Also assume that $g_2$ (the density component corresponding to the two dimensional marginal of $X$) is a decreasing function on $R$ and the eigenvalues of orientation matrix $\Sigma_p$ satisfies: $$0 < \lambda_- \le \lambda_{min}(\Sigma_p) \le \lambda_{max} (\Sigma_p) \le \lambda_+ < \infty$$ for all $p$. 
Then, under part (i) of Assumption \ref{ass:A2:upper} we have $$S(\beta^0) - S(\beta) \le u_+ \|\beta - \beta^0\|_2^2$$ for all $\beta \in S^{p-1}$ where $u_+ = 4\pi kk_1\frac{\lambda_+}{\sqrt{\lambda_-}}$ for some constant $k_1$ defined in the proof.
\end{lemma}

\begin{proof}
As in the proof of proposition \ref{alpha-kappa} we have (with the same notation):
\allowdisplaybreaks
\begin{align*}
& S(\beta^0) - S(\beta) \\
& = 4\int_{X_{\beta}} |\eta(x) - 1/2| f(x) \ dx \\
& \le 4k\int_{X_{\beta}} |x^T\beta^0| f(x) \ dx \\
& =  4k\left[\int_{\substack{y_1 \ge 0 \\ a_1y_1 + a_2y_2 < 0}} |y_1| \ f_{Y_1, Y_2} (y_1, y_2) \,\, dy_1 dy_2 + \int_{\substack{y_1 < 0 \\ a_1y_1 + a_2y_2 \ge 0}} |y_1| \ f_{Y_1, Y_2} (y_1, y_2) \,\, dy_1 dy_2\right] \\
& =  4k|\tilde{\Sigma}|^{-1/2}\left[\int_{\substack{y_1 \ge 0 \\ a_1y_1 + a_2y_2 < 0}} |y_1| \  g_2(y^T\tilde{\Sigma}^{-1} y) \,\, dy_1 dy_2 + \int_{\substack{y_1 < 0 \\ a_1y_1 + a_2y_2 \ge 0}} |y_1| \ g_2(y^T\tilde{\Sigma}^{-1} y) \,\, dy_1 dy_2\right] \\
& =  4k|\tilde{\Sigma}|^{-1/2}\left[\int_{\frac{3\pi}{2}}^{\tan^{-1}(-a_1 / a_2)+2\pi}\int_0^{\infty}  r^2|\cos(\theta)| \  g_2(r^2/\lambda_2) \,\, dy_1 dy_2  \right. \\
& \ \ \ \ \ \ \ \  + \left. \int_{\frac{\pi}{2}}^{\tan^{-1}(-a_1 / a_2)+\pi} \int_0^{\infty} r^2|\cos(\theta)| \ g_2(r^2/\lambda_2) \,\, dr \ d\theta\right] \hspace{0.2in} [\lambda_1 \le \lambda_2 \ \text{are the two eigenvalues of} \ \tilde{\Sigma}]\\
& \le 8kk_1\frac{\lambda_2}{\sqrt{\lambda_1}} \cos{(\tan^{-1}(a_1 / a_2))}(\tan^{-1}(a_1 / a_2) + \pi/2) \hspace{0.1in} \left[k_1 = \int_0^{\infty} r^2 g_2(r^2/\lambda_2) \ dr\right]\\
& \le 4\pi kk_1\frac{\lambda_2}{\sqrt{\lambda_1}}\|\beta - \beta^0\|^2_2 \le 4\pi kk_1\frac{\lambda_+}{\sqrt{\lambda_-}}\|\beta - \beta^0\|^2_2 
\end{align*}
\end{proof}\noindent
\begin{remark}
The Lipschitz type condition i.e. $|\eta(x) - 1/2|  \le k|x^T\beta^0|$ controls how the function varies around the true hyperplane. This condition is easily satisfied if we assume that the conditional density of $\epsilon$ given $X$ has an uniform upper bound over all $x$ and dimension. Note that the two conditions in the above lemma and Assumption \ref{ass:A2:upper} are readily satisfied, for example, for a broad class of elliptically symmetric densities centered at 0. 
\end{remark}

\subsection{Proof of Theorem \ref{rate-manski-p-less-n}}
In this proof, $K$ will denote a generic constant (not depending on $(n, C_n, p)$) which may change from line to line. We use Theorem \ref{massart} to establish the rate of convergence of the maximum score estimator. In our problem, the set of classifiers $$\mathcal{F} = \mathcal{F}_p = \{f_{\beta}: \R^p \rightarrow \{-1,1\}, \ f_{\beta}(X) =  \s(X^T{\beta}), \beta \in S^{p-1}\}\,,$$ and $\mathcal{Z} = \mathbb{R}^p \times \{-1,1\}$. We define the following affine transformations of our score functions: 
\begin{enumerate}
\item $\gamma(f_{\beta}, (X,Y)) = (1- Y\s(X^T\beta))/2$. 
\item $\gamma_n(f_{\beta}) =(1 -S_n(\beta))/2$.  
\item $P(\gamma(f_{\beta}, .)) = (1-S(\beta))/2$.
\item $\bar{\gamma}_n(f_{\beta}) = -(S_n(\beta) - S(\beta))/2$. 
\end{enumerate}
Also, note that $f^*$ in Theorem \ref{massart} is $f_{\beta^0}$ in our situation and the excess risk is $l(f_{\beta^0}, f_{\beta}) = (S(\beta^0) - S(\beta))/2$. Next we argue that the assumptions of Theorem \ref{massart} hold in our situation. 
For the first assumption, take $\mathcal{F} = F$ and take $F' = \{f_{\beta} \in \mathcal{F}: \beta \in S_1\}$ where $S_1$ is a countable dense subset of $S^{p-1}$. It is easy to check that the convergence criterion in condition (1) of 
Theorem \ref{massart} is satisfied on the set $\mathcal{X}_0 \times \{-1,1\}$ where $\mathcal{X}_0$ is the set of all $x$ such that $\beta^T x \ne 0$ for all $\beta \in S_1$. Since the random variable $X$ is continuous and $S_1$ is countable, $\mathcal{X}_0$ has probability 1, and this is sufficient for the conclusions of the theorem to hold. Also note that the collection $\mathcal{F}$ is VC class of functions with VC dimension $V \lesssim p$. 

We apply Theorem \ref{massart} with the distance metric $d_{\Delta}$. From Proposition \ref{alpha-kappa}: 
\begin{align}
\ell(f_{\beta^0}, f_{\beta}) & = \frac{S(\beta^0) - S(\beta)}{2} \notag\\
\label{eq:loss_lower_bound} & \ge \frac{1}{2}\left[\frac{d^2_{\Delta}(\beta, \beta^0)}{C_n}\mathds{1}_{\left(d_{\Delta}(\beta, \beta^0) \le 2t^*C_n\right)} + 2t^*d_{\Delta}(\beta, \beta^0)\mathds{1}_{\left(d_{\Delta}(\beta, \beta^0) > 2t^*C_n\right)}\right] 
\end{align}
Next, we construct a function $\omega$ which satisfies condition (2) of Theorem \ref{massart} with respect to the distance $\sqrt{d_{\Delta}}$. Note that we need $\omega$ to satisfy: 
$$\sqrt{d_{\Delta}(\beta, \beta^0)} \le \omega\left(\sqrt{\ell(f_{\beta^0}, f_{\beta})}\right)$$
or inverting it, 
$$\left(\omega^{-1}\left(\sqrt{d_{\Delta}(\beta, \beta^0)}\right)\right)^2 \le \ell(f_{\beta^0}, f_{\beta}) \,.$$
Hence, from Proposition \ref{alpha-kappa} we need $\omega$ to satisfy:
$$\left(\omega^{-1}\left(\sqrt{d_{\Delta}(\beta, \beta^0)}\right)\right)^2 = \frac{d^2_{\Delta}(\beta, \beta^0)}{2C_n}\mathds{1}_{\left(d_{\Delta}(\beta, \beta^0) \leq 2t^*C_n\right)} + t^*d_{\Delta}(\beta, \beta^0)\mathds{1}_{\left(d_{\Delta}(\beta, \beta^0) > 2t^*C_n\right)}$$ 
which further implies: 
$$\omega^{-1}\left(\sqrt{d_{\Delta}(\beta, \beta^0)}\right) = \frac{d_{\Delta}(\beta, \beta^0)}{\sqrt{C_n}}\mathds{1}_{\left(d_{\Delta}(\beta, \beta^0) \leq 2t^*C_n\right)} + \sqrt{2t^*d_{\Delta}(\beta, \beta^0)}\mathds{1}_{\left(d_{\Delta}(\beta, \beta^0) > 2t^*C_n\right)}$$
Parametrizing $\sqrt{d_{\Delta}(\beta, \beta^0)} = t$ we have: 
$$\omega^{-1}\left(t\right) = \frac{t^2}{\sqrt{C_n}}\mathds{1}_{\left(t < \sqrt{2t^*C_n}\right)} + \sqrt{2t^*}t\mathds{1}_{\left(t > \sqrt{2t^*C_n}\right)}$$
Hence inverting: 
\begin{align*}
\omega(x) & = \sqrt{x}C^{1/4}_{n}\mathds{1}_{\left(x < 2t^*\sqrt{C_n}\right)} + \frac{x}{\sqrt{2t^*}}\mathds{1}_{\left(x > 2t^*\sqrt{C_n}\right)} \\
& = \sqrt{x}C^{1/4}_{n} \vee \frac{x}{\sqrt{2t^*}}
\end{align*}
which immediately implies $\omega \in C_1$ as defined in Theorem \ref{massart}. 

It also follows that this pseudo-distance $\sqrt{d_{\Delta}}$ provides an upper bound on the variability of the difference between the loss functions at any two $\beta_1, \beta_2 \in S^{p-1}$:   
\allowdisplaybreaks
\begin{align*}
Var_P(\gamma(f_{\beta_1}, .) - \gamma(f_{\beta_2}, .)) & = \frac{1}{4}\,Var_P(Y\s(X^T\beta_1) - Y\s(X^T\beta_2)) \\
& \le\frac{1}{4}\, \E\left(\s(X^T\beta_1) - \s(X^T\beta_2)\right)^2 \\
& = \E(\mathds{1}(\s(X^T\beta_1) \neq \s(X^T\beta_2))) \\
& = \P(\s(X^T\beta_1) \neq \s(X^T\beta_2)) = d_{\Delta}^2(\beta_1, \beta_2)
\end{align*}

Finally we need to find $\phi$ which satisfies condition (3) of Theorem \ref{massart}. As $\mathcal{F}$ is a VC class of functions, we can follow the same line of argument in Section 2.4 of \cite{massart2006risk}: 
\begin{align*}
\phi(\sigma) & = K\sigma\sqrt{V\left(1 + \log{\left(\frac{1}{\sigma} \vee 1\right)}\right)} \\
& \le K \sigma\sqrt{p\log{\left(\frac{1}{\sigma}\right)}}
\end{align*} 
for all $\sigma \le 1$. The quantity $V$ in the above display is the VC-dimension of the class of all half-spaces in $\mathbb{R}^p$ where $V \lesssim p$.  Solving the equation $\sqrt{n}\epsilon_*^2 \ge \phi(\omega(\epsilon_*))$
we get:  
\begin{align*}
\epsilon_*^{-2}\phi(\omega(\epsilon_*)) & \le K \epsilon_*^{-2}\omega(\epsilon_*)\sqrt{p\log{\left(\frac{1}{\omega(\epsilon_*)}\right)}} \\
& = K \epsilon_*^{-2}\left(\sqrt{\epsilon_*}C^{1/4}_{n} \vee \frac{\epsilon_*}{\sqrt{2t^*}}\right)\sqrt{p\log{\left(\frac{1}{\sqrt{\epsilon_*}C^{1/4}_{n} \vee \frac{\epsilon_*}{\sqrt{2t^*}}}\right)}} \\
& \le K \epsilon_*^{-3/2}C^{1/4}_{n}\sqrt{p\log{\left(\frac{1}{C^{1/2}_{n}\epsilon_*}\right)}} \vee \frac{K}{\sqrt{2t^*}}\epsilon_*^{-1}\sqrt{p\log{\left(\frac{\sqrt{2t^*}}{\epsilon_*}\right)}}
\end{align*}
Hence we need to find $\epsilon_*$ such that: 
$$K \epsilon_*^{-3/2}C^{1/4}_{n}\sqrt{p\log{\left(\frac{1}{\sqrt{C_n}\epsilon_*}\right)}} \le \sqrt{n} \ \ \ \& \ \ \ \frac{K}{\sqrt{2t^*}}\epsilon_*^{-1}\sqrt{p\log{\left(\frac{\sqrt{2t^*}}{\epsilon_*}\right)}} \le \sqrt{n} \,.$$
Solving these two inequalities and ignoring constants we get: 
$$\epsilon_* = \left(\frac{p\sqrt{C_n}\log{(n/pC^2_{n})}}{n}\right)^{1/3} \vee \left(\frac{p\log{(n/p)}}{n}\right)^{1/2}$$
Using the above $\epsilon_*$ we conclude using Theorem \ref{massart}: 
\begin{equation}
\label{ineq_1_bound}
\P\left(S(\beta^0) - S(\hat{\beta}) \ge K y\epsilon_*^2\right) \le e^{-y}
\end{equation}
for all $y \ge 1$. Here also $K$ is a different constant than before, which is now a function of some universal constant and $t^*$, but it does not depend on $(n, p, s_0)$. Using Proposition \ref{alpha-kappa} and equation \eqref{ineq_1_bound}, we get the following concentration bound:
\begin{equation*}
\P\left(\epsilon_*^{-2} \left[\frac{d^2_{\Delta}(\hat \beta, \beta^0)}{C_n}\mathds{1}_{\left(d_{\Delta}(\hat \beta, \beta^0) \le 2t^*C_n\right)} + 2t^*d_{\Delta}(\hat \beta, \beta^0)\mathds{1}_{\left(d_{\Delta}(\hat \beta, \beta^0) > 2t^*C_n\right)}\right]  \ge Ky\right) \le e^{-y}
\end{equation*}
Consequently: 
\begin{align*}
\P\left(\epsilon_*^{-2} \left[\frac{d^2_{\Delta}(\hat \beta, \beta^0)}{C_n}\mathds{1}_{\left(d_{\Delta}(\hat \beta, \beta^0) \le 2t^*C_n\right)} \right]  \ge Ky\right) \le e^{-y} 
\end{align*}
which, along with Assumption \ref{ass:wedge_binary} yields: 
\begin{align}
\label{subbound_1}  \P\left(\frac{\epsilon_*^{-1}}{\sqrt{C_n}}\left[\|\hat \beta - \beta^0\|_2 \mathds{1}_{\left(d_{\Delta}(\hat \beta, \beta^0) \le 2t^*C_n\right)} \right]  \ge Ky\right) \le e^{-y^2} 
\end{align}
and 
\begin{align*}
\P\left(\epsilon_*^{-2} \left[2t^*d_{\Delta}(\hat \beta, \beta^0)\mathds{1}_{\left(d_{\Delta}(\hat \beta, \beta^0) > 2t^*C_n\right)}\right]  \ge Ky\right) \le e^{-y} 
\end{align*}
which can be rewritten using Assumption \ref{ass:wedge_binary} as:
\begin{align}
\label{subbound_2} \P\left(\epsilon_*^{-2} \left[\|\hat \beta - \beta^0\|_2\mathds{1}_{\left(d_{\Delta}(\hat \beta, \beta^0) > 2t^*C_n\right)}\right]  \ge Ky\right) \le e^{-y} \,.
\end{align} 
Combining equation \eqref{subbound_1} and \eqref{subbound_2} we conclude: 
\begin{equation*}
\P\left(\left(\frac{r_n}{\sqrt{C_n}} \wedge r_n^2 \right) \|\hat \beta - \beta^0\|_2 \ge Ky\right) \le e^{-y^2} + e^{-y} \le 2e^{-y}
\end{equation*}
for all $y \ge 1$ and for some constant $K$ not depending on $n,p$ with $r_n = \epsilon_*^{-1}$. 
which completes the proof of the concentration bound.

The upper bound on the expectation follows from this exponential tail bound using the following calculation:
\allowdisplaybreaks
\begin{align*}
& \E\left(\left(\frac{r_n}{\sqrt{C_n}} \wedge r_n^2 \right) \|\hat{\beta} - \beta^0\|_2\right) \\
& = \int_{0}^{\infty} \P\left(\left(\frac{r_n}{\sqrt{C_n}} \wedge r_n^2 \right)\|\hat{\beta} - \beta^0\|_2 \ge T\right) \, dT \\
& = \int_{0}^{1} \P\left(\left(\frac{r_n}{\sqrt{C_n}} \wedge r_n^2 \right)\|\hat{\beta} - \beta^0\|_2 \ge T\right) \, dT + \int_{1}^{\infty} \P\left(\left(\frac{r_n}{\sqrt{C_n}} \wedge r_n^2 \right)\|\hat{\beta} - \beta^0\|_2 \ge T\right) \, dT \\
& \le 1 + 2\int_{1}^{\infty} e^{-T/K} \, dT \\
& = 1 + 2K \int_{1/K}^{\infty} e^{-T} \, dT = 1 + 2Ke^{-1/K} < \infty
\end{align*}
which completes the proof of minimax upper bound.

\subsection{Proof of Theorem \ref{minimax-lower-bound}}
To obtain a lower bound on the minimax error, we use Assouad's Lemma \cite{assouad1983deux} which we state below for convenience:
\begin{lemma}{\bf [Assouad's Lemma]}
\label{lem:Assouad}
Let $\Omega = \{0,1\}^m$ (or $\{-1,1\}^m$) be the set of all binary sequences of length $m$. Let $P_{\omega} , \omega \in \Omega$ be a set of $2^m$ measures on some space $\{\mathcal{X}, A\}$ and let the corresponding expectations be $\E_{\omega}$. Then: $$\inf_{\hat{\omega}}\sup_{\omega \in \Omega} \E_{\omega}(d_H(\hat{\omega}, \omega)) \ge \frac{m}{2}(1 - \max_{\omega \sim \omega'}\|P^n_{\omega} - P^n_{\omega'}\|_{TV})$$ where $\hat{\omega}$ is an estimator based on $n$ i.i.d. observations $z_1, \dots, z_n \sim P_{\omega}$, $P^n_{\omega}$ denotes the $n$-fold product measure of $P_{\omega}$, $d_H$ is the Hamming distance and $\omega \sim \omega'$ means $d_H(\omega , \omega') = 1$.\footnote{For some discussions and applications of this lemma, see \cite{tsybakov2009introduction}.} 
\end{lemma}
\noindent
To apply this lemma in our model, define for small $\epsilon > 0$: $$\tilde{\Theta} = \left\{\beta: \,\,\, \text{where} \,\,\, \beta = \frac{\gamma}{\|\gamma\|_2}, \gamma_1 = 1, \gamma_j \in \{-\epsilon, \epsilon\}, \forall \,\,\, 2\le j \le p\right\} \,.$$ We will motivate the choice of $\epsilon$ in the later part of the proof. Observe that, $\|\gamma\|_2$ is same for all $\gamma \in \Theta$ and equals $\sqrt{1+(p-1)\epsilon^2}$. For notational simplicity, define $m(\epsilon) = \sqrt{1+(p-1)\epsilon^2}$. Now, for any $\omega \in \{-1,1\}^{p-1}$, define $\gamma_{\omega} = (1, \epsilon \omega)$ and $\beta_{\omega} = \gamma_{\omega}/ \|\gamma_{\omega}\|_2$. This establishes a 1-1 correspondence between $\Omega$ and $\tilde{\Theta}$, with $m = p-1$.  For any $\beta \in \tilde{\Theta}$ define the joint distribution $P_{\beta}$ of $(X,Y)$ as:
\begin{enumerate}
\item $X \sim \mathcal{N}(0, I_p)$
\item $P_{\beta}(Y = 1| X) =
\begin{cases}
  \frac{1}{2} + \frac{1}{C_n}\beta'X, & \text{if } |\beta'X| \le \left[C_n\epsilon\sqrt{p} \vee \frac{|X_1|}{2 m(\epsilon)} \right ]\wedge 1/4. \\
  \frac{1}{2} + \left(\left[\epsilon\sqrt{p} \vee \frac{|X_1|}{2 C_nm(\epsilon)} \right ]\wedge 1/4\right)\text{sgn}(\beta'X), & \text{otherwise}.
\end{cases}$
\end{enumerate}

\noindent
The Gaussian distribution of $X$ trivially satisfies Assumption (A2). In the following lemma we show that this construction also satisfies Assumption (A1).

\noindent
From now on, we define $C = C_n$ for notational simplicity. 
\begin{lemma}
\label{low_noise_minimax}
The above construction of satisfies a part of Assumption 1, i.e. 
$$\P\left(|\eta(X) - 1/2| \le t \right) \le 5\sqrt{\frac{2}{\pi}}Ct \ \ \forall \ t \in (0, 1/4) \,.$$
\end{lemma}

\begin{proof}
Fix $t$ such that $0 < t < \frac{1}{4}$. Then,
\allowdisplaybreaks
\begin{align*}
& \P_X(|\eta(X) - 0.5| \le t) \\
& =  \P_X\left(|\eta(X) - 0.5| \le t, \frac{|\beta'X|}{C} \le \left[\epsilon\sqrt{p} \vee \frac{|X_1|}{2Cm(\epsilon)} \right ]\wedge 1/4\right) \\ & \qquad \qquad+ \P_X\left(|\eta(X) - 0.5| \le t, \frac{|\beta'X|}{C}> \left[\epsilon\sqrt{p} \vee \frac{|X_1|}{2Cm(\epsilon)} \right ]\wedge 1/4\right) \\
& =   \P_X\left(\frac{|\beta'X|}{C} \le t, \frac{|\beta'X|}{C} \le \left[\epsilon\sqrt{p} \vee \frac{|X_1|}{2Cm(\epsilon)} \right ]\wedge 1/4\right) \\ & \qquad \qquad + \P_X\left(\left[\epsilon\sqrt{p} \vee \frac{|X_1|}{2Cm(\epsilon)} \right ]\wedge 1/4 \le t, \frac{|\beta'X|}{C} > \left[\epsilon\sqrt{p} \vee \frac{|X_1|}{2Cm(\epsilon)} \right ]\wedge 1/4\right) \\
& \le \P_X(|\beta^TX| \le Ct) + \P_X\left(\left[\epsilon\sqrt{p} \vee \frac{|X_1|}{2Cm(\epsilon)} \right ]\wedge 1/4 \le t\right) \\
& \le \P_X(|\beta^TX| \le Ct) + \P_X(|X_1| \le 2Cm(\epsilon)t) \\
& \le \sqrt{\frac{2}{\pi}}\left[Ct + 2Cm(\epsilon)t\right] \le 5\sqrt{\frac{2}{\pi}}Ct
\end{align*}
The last inequality is valid when $m(\epsilon) \le 2$, which happens for $\epsilon \sqrt{p}$ sufficiently small. 
\end{proof}

\noindent
We use the notation $\beta \sim_j \beta'$ if $\beta$ and $\beta'$ differs only in $j^{th}$ position for $2 \le j \le p$. So, in order use Assouad's lemma, we need an  on $\|P^n_{\beta}-P^n_{\beta'}\|_{TV}$ when $\beta \sim_j \beta'$ for any $2 \le j \le p$.  Fix $\beta_1$ and $\beta_2$ and $j \in \{2, \cdots, p\}$ such that $\beta_1 \sim_j \beta_2$. Using the standard relation between the total variation norm and Hellinger distance, we have:  
$$\|P^n_{\beta_1}-P^n_{\beta_2}\|_{TV} \le \sqrt{H^2(P^n_{\beta_1}, P^n_{\beta_2})} \le \sqrt{n H^2(P_{\beta_1}, P_{\beta_2})}$$ 
To make the minimax lower bound non-trivial, we will choose $\epsilon = \epsilon(n,p)$ in a way that ensures $H^2(P_{\beta_1}, P_{\beta_2}) \sim n^{-1}$. Towards that, we need the following lemma:

\begin{lemma}
\label{Hellinger}
If $P_1$ = Ber($p_1$) and $P_2$ = Ber($p_2$) with $p_1, p_2 \in [1/4, 3/4]$, then $H^2(P_1, P_2) \le \frac{\nu^2}{4\sqrt{3}s(1-s)}$ where $\nu = p_2 - p_1, s = (p_1+p_2)/2$.
\end{lemma}

\noindent
The proof of this Lemma can be found in section \ref{Hellinger Proof} of supplement. For the rest of the proof,  define 
$$A_i = \left[|\beta_i'X| \le \left\{C\epsilon\sqrt{p} \vee \frac{|X_1|}{2 m(\epsilon)} \right\}\wedge 1/4 \right ]$$ for $i=1,2$. 
Now,
\begin{equation*}
\begin{split}
& H^2(P_{\beta_1}, P_{\beta_2}) \\
= & E_X\left[H^2(P_{\beta_1}(Y=1|X), P_{\beta_2}(Y=1|X))\right] \\
= & E_X\left[H^2(P_{\beta_1, X}, P_{\beta_2, X})\right]  \hspace{0.3in} [\text{Say}]
\end{split}
\end{equation*}
We next divide the domain of $X$ into two sub-parts and compute the corresponding values of $\nu_X = P_{\beta_1}(Y=1|X) - P_{\beta_2}(Y=1|X)$, on these sub-parts.
\\\\
\textbf{Case 1:} $X \in A_1 \cup A_2$.
\\\\
\textbf{Case 2: } $X \in A_1^c \cap A_2^c$. Note that, in this case, $|\nu_X| = 0$, if $\text{sign}(\beta_1'X) = \text{sign}(\beta_2'X)$, $|\nu_X| \le 2\left(\epsilon\sqrt{p} \vee \frac{|X_1|}{2Cm(\epsilon)}\right)$ otherwise.
\begin{lemma}
\label{intermediate-lemma}
Under \textbf{Case 1}, $|\nu_X| = \left|P_{\beta_1}(Y=1|X) - P_{\beta_2}(Y=1|X)\right| \le 2\epsilon|X_j|/Cm(\epsilon)$ where $\beta_1 \sim_j\beta_2$.
\end{lemma}

\begin{proof}
First assume that, $X \in A_1 \cap A_2$. Then,
\begin{equation*}
\begin{split}
|\nu_X| = \left|P_{\beta_1}(Y=1|X) - P_{\beta_2}(Y=1|X)\right| & =  \frac{|\beta_1' X - \beta_2' X|}{C} = \frac{2\epsilon |X_j|}{Cm(\epsilon)}
\end{split}
\end{equation*}
Next, consider the case that $X \in A_1 \cap A_2^c$. Then, $|\beta_2'X| > \left(C\epsilon\sqrt{p} \vee \frac{|X_1|}{2 m(\epsilon)}\right) \wedge 1/4$ but $|\beta_1'X| < \left(C\epsilon\sqrt{p} \vee \frac{|X_1|}{2 m(\epsilon)}\right) \wedge 1/4$. Hence,
\begin{equation*}
\begin{split}
|\nu_X| = \left|P_{\beta_1}(Y=1|X) - P_{\beta_2}(Y=1|X)\right| \le \frac{|\beta_1'X - \beta_2'X|}{C} = \frac{2\epsilon |X_j|}{Cm(\epsilon)}\,.
\end{split}
\end{equation*}
The case when $A_1^c\cap A_2$ follows in the exact same manner, by symmetry. 
\end{proof}
We are now in a position to tackle $H^2(P_{\beta_1}, P_{\beta_2})$ as shown below. 
\begin{align}
& H^2(P_{\beta_1}, P_{\beta_2}) \\
& \E_X\left[H^2(P_{\beta_1, X}, P_{\beta_2, X})\right] \notag \\
& \le \frac{1}{4\sqrt{3}}\E\left[\frac{\nu_X^2}{s(1-s)}\right] \notag \\
& = \frac{1}{4\sqrt{3}}\E\left[\frac{\nu_X^2}{s(1-s)}\mathds{1}_{X \in A_1 \cup A_2} + \frac{\nu_X^2}{s(1-s)}\mathds{1}_{X \in A_1^c \cap A_2^c}\right] \notag \\
& \le \frac{4}{3\sqrt{3}} \E\left[\nu_X^2 \mathds{1}_{X \in A_1 \cup A_2} + \nu_X^2\mathds{1}_{X \in A_1^c \cap A_2^c}\right] \hspace*{1in}[\because \frac{1}{4} \le s \le \frac{3}{4}] \notag \\
& \le \frac{16}{3\sqrt{3}} \E\left[\frac{\epsilon^2 X_j^2}{C^2m(\epsilon)^2} \mathds{1}_{X \in A_1 \cup A_2} +  4\left(\epsilon\sqrt{p} \vee \frac{|X_1|}{2Cm(\epsilon)}\right)^2\mathds{1}_{X \in A_1^c \cap A_2^c, \,\,\, \text{sign}(\beta_1'X) \neq \text{sign}(\beta_2'X)}\right] \notag\\
& = \frac{16}{3\sqrt{3}}\E[T_1 + 4T_2] \hspace*{0.3in} [\text{Say}] \,.\label{eq1-mlb-sg}
\end{align}
We will analyze the expectation of each summand separately. Define $\tilde{\beta} \overset{\Delta} = \beta_{[2:p]}$ i.e. $\tilde{\beta}$ is a vector of dimension $(p-1)$ which we obtain by removing the first co-ordinate of $\beta$, and let $\tilde{X}$ be defined similarly in terms of $X$. We have:  
\begin{align}
\E(T_1) = & \frac{1}{C^2m(\epsilon)^2}\E\left(\epsilon^2X_j^2 \mathds{1}_{X \in A_1 \cup A_2}\right) 
\notag \\
\le & \frac{1}{C^2m(\epsilon)^2}\left[\E\left(\epsilon^2X_j^2 \mathds{1}_{X \in A_1}\right) + \E\left(\epsilon^2X_j^2 \mathds{1}_{X \in A_2}\right)\right]\notag \\
= & \frac{2}{C^2m(\epsilon)^2}\E\left(\epsilon^2X_j^2 \mathds{1}_{X \in A_1}\right) \hspace*{0.3in} [\because \text{both terms are identically distributed}] 
\notag \\
\le & \frac{2}{C^2m(\epsilon)^2}\E\left(\epsilon^2X_j^2\mathds{1}_{|\beta_1'X| \le C\epsilon\sqrt{p} \vee \frac{|X_1|}{2 m(\epsilon)}} \right) 
\notag \\
\le & \frac{2}{C^2m(\epsilon)^2}\left[\epsilon^2 \E\left(X_j^2\mathds{1}_{|\beta_1'X| \le \frac{|X_1|}{2 m(\epsilon)}}\right) + \epsilon^2\E\left(X_j^2\mathds{1}_{\frac{|X_1|}{2Cm(\epsilon)} \le \epsilon\sqrt{p}}\right)\right] 
\notag \\
\le & \frac{2}{C^2m(\epsilon)^2}\left[\epsilon^2 \E\left(X_j^2\mathds{1}_{\frac{|X_1|}{m(\epsilon)} - |\tilde{\beta_1}'\tilde{X}| \le \frac{|X_1|}{2 m(\epsilon)}}\right) + \epsilon^2\E\left(X_j^2\mathds{1}_{\frac{|X_1|}{2m(\epsilon)} \le C\epsilon\sqrt{p}}\right)\right] \hspace*{0.2in} [\because |a+b| \ge |a| - |b|]
\notag\\
\le & \frac{2\epsilon^2}{C^2m(\epsilon)^2} \left[\E\left(X_j^2 \mathds{1}_{\frac{|X_1|}{2m(\epsilon)} \le |\tilde{\beta_1}'\tilde{X}|}\right) + \E(X_j^2)P\left(|X_1| \le 2Cm(\epsilon)\epsilon\sqrt{p}\right)\right] \notag\\
\le & \frac{2\epsilon^2}{C^2m(\epsilon)^2} \E_{\tilde{X}}\left(X_j^2\E_{X_1}\left(\mathds{1}_{|X_1| \le 2m(\epsilon)|\tilde{\beta_1}'\tilde{X}|} \,\,\, | \,\,\, \tilde{X}\right)\right) + 4\sqrt{\frac{2}{\pi}}\frac{\epsilon^3\sqrt{p}}{Cm(\epsilon)}  
\notag\\
\le & \frac{2\epsilon^2}{C^2m(\epsilon)^2} \E_{\tilde{X}}\left(X_j^2P\left(|X_1| \le 2m(\epsilon)|\tilde{\beta_1}'\tilde{X}| \,\,\, \mid \,\,\, \tilde{X}\right)\right) + 4\sqrt{\frac{2}{\pi}}\frac{\epsilon^3\sqrt{p}}{Cm(\epsilon)^2} 
\notag\\
\le & \sqrt{\frac{8}{\pi}}\left[\frac{\epsilon^2}{C^2 m(\epsilon)} \E_{\tilde{X}}\left(X_j^2|\tilde{\beta}'\tilde{X}|\right) + \frac{2\epsilon^3\sqrt{p}}{C^2m(\epsilon)}\right] \hspace*{1in} [ \text{as} \,\,X_1 \& \tilde{X} \,\,\,\, \text{are independent}]
\notag\\
\le & 4\sqrt{\frac{2}{\pi}}\left[\frac{\epsilon^2}{C^2 m(\epsilon)} (\E(X_j^4))^{\frac{1}{2}}(\E((\tilde{\beta_1}'\tilde{X})^2))^{\frac{1}{2}} + \frac{\epsilon^3\sqrt{p}}{C^2m(\epsilon)}\right]  
\notag\\
\le & 4\sqrt{\frac{6}{\pi}}\frac{\epsilon^2}{C^2m(\epsilon)}(\E((\tilde{\beta_1}'\tilde{X})^2))^{\frac{1}{2}} + 4\sqrt{\frac{2}{\pi}}\frac{\epsilon^3\sqrt{p}}{C^2m(\epsilon)}  \notag\\
\le & 4\sqrt{\frac{2}{\pi}}(1+\sqrt{3})\frac{\epsilon^3\sqrt{p}}{C^2}\left[\frac{1}{m(\epsilon)^2} + \frac{1}{m(\epsilon)}\right] \notag \\
\le & 8\sqrt{\frac{2}{\pi}}(1+\sqrt{3})\frac{\epsilon^3\sqrt{p}}{C^2} \hspace{0.2in} [\because m(\epsilon) \ge 1] \label{eq2-mlb-sg}
\end{align}
For the second part, observe that,
\allowdisplaybreaks
\begin{align}
\left\{X \in A_1^c \cap A_2^c \,\,\, \text{and} \,\,\, \text{sign}(\beta_1' X) \neq \text{sign}(\beta_2' X)\right\} \Rightarrow & |\beta_1'X - \beta_2'X| \ge 2C\epsilon\sqrt{p}\notag \\
\Rightarrow & \frac{2\epsilon|X_j|}{m(\epsilon)} \ge 2C\epsilon\sqrt{p} \notag \\
\Rightarrow & |X_j| \ge Cm(\epsilon)\sqrt{p} \notag
\end{align}
Using this observation, we get,
\allowdisplaybreaks
\begin{align}
E(T_2) \le & \left(E\left(\left(\epsilon\sqrt{p} \vee \frac{|X_1|}{2Cm(\epsilon)}\right)\wedge 1/4\right)^4\right)^{\frac{1}{2}}\left(E\left(\mathds{1}_{X \in A_1^c \cap A_2^c, \text{sign}(\beta_1'X) \neq \text{sign}(\beta_2'X)}\right)\right)^{\frac{1}{2}} \notag \\
\le & K\left(E\left(\mathds{1}_{X \in A_1^c \cap A_2^c, \text{sign}(\beta_1'X) \neq \text{sign}(\beta_2'X)}\right)\right)^{\frac{1}{2}} \notag \\
\le & K\left(E\left(\mathds{1}_{|X_j| \ge  Cm(\epsilon)\sqrt{p}}\right)\right)^{\frac{1}{2}} \notag \\
\le & Ke^{-\frac{C^2m(\epsilon)^2p}{4}} \label{eq3-mlb-sg}\,,
\end{align}
where $K$ is an absolute constant. Putting together \ref{eq1-mlb-sg}, \ref{eq2-mlb-sg} and \ref{eq3-mlb-sg}, we get, 
$$H^2(P_{\beta_1}, P_{\beta_2}) = E_{X}\left(H^2(P_{\beta_1, X}, P_{\beta_2, X})\right)  \le \frac{16}{3\sqrt{3}}\left[8\sqrt{\frac{2}{\pi}}(1+\sqrt{3})\frac{\epsilon^3\sqrt{p}}{C^2} + 4Ke^{-\frac{C^2m(\epsilon)^2p}{4}}\right]\,.$$  
Set $\zeta = \frac{128}{3\sqrt{3}}\sqrt{\frac{2}{\pi}}(1+\sqrt{3})$. If we choose $\epsilon = \left(\frac{1}{2\zeta}\right)^{\frac{1}{3}}n^{-\frac{1}{3}}p^{-\frac{1}{6}}C^{\frac{2}{3}}$, then $E_{X}\left(H^2(P_{\beta_1}, P_{\beta_2})\right) \le \frac{1}{2n} + \frac{64K}{3\sqrt{3}}e^{-\frac{C^2m(\epsilon)^2p}{4}}$. So we have 
$$\sqrt{nH^2(P_{\beta_1}, P_{\beta_2})} \le \sqrt{\frac{1}{2} + \frac{64K}{3\sqrt{3}}ne^{-\frac{C^2m(\epsilon)^2p}{4}}} \le \sqrt{\frac{2}{3}}$$
for all large $n$, as $ne^{-\frac{C^2m(\epsilon)^2p}{4}} \rightarrow 0$. Now we can relate Hamming distance to $\ell_2$ distance via $$\|\hat{\beta} - \beta^0\|_2^2 = \frac{\epsilon^2}{m(\epsilon)^2}d_H(\hat{\beta}, \beta^0)$$ and use Assouad's lemma to deduce: 
\allowdisplaybreaks
\begin{align*}
\inf_{\hat{\beta}_n \in \tilde{\Theta}}\sup_{P_{\beta}: \beta \in \tilde{\Theta}}\mathbb{E}_{\beta}\left(\|\hat{\beta}_n - \beta\|_2^2\right) & \ge \frac{\epsilon^2(p-1)}{2m(\epsilon)^2}(1-\max_{\beta \sim \beta'}\|P_{\beta}^n - P_{\beta'}^n\|_{TV}) \\
& \ge \frac{\epsilon^2(p-1)}{4}(1-\sqrt{2/3}) \hspace{0.3in}[\because m(\epsilon) \rightarrow 1 \ \text{as} \ p \rightarrow \infty]\\
& \ge \tilde{K_L}\left(\frac{pC^2}{n}\right)^{2/3} \hspace*{0.2in} [\because \epsilon = \left(\frac{1}{2\zeta}\right)^{\frac{1}{3}}n^{-\frac{1}{3}}p^{-\frac{1}{6}}C^{\frac{2}{3}}]
\end{align*}
for some constant $\tilde{K}_L$. Finally, let $\hat{\beta}$ be any estimator assuming values in $S^{p-1}$. Define $\tilde{\beta}$ to be the projection of $\hat{\beta}$ on the hypercube, i.e. $$\tilde{\beta} = \argmin_{\beta \in \tilde{\Theta}} \|\hat{\beta} - \beta\|_2$$ Then for any $\beta \in \tilde{\theta}$ we have: 
\allowdisplaybreaks
\begin{align*}
\|\tilde{\beta} - \beta\|_2 & \le \|\tilde{\beta} - \hat{\beta} + \hat{\beta} - \beta\|_2 \le 2\| \beta - \hat{\beta}\|_2 
\end{align*}

Using this relation we can conclude that: 
\allowdisplaybreaks
\begin{align*}
\inf_{\hat{\beta}_n}\sup_{P_{\beta}}\mathbb{E}_{\beta}\left(\|\hat{\beta}_n - \beta\|_2^2\right) & \ge \frac{1}{4}\inf_{\hat{\beta}_n \in \tilde{\theta}}\sup_{P_{\beta}: \beta \in \tilde{\theta}}\mathbb{E}_{\beta}\left(\|\hat{\beta}_n - \beta\|_2^2\right) \\
& \ge K_L \left(\frac{pC^2}{n}\right)^{2/3}\,,
\end{align*}
where $K_L = \tilde{K}_L/4$. To prove that the minimax rate cannot be improved upon $(p/n)$, one can resort to a minimax construction taking $C = 0$. The construction and the rest of the proof follow a similar pattern as above, and are skipped for the sake of brevity. 

%\begin{remark}
%The proof of part (i) (when $\alpha = 1$) is exactly the same with $\xi = 1$. Note that, when $\alpha = 1$, $\xi = 1$. Hence it follows immediately from Lemma \ref{low_noise_minimax} that our construction satisfies Assumption \ref{ass:low_noise_binary} on $(0, t^*)$ with $t^{\star} = 1/4$. The only difference between the parts of the Theorem is that, when $\xi > 1$, Assumption \ref{ass:low_noise_binary} is only satisfied on a sub-interval $(0,1/4)$, namely $(c_n, 1/4)$. 
%\end{remark}

\subsection{Proof of Theorem \ref{high-dim-rate}}
This proof is based on Theorem \ref{our_model_selection}. Recall that $\mathscr{M}_i$  is the collection of all models with  $\|\beta\|_0 \le i$ for $1 \le i \le L = \lfloor n/4\log{p} \rfloor$. As mentioned previously, the model $\mathscr{M}_i$ has VC dimension $V_i \asymp i\log{(ep/i)}$. Following the same line of argument as in the proof of Theorem \ref{rate-manski-p-less-n} we can conclude: 
$$\omega(x) = \sqrt{x}C^{1/4}_{n} \vee \frac{x}{\sqrt{2t^*}}$$
and the values of $\epsilon_i$ can be taken as: 
$$\epsilon_i^2 = \left(\frac{V_i\sqrt{C_n}\log{(n/V_iC^2_{n})}}{n}\right)^{2/3} \vee \left(\frac{V_i\log{(n/V_i)}}{n}\right) \,.$$
We know the function $f(x) = x\log{(t/x)}$ increases between $(0,t/e)$ and then decreases. As $L < p$, the sequence of VC dimensions is increasing: $V_1 \le V_2 \le \cdots \le V_L$. Next, we establish that $V_L \le n/e$ for all large $n$. Assume to the contrary that $V_L > n/e$. Then: 
\allowdisplaybreaks
\begin{align*}
    V_{L} > n/e & \Rightarrow \frac{1}{4\log{p}}\log{\left(\frac{2ep\log{p}}{n}\right)} >  1/e\\
    & \Rightarrow \frac{1+ \log{p} + \log{\log{p}} - \log{n}}{\log{p}} > 4/e 
\end{align*}
which is a contradiction since the LHS goes to 1 as $n\rightarrow \infty$. This immediately implies that: 
$$\left(\frac{V_1\sqrt{C_n}\log{(n/V_1C^2_{n})}}{n}\right)^{2/3} \le \dots \le \left(\frac{V_L\sqrt{C_n}\log{(n/V_LC^2_{n})}}{n}\right)^{2/3}$$
and 
$$\left(\frac{V_1\log{(n/V_1)}}{n}\right) \le \dots \le \left(\frac{V_L\log{(n/V_L)}}{n}\right) \,.$$
This proves that 
$$\min_{1 \le i \le L}\epsilon_i^2 = \epsilon_1^2 =  \left(\frac{V_1\sqrt{C_n}\log{(n/V_1C^2_{n})}}{n}\right)^{2/3} \vee \left(\frac{V_1\log{(n/V_1)}}{n}\right)$$ 
Hence in our case, 
\begin{align*}
    b(n) = \frac{\omega^2(\epsilon_1)}{\epsilon_1^2} = \frac{C^{1/2}_{n}}{\epsilon_1} \vee \frac{1}{2t^*} = \frac{C^{1/2}_{n}}{\left(\frac{V_1\sqrt{C_n}\log{(n/V_1C^2_{n})}}{n}\right)^{1/3} \vee \sqrt{\left(\frac{V_1\log{(n/V_1)}}{n}\right)}} \vee  \frac{1}{2t^*} \\
\end{align*}
Now we need to choose a penalty function such that: 
$$\pen(\mathscr{M}_i) \ge K\left(\epsilon_i^2 + \frac{\omega^2(\epsilon_i)x_i}{n\epsilon_i^2}\right) \,.$$
If we choose 
$$x_i = \frac{n\epsilon_i^4}{\omega^2(\epsilon_i)} = \frac{n\epsilon_i^3}{\sqrt{C_n} \vee \frac{\epsilon_i}{2t^*}}$$
then a permissible penalty function is given by $\text{Pen}(\mathscr{M}_i) = 2K\epsilon_i^2$, provided we can show that $\sum_{i=1}^{L} e^{-x_i} < \infty$. 
Towards that end:
\begin{align*}
\sum_{i=0}^{L} e^{-x_i} & = \sum_{i=0}^{L} e^{-\frac{n\epsilon_i^3}{\sqrt{C_n} \vee \frac{\epsilon_i}{2t^*}}} \\
& = \sum_{i: \epsilon_i \le 2t^*\sqrt{C_n}} e^{-\frac{n\epsilon_i^3}{\sqrt{C_n}}}  + \sum_{i: \epsilon_i \ge 2t^*\sqrt{C_n}} e^{-\frac{2t^* n\epsilon_i^3}{\epsilon_i}} \\
& \le \sum_{i: \epsilon_i \le 2t^*\sqrt{C_n}} e^{-V_i\log{\left(\frac{n}{V_iC_n^2}\right)}}  + \sum_{i: \epsilon_i \ge 2t^*\sqrt{C_n}} e^{-2t^*V_i\log{\left(\frac{n}{V_i}\right)}} \\
& \le \sum_{i: \epsilon_i \le 2t^*\sqrt{C_n}} e^{-V_i\log{\left(\frac{n}{V_i}\right)}}  + \sum_{i: \epsilon_i \ge 2t^*\sqrt{C_n}} e^{-2t^*V_i\log{\left(\frac{n}{V_i}\right)}} \hspace{0.2in} [\because C_n \le 1]\\
& \le \sum_{i: \epsilon_i \le 2t^*\sqrt{C_n}} e^{-2t^* V_i\log{\left(\frac{n}{V_i}\right)}}  + \sum_{i: \epsilon_i \ge 2t^*\sqrt{C_n}} e^{-2t^*V_i\log{\left(\frac{n}{V_i}\right)}}  \hspace{0.2in} [\because 2t^* \le 1]\\
& = \sum_{i=1}^L e^{-2t^* V_i\log{\left(\frac{n}{V_i}\right)}}  \\
& \le \exp{\left(\log{L} - 2t^* V_1\log{\left(\frac{n}{V_1}\right)}\right)} \\
& = \exp{\left(\log{\frac{n}{4\log{p}}} - 2t^*\log{p}\log{\frac{n}{\log{p}}}\right)} \overset{n \to \infty} \to 0 \,.
\end{align*}
This ensures that our choice of $x_i's$ are valid. Applying Theorem \ref{our_model_selection} along with the penalty function $\pen(\mathscr{M}_i) = 2K\epsilon_i^2$ we obtain the following concentration bound on the excess risk: 
\begin{equation}
\label{final_con_highdim_1}
\P\left(\ell\left(\hat \beta_{\hat m}, \beta^0\right) > C\left[\inf_{1 \le i \le L}\left(\ell(\beta_i , \beta^0) + \pen(\mathscr{M}_i)\right)\right] + C_1\frac{tb(n)}{n}\right) \le \Sigma e^{-t}\,.
\end{equation}
Taking $i = s_0$: 
\begin{align*}
\left[\inf_{1 \le i \le L}\left(\ell(\beta_i , \beta^0) + \pen(\mathscr{M}_i)\right)\right] & \le \pen(\mathscr{M}_{s_0}) \\
& = 2K\epsilon_{s_0}^2 \\
& = 2K\left[\left(\frac{V_{s_0}\sqrt{C_n}\log{(n/V_{s_0}C_n^2)}}{n}\right)^{1/3} \vee \left(\frac{V_{s_0}\log{(n/V_{s_0})}}{n}\right)^{1/2}\right]\,.
\end{align*}
Putting this back in equation \eqref{final_con_highdim_1} we get: 
\begin{equation*}
\P\left(\ell\left(\hat \beta_{\hat m}, \beta^0\right) > 2KC\epsilon_{s_0}^2 + C_1\frac{tb(n)}{n}\right) \le \Sigma e^{-t}
\end{equation*}
which further implies: 
\begin{equation}
\label{final_con_highdim_2}
\P\left(\epsilon_{s_0}^{-2} \ell\left(\hat \beta_{\hat m}, \beta^0\right) > 2KC + C_1\frac{tb(n)}{n\epsilon_{s_0}^2}\right) \le \Sigma e^{-t}\,.
\end{equation}
We now argue that the remainder term $b(n)/(n\epsilon_{s_0}^2) \to 0$ as $n \to \infty$. First, observe that $n\epsilon_{s_0}^2 \to 0$ as $n \to \infty$. This is because: 
\begin{align*}
n\epsilon_{s_0}^2 & = \left(\sqrt{n}V_{s_0}\sqrt{C_n}\log{(n/V_{s_0}C_n^2)}\right)^{2/3} \vee \left(V_{s_0}\log{(n/V_{s_0})}\right) \to \infty \,.
\end{align*}
As both $V_{s_0}$ and $n/V_{s_0}$ diverge as $n \to \infty$, it suffices to establish $\sqrt{C_n}/(n\epsilon_1 \epsilon^2_{s_0}) \to 0$, to demonstrate that $b(n)/(n\epsilon_{s_0}^2) \to 0$. Towards that end:
\begin{align*}
\frac{\sqrt{C_n}}{n\epsilon_1 \epsilon^2_{s_0}} & = \frac{\sqrt{C_n}}{n\epsilon_{s_0}^2 \left[\left(\frac{V_1\sqrt{C_n}\log{(n/V_1C_n^2)}}{n}\right)^{1/3} \vee \sqrt{\left(\frac{V_1\log{(n/V_1)}}{n}\right)}\right]} \\
& \le \frac{\sqrt{C_n}}{n\epsilon_{s_0}^2 \left(\frac{V_1\sqrt{C_n}\log{(n/V_1C_n^2)}}{n}\right)^{1/3}} \\
& = \frac{\sqrt{C_n}}{\left(V_1\sqrt{C_n}\log{(n/V_1C_n^2)}\right)^{1/3}\left(V_{s_0}\sqrt{C_n}\log{(n/V_{s_0}C_n^2)}\right)^{2/3}} \\
& = \frac{1}{\left(V_1\log{(n/V_1C_n^2)}\right)^{1/3}\left(V_{s_0}\log{(n/V_{s_0}C_n^2)}\right)^{2/3}} \to 0 \,.
\end{align*}
This completes the proof of the concentration bound on the excess risk. Using Proposition \ref{alpha-kappa} we have: 
\begin{equation*}
\P\left(\epsilon^{-2}_{s_0}\left[\frac{\|\hat \beta - \beta^0\|_2^2}{C_n}\mathds{1}_{\left(d_{\Delta}(\beta, \beta^0) \leq 2t^*C_n\right)} + 2t^*\|\hat \beta - \beta^0\|_2 \mathds{1}_{\left(d_{\Delta}(\beta, \beta^0) > 2t^*C_n\right)}\right] \ge 2KC + C_1s_n t\right) \le \Sigma e^{-t}\,.
\end{equation*}
Thus, we have: 
\begin{equation}
\label{eq:part_1}
\P\left(\frac{\epsilon^{-1}_{s_0}}{\sqrt{C_n}}\left[\|\hat \beta - \beta^0\|_2\mathds{1}_{\left(d_{\Delta}(\beta, \beta^0) \leq 2t^*C_n\right)}\right] \ge \sqrt{2KC} + \sqrt{C_1s_n} t\right) \le \Sigma e^{-t^2}
\end{equation}
and 
\begin{equation}
\label{eq:part_2}
\P\left(\epsilon^{-2}_{s_0}\left[\|\hat \beta - \beta^0\|_2 \mathds{1}_{\left(d_{\Delta}(\beta, \beta^0) > 2t^*C_n\right)}\right] \ge 2KC + C_1s_n t\right) \le \Sigma e^{-t} \,.
\end{equation}
Combining equations \eqref{eq:part_1} and \eqref{eq:part_2} we conclude: 
\begin{equation}
\label{final_con_highdim_3}
\P\left(\left(\frac{r_n}{\sqrt{C_n}} \wedge r_n^2\right) \left\|\hat \beta_{\hat m} - \beta^0\right\|_2 > K_1 + K_2\sqrt{s_n}t\right) \le \Sigma e^{-t^2} + \Sigma e^{-t} \le 2\Sigma e^{-t} 
\end{equation}
for all $t \ge 1$, for some constant $K_1, K_2$ not depending on $(n, p, s_0, C_n)$, with $r_n = \epsilon_{s_0}^{-1}$ and 
$s_n = \frac{b(n)}{n\epsilon^2_{s_0}} \to 0$.

The proof of minimax upper bound of this estimation problem follows immediately from the exponential concentration bound on the estimation error:
\begin{align*}
& \E\left(\left(\frac{r_n}{\sqrt{C_n}} \wedge r_n^2\right)\left\|\hat \beta_{\hat m} - \beta^0\right\|_2\right) \\
& = \int_0^{\infty} \P\left(\left(\frac{r_n}{\sqrt{C_n}} \wedge r_n^2\right)\left\|\hat \beta_{\hat m} - \beta^0\right\|_2 > y\right) \ dy \\
& = \int_0^{K_1} \P\left(\left(\frac{r_n}{\sqrt{C_n}} \wedge r_n^2\right)\left\|\hat \beta_{\hat m} - \beta^0\right\|_2 > y\right) \ dy + \int_{K_1}^{\infty} \P\left(\left(\frac{r_n}{\sqrt{C_n}} \wedge r_n^2\right) \left\|\hat \beta_{\hat m} - \beta^0\right\|_2 > y\right) \ dy \\
& \le K_1 + 2\Sigma  \int_{K_1}^{\infty} e^{-\frac{y - K_1}{K_2 \sqrt{s_n}}} \ dy \hspace{0.2in} [\because K_1 \ge 1]\\
& \le K_1 + 2\Sigma K_2 \sqrt{s_n}  \int_0^{\infty} e^{-y} \ dy  \\
& = K_1 + 2\Sigma K_2\sqrt{s_n} < \infty \,.
\end{align*}

%As $(V_1\log{(en/V_1)})^{\frac{1-\kappa}{2\kappa -1}} < 1$ we have: 
%\begin{equation*}
%    \P\left(S(\beta^0) - S(\hat \beta_{\hat m})>3\left(\frac{V_{s_0}\log{(en/V_{s_0})}}{n}\right)^{\frac{\kappa}{2\kappa-1}}+ 2K\frac{y}{n^{\frac{\kappa}{2\kappa -1}}}\right) \le \Sigma e^{-y}
%\end{equation*}{}
\subsection{Proof of Corollary \ref{cor-high-dim-rate}}
In Theorem \ref{high-dim-rate} we have established: 
$$\P\left(\left(\frac{r_n}{\sqrt{C_n}} \wedge r_n^2\right) \left\|\hat \beta_{\hat m} - \beta^0\right\|_2 > K_1 + K_2\sqrt{s_n} t\right) \le 2\Sigma e^{-t}$$
for some constant $K_1, K_2$ which does not depend on $(n, p, s_0, C_n)$. Taking $t = 1/\sqrt{s_n}$ we get: 
\begin{equation*}
\P\left(\left\|\hat \beta_{\hat m} - \beta^0\right\|_2 > (K_1 + K_2)\left(\frac{r_n}{\sqrt{C_n}} \wedge r_n^2\right)^{-1} \right) \le 2\Sigma e^{-\frac{1}{\sqrt{s_n}}}  \to 0
\end{equation*}
because $s_n \to 0$ as $n \to \infty$. Hence, if $\beta^0_{\min} > (K_1 + K_2)r_n^{-1}$, we conclude from above concentration bound: 
\begin{equation*}
\P\left(\hat m \supseteq m_0\right) \ge 1 - 2\Sigma e^{-\frac{1}{\sqrt{s_n}}} \,.
\end{equation*} 
which completes the proof.

\subsection{Proof of Proposition \ref{multi_pop_min}}
By the definition of population loss function we have: 
\begin{align*}
S^{(mult)}(\beta) & = \frac{1}{m(m-1)}\sum_{j=1}^m \E\left[p_j(\mb{X})\left(\sum_{k \neq j} \mathds{1}((\mb{x}_j - \mb{x}_k)'\beta >  0)\right) \right] \\
& = \frac{1}{m(m-1)}\sum_{j=1}^m \E\left[p_j(\mb{X})\left(\text{rank}(\mb{x}_j^T\beta) - 1\right) \right] \\
\end{align*}
where we define $\text{rank}(\mb{x}_j^T\beta)$ as the rank of the scalar number $\mb{x}_j^T\beta$ among the $m$ numbers $\{\mb{x}_1^T\beta, \mb{x}_2^T\beta, \dots, \mb{x}_m^T\beta\}$ in increasing order. Our claim is that for any realization of the vectors $\mb{x}_1, \dots, \mb{x}_m$, we have: 
$$\sum_{j=1}^m \left[p_j(\mb{X})\left(\text{rank}(\mb{x}_j^T\beta^0) - 1\right) \right] \ge \sum_{j=1}^m \left[p_j(\mb{X})\left(\text{rank}(\mb{x}_j^T\beta) - 1\right) \right]$$
To observe this, first note that from Assumption \ref{ass:rank_mult}, the ordering of the vectors $\{p_j(\mb{X})\}_{j=1}^m$ is same as $\{\mb{x}_j^T\beta^0\}_{j=1}^m$. Hence the above inequality follows from applying rearrangement inequality. The proof of the proposition also immediately follows from the inequality.

\subsection{Proof of Proposition \ref{multi-alpha-kappa}}
First we show that under Assumption \ref{ass:low_noise_mult} we can lower bound the excess risk in terms of the probability of a suitably chosen wedge shaped region. For $j \neq k \in \{1,2,\dots, m\}$ and for any $\beta \in S^{p-1}$ define the region $X_{j,k,\beta}$ as: $$X_{j,k,\beta} = \{x \in \mathbb{R}^{m \times p}: \s((x_{j,*} - x_{k,*})^T\beta) \neq \s((x_{j,*} - x_{k,*})^T \beta^0)\}$$ where for any matrix $A$, its $i^{th}$ row is denoted by $A_{i,*}$. Then we have:
\allowdisplaybreaks
\begin{align*}
& S^{(mult)}(\beta^0) - S^{(mult)}(\beta) \\
& = \frac{1}{m(m-1)}\sum_{(j \neq k)}\E\left[\left(p_j(X) - p_k(X)\right)\left(\mathds{1}((X_j - X_k)'\beta^0 \ge 0) - \mathds{1}((X_j - X_k)'\beta \ge 0\right))\right] \\
& = \frac{1}{m(m-1)}\sum_{(j \neq k)} \int_{X_{j,k,\beta}} |p_j(x) - p_k(x)| \ dF(x) \\
& = \frac{1}{m(m-1)}\sum_{(j \neq k)} \int_{X_{j,k,\beta}} \frac{|p_j(x) - p_k(x)|}{p_j(x) + p_k(x)}(p_j(x) + p_k(x))\ dF(x) \\
& \ge \frac{1}{m(m-1)}\sum_{(j \neq k)} \int_{X_{j,k,\beta} \cap\{\|x\|_F \le R\}} \frac{|p_j(x) - p_k(x)|}{p_j(x) + p_k(x)}(p_j(x) + p_k(x))\ dF(x) \\
& \ge \frac{c_1}{m(m-1)}\sum_{(j \neq k)}\int_{X_{j,k,\beta} \cap\{\|x\| \le R|\}}\frac{|p_j(x) - p_k(x)|}{p_j(x) + p_k(x)}\ dF(x) \\
& \ge \frac{c_1}{m(m-1)}\sum_{(j \neq k)}\int_{X_{j,k,\beta} \cap\{\|x\| \le R|\}} \left|\frac{p_j(x)}{p_j(x) + p_k(x)}- \frac{1}{2}\right|\ dF(x) \\
& \ge\frac{c_1}{m(m-1)}\sum_{(j \neq k)}\int_{X_{j,k,\beta} \cap\{\|x\| \le R|\}} \left|\frac{p_j(x)}{p_j(x) + p_k(x)}- \frac{1}{2}\right|\mathds{1}_{ \left|\frac{p_j(x)}{p_j(x) + p_k(x)}- \frac{1}{2}\right| \ge t}\ dF(x) \\
& \ge   \frac{c_1t}{m(m-1)}\sum_{(j \neq k)} \P\left( \left|\frac{p_j(x)}{p_j(x) + p_k(x)}- \frac{1}{2}\right| \ge t, \left(X_{j,k,\beta} \cap \{\|X\|_F\le R\}\right)\right) \\
& \ge \frac{c_1t}{m(m-1)}\sum_{(j \neq k)} \left[\P\left(\left(X_{j,k,\beta} \cap \{\|X\|\le R\}\right)\right) - \P\left( \left|\frac{p_j(x)}{p_j(x) + p_k(x)}- \frac{1}{2}\right| \le t, \right)\right] \\
& \ge \frac{c_1t}{m(m-1)}\sum_{(j \neq k)} \left[\P\left(\left(X_{j,k,\beta} \cap \{\|X\|\le R\}\right)\right) -Ct\right] \\
& = c_1t \left[\frac{1}{m(m-1)}\sum_{(j \neq k)} \P\left(\left(X_{j,k,\beta} \cap \{\|X\|\le R\}\right)\right)\right] -Cc_1t^2 
% & \ge c_1\sup_{t \in (0, t^*)}\left[2ct  \frac{1}{\dbinom{J}{2}}\sum_{(j,k)} \left[\P\left(\left(X_{j,k,\beta} \cap \{\|X\|\le R\}\right)\right) -Ct^{\alpha}\right]\right] \\
\end{align*}
Defining $d_{\Delta}(\beta, \beta^0)$ to be: 
$$d_{\Delta}(\beta_1, \beta_2) = \left[\frac{1}{m(m-1)}\sum_{(j \neq k)} \P\left(\left(X_{j,k,\beta} \cap \{\|X\|\le R\}\right)\right)\right] $$
we obtain: 
$$S^{(mult)}(\beta^0) - S^{(mult)}(\beta) \ge c_1\left[td_{\Delta}(\beta, \beta^0) -Ct^2\right] $$
As this inequality is true for any $0 \le t \le t^*$,  optimizing the same way as in the proof of Proposition \ref{alpha-kappa} we conclude that: 
\begin{align}
&S^{(mult)}(\beta^0) - S^{(mult)}(\beta) \notag \\
\label{eq:mult_wedge_bound}  & \ge c_1\left[\frac{d^2_{\Delta}(\beta, \beta^0)}{4C_n}\mathds{1}_{\left(d_{\Delta}(\beta, \beta^0) \leq 2t^*C_n\right)} + t^*\frac{d_{\Delta}(\beta, \beta^0)}{2}\mathds{1} _{\left(d_{\Delta}(\beta, \beta^0) > 2t^*C_n\right)}\right] \\
& = c_1 \frac{d^2_{\Delta}(\beta, \beta^0)}{4C}  \hspace{0.2in} [\text{As }2t^*C > 1 \text{ by definition of }C] \notag \\
& \ge \frac{c_1c^2_2}{4C}\|\beta - \beta^0\|_2^2 \hspace{0.2in} [\text{By Assumption }\ref{ass:wedge_mult}] \notag
\end{align}
This concludes the proof.

\subsection{Proof of Theorem \ref{multi_low_dim}}
The proof of this theorem is quite similar to proof of Theorem \ref{rate-manski-p-less-n}. Hence we will skip some details here. As before, we work with the distance metric $\sqrt{d_{\Delta}(\beta,\beta^0)}$ over the parameter space $S^{p-1}$. Borrowing the notations from Theorem \ref{massart} and using Proposition \ref{multi-alpha-kappa} we have: 
$$\ell(\beta, \beta^0) = S^{(mult)}(\beta^0) - S^{(mult)}(\beta) \ge \frac{c_1}{4C} d^2_{\Delta}(\beta, \beta^0) = \left(\omega^{-1}\left(\sqrt{d_{\Delta}\left(\beta, \beta^0\right)}\right)\right)^2$$
for $\omega \in \mathcal{C}_1$. Hence the function $\omega(x)$ satisfy: 
$$\omega^{-1}\left(\sqrt{d_{\Delta}\left(\beta, \beta^0\right)}\right) = \frac{\sqrt{c_1}}{2\sqrt{C}}d_{\Delta}(\beta, \beta^0)$$
Parametrizing $\sqrt{d_{\Delta}\left(\beta, \beta^0\right)} = t$ we get: 
$$\omega^{-1}(t) = \frac{\sqrt{c_1}}{2\sqrt{C}}t^2 \iff \omega(x) = \left(\frac{4C}{c_1}\right)^{1/4}\sqrt{x} \,.$$
Now consider the class of function 
$$\mathcal{F} = \left\{f_{\beta}: f_{\beta}(Y,X) = \frac{1}{m(m-1)}\sum_{k \neq j}\left[Y_j \mathds{1}(X_j'\beta > X_k'\beta)\right]\right\} \,.$$
As $f_{\beta}$ is average of $m(m-1)$ functions, where each function constitutes a VC class of VC dimension of order $p$, the collection $\mathcal{F}$ has bounded uniform entropy integral. More precisely we have for any measure $Q$: 
$$N(\epsilon, \mathcal{F}, L_2(Q)) \lesssim \left(\frac{1}{\epsilon}\right)^{m^2p} \,,$$
as each function is bounded by $1$ and $N(\epsilon, \mathcal{F}, L_2(Q))$ is the covering number of $\mathcal{F}$ with respect to $L_2(Q)$ measure. The variability of the centered function can be bounded as: 
\begin{align*}
    & P(f_{\beta} - f_{\beta^0})^2 \\
    & \le \frac{1}{m(m-1)} \sum_{k \neq j} \left[\P(\s((X_j - X_k)'\beta) \neq \s((X_j - X_k)'\beta^0))\right] \\
    & \le  \frac{c_3}{m(m-1)} \sum_{k \neq j} \left[\P(\s((X_j - X_k)'\beta) \neq \s((X_j - X_k)'\beta^0) \cap \{\|X\| \le R\})\right]  \hspace{0.2in} [\text{Assumption }\ref{ass:wedge_mult}]\\
    &= c_3 \ d_{\Delta}(\beta, \beta^0)
\end{align*}
Finally to apply Theorem \ref{massart} we need to obtain $\phi(\sigma)$ which satisfy condition (2) of that Theorem. Using Theorem 8.7 of \cite{sen2018gentle} one can choose $\phi(\sigma)$ as: 
$$\phi(\sigma) = \sigma\sqrt{m^2p\log{\frac{1}{\sigma}}} \vee \frac{m^2p}{\sqrt{n}}\log{\frac{1}{\sigma}}\,.$$ 
Now from Theorem \ref{massart} we need to find $\phi(\sigma)$ for all the values of $\sigma$ such that $\phi(\sigma) \le \sqrt{n}\sigma^2$. From the above expression of $\phi(\sigma)$, one can immediately conclude that: 
$$\left\{\sigma: \phi(\sigma) \le \sqrt{n} \sigma^2 \right\} = \left\{\sigma: \sigma\sqrt{m^2p\log{\frac{1}{\sigma}}} \ge \frac{m^2p}{\sqrt{n}}\log{\frac{1}{\sigma}}\right\} \,.$$
Hence we can take 
$$\phi(\sigma)= \sigma\sqrt{m^2p\log{\frac{1}{\sigma}}} \,,$$
such that condition (3) of Theorem \ref{massart} will be satisfied for all $\sigma$ such that $\phi(\sigma) \le \sqrt{n}\sigma^2$. 
Now we need to solve the equation 
$$\sqrt{n}\epsilon_*^2 \ge \phi(\omega(\epsilon_*))$$ 
to get $\epsilon_*$. From the expression of $\phi(\sigma)$ we have: 
\begin{align}
\phi(\omega(\epsilon_*)) \le \sqrt{n}\epsilon_*^2 & \iff \omega(\epsilon_*)\sqrt{m^2p\log{\frac{1}{\omega(\epsilon_*)}}} \notag\\
& \iff \omega(\epsilon_*)\sqrt{m^2p\log{\frac{1}{\omega(\epsilon_*)}}}  \le \sqrt{n}\epsilon_*^2   \notag\\
\label{ineq_2} & \iff  \left(\frac{4C}{c_1}\right)^{1/4}\sqrt{\epsilon_*} \sqrt{m^2p\log{\frac{\sqrt{c_1}}{2\sqrt{C}\epsilon_*}}} \le \sqrt{n}\epsilon_*^2 \notag
\end{align}
The above inequality will be satisfied if: 
\begin{align*}
\left(\frac{4C}{c_1}\right)^{1/4}\sqrt{\epsilon_*} \sqrt{m^2p\log{\frac{\sqrt{c_1}}{2\sqrt{C}\epsilon_*}}} \le \sqrt{n}\epsilon_*^2  & \iff \epsilon_*^{-3}\log{\frac{\sqrt{c_1}}{2\sqrt{C}\epsilon_*}}  \le \frac{n}{m^2p}\left(\frac{c_1}{4C}\right)^{1/2} 
\end{align*}
Ignoring the constant $C$, This will be satisfied if we take $\epsilon_*$ to be: 
\begin{equation}
\label{ep_1}
\epsilon_* = \left(\frac{n\sqrt{c_1}}{m^2p}\right)^{-1/3}\log{\left(\frac{nc_1^2}{m^2p}\right)}^{1/3}
\end{equation}
%Focusing on the first inequality of the above two we get: 
%\begin{align*}
%\frac{C_n^{1/4}}{c_1^{1/4}}\sqrt{\epsilon_*}\sqrt{m^2p\log{\frac{\sqrt{c_1}}{\sqrt{C_n}\epsilon_*}}}  \le \sqrt{n}\epsilon_*^2 & \iff \epsilon_*^{-3}\log{\frac{\sqrt{c_1}}{\sqrt{C_n}\epsilon_*}} \le \frac{n}{m^2p}\sqrt{\frac{c_1}{C_n}}
%\end{align*}
%Hence a valid choice for $\epsilon_*$ that satisfies the above equation will be: 
%\begin{equation}
%\label{ep_1}
%\epsilon_* = \left(\frac{m^2p}{n}\sqrt{\frac{C_n}{c_1}}\log{\left(\frac{m^2p}{n}\sqrt{\frac{C_n}{c_1}}\right)}\right)^{1/3}
%\end{equation}
%Going back to the second inequality of equation \eqref{ineq_2} we get: 
%\begin{align*}
%\frac{m^2p}{\sqrt{n}}\log{\frac{\sqrt{c_1}}{2\sqrt{C}\epsilon_*}} \le \sqrt{n}\epsilon_*^2 & \iff \epsilon_*^{-2} \log{\frac{\sqrt{c_1}}{2\sqrt{C}\epsilon_*}} \le \frac{n}{m^2p}
%\end{align*}
%From this inequality, another valid choice of $\epsilon_*$ will be: 
%\begin{equation}
%\label{ep_2}
%\epsilon_* =  \left(\frac{n}{m^2p}\right)^{-1/2}\log{\left(\frac{nc_1}{m^2p}\right)}^{1/2} 
%\end{equation}
%Combining equation \eqref{ep_1}, \eqref{ep_2} we get: 
%\begin{align}
%\epsilon_* & = \left(\frac{n\sqrt{c_1}}{m^2p}\right)^{-1/3}\log{\left(\frac{nc_1^2}{m^2p}\right)}^{1/3} \vee  \left(\frac{n}{m^2p}\right)^{-1/2}\log{\left(\frac{nc_1}{m^2p}\right)}^{1/2} \,.
%\end{align}
%Hence the first term in the above maxima will be dominating, from which we conclude: 
%$$\epsilon_* = \left(\frac{n\sqrt{c_1}}{m^2p}\right)^{-1/3}\log{\left(\frac{n\sqrt{c_1}}{m^2p}\right)}^{1/3} \,.$$
Using this value of $\epsilon_*$ we conclude from Theorem \ref{massart}: 
\begin{equation}
\label{eq:risk_bound_1}
\P\left(S^{(mult)}(\beta^0) - S^{(mult)}(\hat{\beta}) \ge K y\epsilon_*^2\right) \le e^{-y}
\end{equation}
or all $y \ge 1$ and for some constant $K$ which does not depend on $(n,m,p)$. Now Proposition \ref{multi-alpha-kappa} implies along with equation \eqref{eq:risk_bound_1}:  
\begin{equation}
\label{mult_conc_1}
\P\left(r_n \|\hat \beta - \beta^0\|_2 \ge Ky\right) \le e^{-y^2}
\end{equation}
for all $y \ge 1$, where 
\begin{align*}
r_n & =  \sqrt{c_1}\left(\frac{n\sqrt{c_1}}{m^2p}\right)^{1/3}\log{\left(\frac{nc_1^2}{m^2p}\right)}^{-1/3} \\
& = \left(\frac{nc_1^2}{m^2p}\right)^{1/3}\log{\left(\frac{nc_1^2}{m^2p}\right)}^{-1/3}
\end{align*}
and for some constant $K$ but does not depend on $(n, m, p)$.  This concludes the proof.

%Applying Theorem \ref{massart} with $\omega(x) = x^{1/\kappa}$ taking $\epsilon_*^2 = \left(\frac{2L^2 m^2 p \log{(en/(m^2 p))}}{n}\right)^{\frac{\kappa}{2\kappa - 1}}$ for some absolute constant $L$ by using same argument as in the proof of Theorem \ref{rate-manski-p-less-n}, we conclude that: 
%\begin{align}
%\label{eq:mult_eq1}
%\P\left(S(\beta^0) - S(\hat \beta) \ge Ky\epsilon_*^2\right) \le e^{-y}
%\end{align} 
%for all $y \ge 1$ and for some absolute constant $K$. From Proposition \ref{multi-alpha-kappa} we have $$S(\beta^0) - S(\hat \beta) \ge u_- c_1 \|\hat \beta - \beta^0\|_2^{\kappa}$$ which, along with equation \eqref{eq:mult_eq1} gives us for all $y \ge 1$: 
%$$\P\left(r_n \|\hat \beta - \beta^0\|_2 \ge (Ky)^{1/\kappa}(2L^2)^{1/(2\kappa -1)}\right) \le e^{-y}$$
%which concludes the proof of the theorem. 

\subsection{Proof of Theorem \ref{multi_high_dim}}
The proof technique is essentially similar to that of Theorem \ref{high-dim-rate}, hence we will skip some details. As in the proof of previous theorem, the distance function that will be used heavily in the proof: 
$$d_{\Delta}(\beta, \beta^0) =  \sqrt{\frac{1}{m(m-1)} \sum_{k \neq j} \left[\P(X_{j,k,\beta}\cap \{\|X\| \le R\})\right]}$$ 
%This along with Proposition \ref{multi-alpha-kappa} implies that: 
%$$S(\beta^0) - S(\beta) \gtrsim c_1d(\beta, \beta^0)^{2\kappa}$$ which establishes the curvature condition we need to use for the adaptive model selection. Hence we can take $\omega(x) = x^{1/\kappa}$ where $\omega(x)$ is same as in Theorem \ref{massart}. To control the variability of loss function, consider the class of functions: $$\mathcal{F} = \left\{f_{\beta}: f_{\beta}(Y,X) = \frac{2}{m(m-1)}\sum_{k \neq j}\left[Y_j \mathds{1}(X_j'\beta > X_k'\beta)\right]\right\}$$ The variability of the centered process: 
%\begin{align*}
%    & P(f_{\beta} - f_{\beta^0})^2 \\
%    & \le \frac{2}{m(m-1)} \sum_{k \neq j} \left[\P(\s((X_j - X_k)'\beta) \neq \s((X_j - X_k)'\beta^0))\right] \\
%    & \le c_3 \frac{2}{m(m-1)} \sum_{k \neq j} \left[\P(\s((X_j - X_k)'\beta) \neq \s((X_j - X_k)'\beta^0) \cap \{\|X\| \le R\})\right]  \hspace{0.2in} [\text{Assumption }\ref{ass:wedge_mult}]\\
%    &\le c_3 \ d^2(\beta, \beta^0)
%\end{align*}
From Assumption \ref{ass:sparsity_mult} we confine ourselves to search the best model upto sparsity level $L = (nc^2_1)/(4m^2\log{p})$. As in the proof of Theorem \ref{high-dim-rate} we will use Theorem \ref{our_model_selection} to show the model selection consistency here. Recall that $\mathscr{M}_i$ is defined to be the collection of all the models with all the $\beta$ such that $\|\beta\|_0 \le i$. Hence, 
$$\mathscr{M}_i = \left\{f_{\beta}: f_{\beta}(Y,X) = \frac{1}{m(m-1)}\sum_{k \neq j}\left[Y_j \mathds{1}(X_j'\beta > X_k'\beta)\right], \|\beta\|_0 \le i\right\}$$ 
Now $f_{\beta}$ is sum of order $m^2$ many functions each of which has VC dimension of order $i\log{(ep/i)}$ (argued in Theorem \ref{high-dim-rate}). Using the same argument as in Theorem \ref{multi_low_dim} we say that, 
$$N(\epsilon,\mathscr{M}_i, L_2(Q)) \lesssim \left(\frac{1}{\epsilon}\right)^{m^2i\log{(ep)/i}} \,.$$ 
As before, define $V_i = i\log{(ep)/i}$. Using the same calculation as in the proof of Theorem \ref{multi_low_dim} we conclude: 
$$\omega(x) =\left(\frac{4C}{c_1}\right)^{1/4}\sqrt{x} \,,$$
and
$$\phi_i(\sigma) = \sigma\sqrt{m^2V_i\log{\frac{1}{\sigma}}} \,.$$
Hence, the value of  $\epsilon_i$ can be takes as following similar calculation as in Theorem \ref{multi_low_dim}: 
$$\epsilon^2_i = \left(\frac{n\sqrt{c_1}}{m^2V_i}\right)^{-1/3}\log{\left(\frac{nc_1^2}{m^2V_i}\right)}^{1/3}$$
Now we have to take the penalty function so that it satisfies: 
$$\pen(\mathscr{M}_i) \ge K\left(\epsilon_i^2 + \frac{\omega^2(\epsilon_i)x_i}{\epsilon_i^2n}\right) \,.$$
for some constant $K$ (as mentioned inTheorem \ref{our_model_selection}), where $x_i$ is same as defined in Theorem \ref{our_model_selection}. Taking $x_i =  n\epsilon_i^4/\omega^2(\epsilon_i)$ a valid choice of penalty function will be:
$$\pen(\mathscr{M}_i)  = 2K\epsilon_i^2 \,.$$ 
Before applying Theorem \ref{our_model_selection} with this choice of penalty function, we need to argue that this choice of $x_i$ is valid, i.e. 
\begin{equation}
\label{eq:x_value_bound}
\sum_{i=1}^L x_i = \Sigma < \infty \,,
\end{equation}
for some constant $\Sigma$. From the definition of $x_i$ we have: 
\begin{align*}
x_i = \frac{n\epsilon_i^4}{\omega^2(\epsilon_i)} = \left(\frac{c_1}{4C}\right)^{1/2} n\epsilon_i^3  & =  \left(\frac{c_1}{4C}\right)^{1/2} n\left(\frac{n\sqrt{c_1}}{m^2V_i}\right)^{-1} \log{\left(\frac{nc_1^2}{m^2V_i}\right)} \\
& = \frac{m^2V_i}{2\sqrt{C}}\log{\left(\frac{nc_1^2}{m^2V_i}\right)} 
\end{align*}
Similar analysis as in Theorem \ref{high-dim-rate} yields: 
$V_1 \le \dots \le V_L$ and $x_1 \le \dots \le x_L$. Hence we have: 
\begin{align*}
\sum_{i=1}^L e^{-x_i} & = \sum_{i=1}^L e^{-\frac{m^2V_i}{2\sqrt{C}}\log{\left(\frac{nc_1^2}{m^2V_i}\right)} } \\
& = \exp{\left(\log{L} - \frac{m^2V_1}{2\sqrt{C}}\log{\left(\frac{nc_1^2}{m^2V_1}\right)} \right)} \\
& = \exp{\left(\log{\frac{nc_1^2}{m^2\log{p}}} - \frac{m^2\log{p}}{2\sqrt{C}}\log{\left(\frac{nc_1^2}{m^2\log{p}}\right)}\right)} \\
& = \exp{\left(\log{\frac{nc_1^2}{m^2\log{p}}} - \frac{m^2\log{p}}{2\sqrt{C}}\log{\left(\frac{nc_1^2}{m^2\log{p}}\right)} \right)} \\
& = \exp{\left(\log{\frac{nc_1^2}{m^2\log{p}}} \left(1 - \frac{m^2\log{p}}{2\sqrt{C}}\right) \right)} \to 0
\end{align*} 
as $n \to \infty$. Hence we can find some constant $\Sigma$ (in-fact one can take $\Sigma = 1$ for all large $n$) such that equation \eqref{eq:x_value_bound} holds. Now, using Theorem \ref{our_model_selection} we conclude: 
\begin{equation}
\label{final_con_highdim_mult_1}
\P\left(\ell\left(\hat \beta_{\hat m}, \beta^0\right) > C\left[\inf_{1 \le i \le L}\left(\ell(\beta_i , \beta^0) + \pen(\mathscr{M}_i)\right)\right] + C_1\frac{yb(n)}{n}\right) \le \Sigma e^{-y}
\end{equation}
where $b(n) =  \frac{\omega^2(\epsilon_1)}{\epsilon_1^2}$. Now in the RHS of equation \eqref{final_con_highdim_mult_1} we can replace the infimum by its value at $i = s_0$, the true sparsity of $\beta^0$ and get the following concentration bound on the excess risk:
\begin{equation}
\label{final_con_highdim_mult_2}
\P\left(S^{(mult)}\left(\beta^0\right) - S^{(mult)}\left(\hat \beta_{\hat m}\right)  > 2KC\epsilon^2_{s_0} + C_1\frac{yb(n)}{n}\right) \le \Sigma e^{-y}
\end{equation}
Using Proposition \ref{multi-alpha-kappa} we conclude: 
$$\P\left(r_n \|\hat \beta_{\hat m} - \beta^0\|_2 \ge \sqrt{2KC} + y\sqrt{C_1s_n}\right) \le \Sigma e^{-y^2}$$
where 
$$s_n = \frac{b(n)}{n\epsilon^2_{s_0}} \to 0$$
via same argument in Theorem \ref{multi-alpha-kappa} and 
\begin{align*}
r_n & =\left(\frac{nc_1^2}{m^2V_{s_0}}\right)^{1/3}\log{\left(\frac{nc_1^2}{m^2V_{s_0}}\right)}^{-1/3} \\
& = \left(\frac{nc_1^2}{m^2s_0\log{(ep/s_0)}}\right)^{1/3}\log{\left(\frac{nc_1^2}{m^2s_0\log{(ep/s_0)}}\right)}^{-1/3}
\end{align*}
This concludes the proof of the Theorem with $a_n = \sqrt{s_n}$.

\subsection{Proof of Theorem \ref{our_model_selection}: }
The proof is quite long and involved. Before going into details, we define some notation which will be frequently used throughout. Fix $y \ge 0$ and $\{x_m\}_{m \in \mathcal{M}}$ satisfying the condition of the theorem: 
\begin{enumerate}
\item For all $m \in \mathcal{M}$, $s_m = \argmin_{t \in S_m}\ell(s,t)$.  
\\\\
\item $y^2_{m} = 2K \left(\epsilon_{m}^2 + \frac{\omega^2(\epsilon_{m})(x_{m}+y)}{n\epsilon_{m}^2}\right) \le 2\pen(M_m) + 2Ky\frac{\omega^2(\epsilon_m)}{n\epsilon_m^2} \hspace{0.2in} [\text{Defintion of }\pen]$. 
\\\\
\item $V_{m} = \sup_{t \in S_m} \frac{\left[\bar{\gamma}_n(s_{m'}) - \bar{\gamma}_n(t)\right]}{\ell(s, t)+\ell(s,s_{m'}) + y^2_{m}}$ for all $m \in \mathcal{M}$. 
\end{enumerate}
Recall that $\hat m$ is defined to be the optimal model, i.e.: 
$$\hat m = \argmin_{m \in \mathcal{M}} \left[\gamma_n(\hat s_m)+ \pen(m)\right] \,.$$
Fix $m' \in \mathcal{M}$. Then by definition of $\hat m$:  
\allowdisplaybreaks
\begin{align*}
& \gamma_n(\hat s_{\hat m}) + \text{pen}(\hat m) \le \gamma_n(\hat s_{m'}) + \text{pen}(m') \le \gamma_n(s_{m'}) + \text{pen}(m')
\end{align*}
which implies: 
\begin{align}
\ell(\hat s_{\hat m}, s) & \le \ell(s_{m'}, s) + \bar\gamma_n(s_{m'}) - \bar \gamma_n (\hat s_{\hat m}) - \text{pen}(\hat m) +\text{pen}(m') \notag\\
& \le \ell(s_{m'}, s) + \sup_{t \in S_{\hat m}}\frac{\bar\gamma_n(s_{m'}) - \bar \gamma_n (t)}{\ell(s,t)+\ell(s,s_{m'})+y^2_{\hat m}}\left(\ell(s,s_{m'})+\ell(s,\hat s_{\hat m})+y^2_{\hat m}\right) - \pen(\hat m) +\pen(m') \notag\\
\label{ineq}  & \le \ell(s_{m'}, s) + V_{\hat m}\left(\ell(s,s_{m'})+\ell(s,\hat s_{\hat m})+y^2_{\hat m}\right)- \text{pen}(\hat m) +\text{pen}(m') \,.
\end{align}
The rest of the proof is organized as follows. We first show that: 
\begin{equation}
\label{main_ineq_talagrand}
P(V_{m} >1/2) \le e^{-(x_{m} + y)}
\end{equation}
for any $m \in \mathcal{M}$, which implies from the union bound and the fact that $\sum_{i=1}^L e^{-x_i} \le \Sigma$,
$$\P(V_{\hat m} > 1/2) \le \Sigma e^{-y} \,.$$ 
Now, using equation \eqref{ineq}, we obtain with probability larger that $1 - \Sigma e^{-y}$: 
\begin{align*}
& \ell(\hat s_{\hat m}, s) \notag \\
& \le\ell(s,s_{m'}) + \frac{1}{2}\left(\ell(s,s_{m'})+\ell(s,\hat s_{\hat m})+y^2_{\hat m}\right)- \text{pen}(\hat m) +\text{pen}(m') \\
& \le \frac32\ell(s,s_{m'}) + \text{pen}(m') + \frac12\ell(s,\hat s_{\hat m}) + K\frac{\omega^2(\epsilon_{\hat m})y}{n\epsilon_{\hat m}^2} \hspace{0.2in} [\text{Using defintion of }y_m]\\
& \le \frac32\ell(s,s_{m'}) + \text{pen}(m') + \frac12\ell(s,\hat s_{\hat m}) + \frac{K y}{n}\sup_{m \in \mathcal{M}} \frac{\omega^2(\epsilon_{\hat m})}{\epsilon_{\hat m}^2} \\
& \le \frac32\ell(s,s_{m'}) + \text{pen}(m') + \frac12\ell(s,\hat s_{\hat m}) + \frac{Ky}{n} \frac{\omega^2(\epsilon_{m_{(1)}})}{\epsilon_{m_{(1)}}^2} \hspace{0.2in} [\because \omega(x)/x \text{ is } \downarrow \ \text{and using definition of }\epsilon_{m_{(1)}}]\\
& \le \frac32\ell(s,s_{m'}) + \text{pen}(m') + \frac12\ell(s,\hat s_{\hat m}) + K\frac{yb(n)}{n}  \hspace{0.2in} [\text{Definition of }b(n)]
\end{align*}
Multiplying both sides by 2, we get: 
$$\ell(\hat s_{\hat m}, s)  \le 3\left[\ell(s,s_{m'}) + \text{pen}(m')\right] + 2K\frac{yb(n)}{n} \,.$$
As this is true for any $m' \in \mathcal{M}$, we conclude: 
$$\ell(\hat s_{\hat m}, s)  \le 3\inf_{m'\in\mathcal{M}}\left[\ell(s,s_{m'}) + \text{pen}(m')\right] + 2K\frac{yb(n)}{n} \,,$$
which implies for all $y \ge 0$:
\begin{equation}
    \label{pen-choice} \P\left(\ell(\hat s_{\hat m}, s) > 3\left[\inf_{m \in 
    \mathcal{M}}\left(\ell(s_{m}, s) + \text{pen}(m)\right)\right] + 2K\frac{y b(n)}{n}\right) \le \Sigma e^{-y}\,.
\end{equation}
Integrating with respect to $y$ we get the following upper bound on the expectation: 
\begin{equation}
   \E(\ell(\hat s_{\hat m}, s))\le 3\left[\inf_{m \in \mathcal{M}}\left(\ell(s, S_m) + \text{pen}(m)\right)\right] + 2K\frac{\Sigma b(n)}{n} \,.
\end{equation}

\noindent
Now we prove the bound \eqref{main_ineq_talagrand}, for which we use Theorem \ref{thm:bousquet_talagrand}. Our function class will be: 
$$\mathcal{F} = \left\{\frac{\gamma(t,\cdot) - \gamma(s_{m'} , \cdot)}{\ell(t, s)+\ell(s_{m'}, s) + y^2_{m}}: \ t \in S_m\right\} \,.$$
First, we observe that the functions are uniformly bounded in terms of $y_m$: 
\allowdisplaybreaks
\begin{align*}
\left\|\frac{\gamma(t,\cdot) - \gamma(s_{m'} , \cdot)}{\ell(t, s)+\ell(s_{m'}, s) + y^2_{m}}\right\|_{\infty} \le \frac{1}{y^2_{m}}
\end{align*}
\noindent
Next, we bound the variability of the functions. Define $\omega_1 = 1 \wedge \omega$. We have:
\begin{align*}
Var\left(\frac{\gamma(t,\cdot) - \gamma(s_{m'} , \cdot)}{\ell(s, t)+\ell(s, s_{m'}) + y^2_{m}}\right) & \le \frac{d^2(s_m,t)}{\left(\ell(t, s)+\ell(s_m, s) + y^2_{m}\right)^2} \hspace{0.2in} [\because \gamma \in [0,1]] \\
& \le 2\left[\frac{d^2(s_{m'},s)}{(\ell(s, s_{m'}) + y^2_{m})^2}+\frac{d^2(t,s)}{(\ell(s, t) + y^2_{m})^2}\right] \\
& \le 2\left[\frac{\omega^2\left(\sqrt{\ell(s_{m'},s)}\right)}{(\ell(s,s_{m'}) + y^2_{m})^2}+\frac{\omega^2\left(\sqrt{\ell(t,s)}\right)}{(\ell(s,t) + y^2_{m})^2}\right] \\
& \le 4\sup_{\epsilon \ge 0} \frac{\omega^2(\epsilon)}{(\epsilon^2 + y^2_{m})^2}  \\
& \le \frac{4}{y^2_{m}} \sup_{\epsilon \ge 0}  \frac{\omega^2(\epsilon)}{(\epsilon \vee y_{m})^2} \le \frac{4\omega^2(y_{m})}{y^4_{m}}
\end{align*}
Applying Theorem \ref{thm:bousquet_talagrand} yields with probability larger than $1-\exp{(-(x_m + y))}$: 
\begin{align}
\label{talagrand-model-selection} V_{m} \le \E\left[V_{m}\right] + \sqrt{\frac{2\left(\omega^2(y_{m})y^{-2}_{m} + 4\E\left[V_{m}\right]\right)(x_{m} + y)}{ny^2_{m}}} + \frac{2(x_{m}+y)}{3ny^2_{m}}
\end{align}
Now we bound $E(V_{m})$: 
\begin{align*}
E(V_{m}) & \le  E\left[\sup_{t \in S_m} \frac{\left[\bar{\gamma}_n(s_{m'}) - \bar{\gamma}_n(t)\right]}{\ell(s, t)+\ell(s, s_{m'}) + y^2_{m}}\right] \\
& \le E\left[\sup_{t \in S_m} \frac{\left|\bar \gamma_n(s_m) - \bar \gamma_n(t)\right|}{\ell(s,t) + \ell(s, s_{m'}) + y^2_{m}}\right] + E\left[\sup_{t \in S_m} \frac{\left|\bar \gamma_n(s_m) - \bar \gamma_n(s_{m'})\right|}{\ell(s,t) + \ell(s, s_{m'}) + y^2_{m}}\right]  \\
& = T_1 + T_2 \hspace{0.2in} [Say]
\end{align*}
For the next analysis, define $\omega_1 = 1 \wedge \omega$. We first analyze $T_2$ as follows: 
\begin{align}
T_2 = E\left[\sup_{t \in S_m} \frac{\left|\bar \gamma_n(s_m) - \bar \gamma_n(s_{m'})\right|}{\ell(s,t) + \ell(s, s_{m'}) + y^2_{m}}\right] & = E\left[\frac{\left|\bar \gamma_n(s_m) - \bar \gamma_n(s_{m'})\right|}{\inf_{t \in S_m} \left(\ell(s,t) + \ell(s, s_{m'}) + y^2_{m}\right)}\right] \notag \\
& = E\left[\frac{\left|\bar \gamma_n(s_m) - \bar \gamma_n(s_{m'})\right|}{\ell(s,s_m) + \ell(s, s_{m'}) + y^2_{m}}\right] \notag\\
& \le \frac{\sqrt{Var\left(\gamma(s_m, \cdot) - \gamma(s_{m'}, \cdot)\right)}}{\sqrt{n}\left(\ell(s,s_m) + \ell(s, s_{m'}) + y^2_{m}\right)} \notag\\
& \le \frac{1 \wedge (s_m, s_{m'})}{\sqrt{n}\left(\ell(s,s_m) + \ell(s, s_{m'}) + y^2_{m}\right)} \hspace{0.2in} [\because \gamma \in [0,1]] \notag\\
& \le  \frac{1 \wedge d(s, s_{m'}) + 1 \wedge d(s,s_m)}{\sqrt{n}\left(\ell(s,s_m) + \ell(s, s_{m'}) + y^2_{m}\right)} \notag\\
& \le \frac{1}{\sqrt{n}}\left[\frac{1 \wedge d(s,s_{m'})}{\ell(s, s_{m'}) + y^2_{m}} + \frac{1 \wedge d(s,s_m)}{\ell(s,s_m) + y^2_{m}}\right] \notag\\
& \le \frac{2}{\sqrt{n}}\sup_{\epsilon \ge 0}\frac{\omega_1(\epsilon)}{\epsilon^2 + y^2_{m}}  \notag\\
\label{T2} & \le \frac{2\omega_1(y_{m})}{\sqrt{n}y^2_{m}}
\end{align}
Next we analyze $T_1$ using Lemma \ref{miwep}: 
\begin{align*}
E\left[\sup_{t \in S_m} \frac{\left|\bar \gamma_n(s_m) - \bar \gamma_n(t)\right|}{\ell(s,t) + \ell(s, s_{m'}) + y^2_{m}}\right] & \le E\left[\sup_{t \in S_m} \frac{\left|\bar \gamma_n(s_m) - \bar \gamma_n(t)\right|}{\ell(s,t) \vee \ell(s, s_{m'}) + y^2_{m}}\right]\\
& \le E\left[\sup_{t \in S_m} \frac{\left|\bar \gamma_n(s_m) - \bar \gamma_n(t)\right|}{a^2(t) + y^2_{m}}\right]
\end{align*}
where we define $a^2(t) = \ell(s,t) \vee \ell(s, s_{m'})$. We can relate to $a(t)$ to $d(t,s_m)$ in the following way: 
\begin{align*}
a^2(t) \le \epsilon^2 \Rightarrow  \ell(s,t) \le \epsilon^2 & \Rightarrow \frac{1}{2}(\ell(s,t) +\ell(s,t)) \le \epsilon^2 \\
& \Rightarrow  \frac{1}{2}(\ell(s,s_m) +\ell(s,t)) \le \epsilon^2 \\
& \Rightarrow  (\ell(s,s_m) \vee \ell(s,t)) \le 2\epsilon^2 \\
& \Rightarrow \sqrt{\ell(s,s_m)} \vee \sqrt{\ell(s,t)} \le \sqrt{2}\epsilon \\
& \Rightarrow \omega\left(\sqrt{\ell(s,s_m) }\right) \vee \omega\left(\sqrt{\ell(s,t)}\right) \le \omega(\sqrt{2}\epsilon) \\
& \Rightarrow d(s,s_m) \vee d(s,t) \le \omega(\sqrt{2}\epsilon) \\
& \Rightarrow 2\left[d(s,s_m) \vee d(s,t)\right] \le 2\omega(\sqrt{2}\epsilon) \\
& \Rightarrow d(s,s_m) + d(s,t)\le 2\omega(\sqrt{2}\epsilon) \\
& \Rightarrow d(s_m,t) \le 2\omega(\sqrt{2}\epsilon) 
\end{align*} 
Hence we have: 
\begin{align*}
E\left[\sup_{a(t) \le \epsilon} \left|\bar \gamma_n(s_m) - \bar \gamma_n(t)\right|\right] & \le E\left[\sup_{d(t,s_m) \le 2\omega(\sqrt{2}\epsilon)} \left|\bar \gamma_n(s_m) - \bar \gamma_n(t)\right|\right] \\ 
& \le \frac{\phi_{m}(2\omega(\sqrt{2}\epsilon))}{\sqrt{n}} \equiv \psi(\epsilon)
\end{align*}
Using Lemma \ref{miwep} we conclude: 
\begin{equation}
\label{T1}
T_1 = E\left[\sup_{t \in S_m} \frac{\left|\bar \gamma_n(s_m) - \bar \gamma_n(t)\right|}{\ell(s,t) + \ell(s, s_{m'}) + y^2_{m}}\right] \le \frac{4\phi_{m}(2\omega(\sqrt{2}y_{m}))}{\sqrt{n}y^2_{m}}
\end{equation}
Combining equation \ref{T1} and \ref{T2} we conclude that: 
\begin{align}
E(V_{m})  & \le \frac{4\phi_{m}(2\omega(\sqrt{2}y_{m}))}{\sqrt{n}y^2_{m}} + \frac{2\omega_1(y_{m})}{\sqrt{n}y^2_{m}} \notag\\
& \le \frac{4\phi_{m}(2\omega(\sqrt{2}\epsilon_{m}))}{\sqrt{n}y_{m}\epsilon_{m}} + \frac{2}{\sqrt{n}y_{m}}\sqrt{\frac{\omega_1^2(\epsilon_{m})}{\epsilon^2_{m}}} \hspace{0.2in} [y_m \ge \epsilon_m \text{ as } 2K > 1]\notag\\
& \le \frac{8\sqrt{2}\phi_{m}(\omega(\epsilon_{m}))}{\sqrt{n}y_{m}\epsilon_{m}} + \frac{2}{\sqrt{n}y_{m}}\sqrt{\frac{\phi^2_{m}\left(\omega_1(\epsilon_{m})\right)}{\epsilon^2_{m}}} \hspace{0.2in} [\because \omega_1 \le 1]\notag\\
& \le \frac{8\sqrt{2}\epsilon_{m}}{y_{m}} + \frac{2}{\sqrt{n}y_{m}}\sqrt{\frac{\phi^2_{m}\left(\omega(\epsilon_{m})\right)}{\epsilon^2_{m}}}  \hspace{0.2in} [\because \phi_m \text{ is } \uparrow, \omega_1 \le \omega]\notag\\
& \le \frac{8}{\sqrt{K}} + \frac{2}{\sqrt{n}y_{m}}\frac{\phi_{m}\left(\omega(\epsilon_{m})\right)}{\epsilon_{m}}\notag \notag\\ 
\label{bound-on-expectation} & \le \frac{8}{\sqrt{K}} + \frac{2\epsilon_{m}}{y_{m}} \le \frac{8+\sqrt{2}}{\sqrt{K}} \hspace{0.2in} [\because \phi_{m}\left(\omega(\epsilon_{m})\right) = \sqrt{n}\epsilon_m^2]
\end{align}
Putting this bound in equation \ref{talagrand-model-selection} we have: 
\begin{align*}
V_{m} & \le  \frac{8+\sqrt{2}}{\sqrt{K}} + \sqrt{\frac{2\left(\omega^2(y_{m})y^{-2}_{m} + 4\frac{8+\sqrt{2}}{\sqrt{K}}\right)(x_{m} + y)}{ny^2_{m}}} + \frac{2(x_{m}+ y)}{3ny^2_{m}} \\
& \le \frac{8+\sqrt{2}}{\sqrt{K}} + \sqrt{\frac{2\left(\omega^2(y_{m})y^{-2}_{m} + 4\frac{8+\sqrt{2}}{\sqrt{K}}\right)(x_{m} + y)}{ny^2_{m}}} + \frac{\epsilon^2_{m}}{3K \omega^2(\epsilon_{m})} \\
& \le \frac{8+\sqrt{2}}{\sqrt{K}} + \sqrt{\frac{2\left(\omega^2(\epsilon_{m})\epsilon^{-2}_{m} + 4\frac{8+\sqrt{2}}{\sqrt{K}}\right)(x_{m} + y)}{ny^2_{m}}} + \frac{1}{3K} \\
& \le  \frac{8+\sqrt{2}}{\sqrt{K}} + \sqrt{\frac{2\left(\omega^2(\epsilon_{m})\epsilon^{-2}_{m}\right)(x_{m} + y)}{ny^2_{m}} + \frac{4(8+\sqrt{2})(x_{m} + y)}{\sqrt{K} ny^2_{m}}} + \frac{1}{3K} \\
& \le  \frac{8+\sqrt{2}}{\sqrt{K}}+ \sqrt{\frac{2\left(\omega^2(\epsilon_{m})\epsilon^{-2}_{m}\right)(x_{m} + y)}{ny^2_{m'}} + \frac{2(8+\sqrt{2})}{K\sqrt{K}}} + \frac{1}{3K} \\
& \le \frac{8+\sqrt{2}}{\sqrt{K}} + \sqrt{\frac{1}{K} + \frac{2(8+\sqrt{2})}{K\sqrt{K}}} + \frac{1}{3K}
\end{align*}
For large enough $K$, clearly $V_{m} \le 1/2$. This completes the proof. $\Box$

% \noindent
% Now if we choose $$pen(m) \ge \frac{\kappa}{2}\left(\epsilon^2_{n,m}+\frac{\omega^2(\epsilon_{n,m})x_m}{n\epsilon^2_{n,m'}}\right)$$ then we have from equation \ref{pen-choice}: 

% \allowdisplaybreaks
% \begin{align*}
% \frac{1}{2}\ell(\hat \theta_{\hat m}, \theta_0) & \le \frac{3}{2}\ell(\theta_{m}, \theta_0) + \frac{1}{2}y^2_{\hat m} - pen(\hat m) + pen(m) \\
% & \le \frac{3}{2}\ell(\theta_{m}, \theta_0) + pen(m) + \frac{\kappa \omega^2(\epsilon_{n, \hat m})}{2n\epsilon^2_{n, \hat m}}\xi \\
% \frac{1}{2}\ell(\hat \theta_{\hat m}, \theta_0) & \le \frac{3}{2}\ell(\theta_{m}, \theta_0) + \frac{1}{2}y^2_{\hat m} - pen(\hat m) + pen(m) \\
% & \le \frac{3}{2}\ell(\theta_{m}, \theta_0) + pen(m) + \frac{\kappa \phi^2_{\hat m}\left(\omega(\epsilon_{n, \hat m})\right)}{2n\epsilon^2_{n, \hat m}}\xi 
% \end{align*}

\subsection{Discussion on Lemma \ref{distance}} 
\label{Convex geometry}
%In this sub-section we will discuss Lemma \ref{distance} in little bit details. Consider the following lemma:
\begin{lemma}
\label{newangle}
For any fixed $x \in S^{p-1}$, define $C(x,\epsilon)$ to be $\epsilon$-angular spherical cap around $x$, i.e. $$C(x, \epsilon) = \{y \in S^{p-1}: \langle x,y \rangle \ge \epsilon\}$$ Then we have $$\sigma(C(x, \epsilon)) \le  \frac{1}{2\epsilon\sqrt{p}}(1-\epsilon^2)^{\frac{p-1}{2}} \le \frac{1}{2\sqrt{2}}(1-\epsilon^2)^{\frac{p-1}{2}}$$ for $\sqrt{\frac{2}{p}} \le \epsilon \le 1$. The last inequality follows from the assumption $\sqrt{\frac{2}{p}} \le \epsilon$.
\end{lemma}

This lemma is a well-known fact in convex geometry. Note that, Lemma \ref{distance} and Lemma \ref{newangle} are in different scale as one of them involves the angle and the other one involves the distance. In the following Lemma we bridge this gap:
\begin{lemma}
\label{angle-distance}
For $0 \le r \le 1$ and $p \ge 8$, we have: 
$$\sigma(D(x, r)) \le \frac{1}{2\sqrt{2}}r^{p-1} \,.$$
\end{lemma}
\begin{proof}
Note that $C(x,\epsilon) = D(x,r)$ where $\epsilon = (1-r^2/2)$. If $r \le 1$ and $p \ge 8$ then $\epsilon \ge \sqrt{\frac{2}{p}}$. Hence we have:
\allowdisplaybreaks
\begin{align*}
\sigma( D(x,r)) & \le \frac{1}{2\sqrt{2}}\left(1-\left(1-\frac{r^2}{2}\right)^2\right)^{\frac{p-1}{2}} \le \frac{1}{2\sqrt{2}}r^{p-1}\left(1 - \frac{r^2}{4}\right)^{\frac{p-1}{2}} \le \frac{1}{2\sqrt{2}}r^{p-1}.
\end{align*}
which completes the proof.
\end{proof}
Finally using Lemma \ref{angle-distance} we get the upper bound on $\sigma(D(x, r))$. The lower bound can also be found in convex geometry literature. Combining them together, we get Lemma \ref{distance}.

\subsection{Proof of Lemma \ref{Hellinger}}\label{Hellinger Proof}
Define $x - (p_1 - q_1)/2 = \nu/2$. From the definition of Hellinger distance between two Bernoulli Random variables, we get,
\allowdisplaybreaks
\begin{align}
H^2(P_1, P_2) & = 1 - \sqrt{p_1q_1} - \sqrt{(1-p_1)(1-q_1)} \notag\\
& = 1 - \sqrt{(s+x)(s-x)} - \sqrt{(1-s+x)(1-s-x)} \notag \\
& = 1 - \sqrt{s^2 - x^2} - \sqrt{(1-s)^2 - x^2} \notag\\
& = 1 - s\sqrt{1- \frac{x^2}{s^2}} - (1-s)\sqrt{1-\frac{x^2}{(1-s)^2}} \notag \\
& = 1- s \left[1 -\frac{x^2}{2s^2}\left(1-\frac{\tilde{x}_1^2}{s^2}\right)^{-1/2}\right] - (1-s) \left[1 -\frac{x^2}{2(1-s)^2}\left(1-\frac{\tilde{x}_2^2}{(1-s)^2}\right)^{-1/2}\right] \notag \\
\label{hellinger} & = \frac{x^2}{2s}\left(1-\frac{\tilde{x}_1^2}{s^2}\right)^{-1/2} + \frac{x^2}{2(1-s)}\left(1-\frac{\tilde{x}_2^2}{(1-s)^2}\right)^{-1/2}.
\end{align}
In the second last line we use mean value theorem: $$f(x)= \sqrt{1-x} = 1 - \frac{x}{2}(1-\tilde{x})^{-1/2}$$ for some $\tilde{x}$ between $0$ and $x$. As our parameter space is $[1/4, 3/4]$, we have $p_1 \le 3q_1$ for any choice of $p_1, q_1$. Hence, $\frac{|x|}{s} \le \frac{1}{2}$ and $\frac{|x|}{1-s} \le \frac{1}{2}$ which immediately implies $\frac{\tilde{x}_1^2}{s^2} \le \frac{1}{4}$ and $\frac{\tilde{x}_2^2}{(1-s)^2} \le \frac{1}{4}$, which, in turn, validates $\left(1-\frac{\tilde{x}_1^2}{s^2}\right)^{-1/2} \le \frac{2}{\sqrt{3}}$ and $\left(1-\frac{\tilde{x}_2^2}{(1-s)^2}\right)^{-1/2} \le \frac{2}{\sqrt{3}}$.  Using this in equation \ref{hellinger} we conclude:
\begin{equation*}
\begin{split}
\frac{x^2}{2s}\left(1-\frac{\tilde{x}_1^2}{s^2}\right)^{-1/2} + \frac{x^2}{2(1-s)}\left(1-\frac{\tilde{x}_2^2}{(1-s)^2}\right)^{-1/2} & \le \frac{2}{\sqrt{3}}\left[\frac{x^2}{2s}+\frac{x^2}{2(1-s)}\right] \\
& = \frac{(q_1 - p_1)^2}{4\sqrt{3}s(1-s)}.
\end{split}
\end{equation*}

\subsection{Proof of Lemma \ref{KL}}\label{Proof of KL} 
\begin{align*}
KL(P||Q) & = p_1\log{\frac{p_1}{q_1}} + (1-p_1)\log{\frac{1-p_1}{1-q_1}} \\
& \le \frac{p_1}{q_1}(p_1 - q_1) + \frac{1-p_1}{1-q_1}(q_1 - p_1) \hspace*{0.3in} [\because \log{x} \le x-1]\\
& =  (p_1 - q_1)\left[\frac{p_1}{q_1} - \frac{1-p_1}{1-q_1}\right] \\
& = \frac{(p_1 - q_1)^2}{q_1(1-q_1)} \le \frac{16}{3}(p_1 - q_1)^2 \hspace*{0.3in} \left[\because \frac{1}{4} \le q_1 \le \frac{3}{4}\right].
\end{align*}

%\subsection{Discussion on Assumption A2:upper} In this subsection we at first show that under mild assumption on $\eta(x)$, one can establish the quadratic upper bound on $S(\beta) - S(\beta^0)$. 

\section{A discussion of the model with intercept}
\label{Discussion of intercept}
The binary choice model in the presence of intercept can be formulated as follows: 
\begin{enumerate}
\item $(X,\epsilon) \overset{i.i.d.} \sim P$ with $\med(\epsilon|X) = 0$ almost surely. 
\item $Y = \s(Y^*)$ where $Y^* = \tau_0 + X^T\beta^0 + \epsilon$.
\end{enumerate}
The maximum score estimator can be defined as: $$(\hat{\tau}, \hat{\beta}) = \argmax_{\tau, \beta} S_n(\tau, \beta) = \argmax_{\tau, \beta}  \frac{1}{n} \sum_{i=1}^n Y_i \ \s(\tau + X_i^T\beta)$$ with the population score function being $S(\tau, \beta) = \E(Y\s(\tau + X^T\beta))$. In this model we can write the function $\eta(X) = \P(Y = 1|X) = 1 - F_{\epsilon|X} (-\tau - X^T\beta)$.
\noindent
We take our parameter space to be $\{(\tau, \beta): \tau \in (-U, U)\,,\,\|\beta\|_2  = 1\}$. For notational simplicity, define $\gamma = (\tau, \beta)$. The transition assumption (Assumption \ref{ass:low_noise_binary}) remains unchanged under the intercept model. Assumption \ref{ass:wedge_binary} can be generalized for this model as follows: 
\begin{assumption}
\label{ass:wedge_intercept_binary}
For all $\gamma$ sufficiently close to $\gamma^0$, 
$$a^-\|\gamma - \gamma^0\|_2 \le \P(\s(\tau + \beta^TX) \neq \s(\tau^0 + X^T\beta^0)) \,.$$
\end{assumption}
Consider the linear transformation $Y = PX$ where: 
\allowdisplaybreaks
\[
P_{\beta} = P = \begin{bmatrix}
\frac{\beta^{0'} + \beta'}{\|\beta + \beta^0\|_2} \\
\frac{\beta^{0'} - \beta'}{\|\beta - \beta^0\|_2} \\
v_3 \\
\vdots \\
v_{p}
\end{bmatrix}
\]
with $v_3, \cdots, v_p$ being orthogonal extensions to a basis of $\mathbb{R}^p$. Note that $Y$ depends on $\beta$, but this will be suppressed in the notation. The following lemma presents conditions on the distribution of $X$ under which Assumption \ref{ass:wedge_intercept_binary} is valid. 

\begin{lemma}
\label{lower bound}
Suppose there exists $0 < \delta < 2$ and a constant $K$ such that $f_{Y_1, Y_2}(y_1, y_2) \ge F$ for all $\{(y_1, y_2): \|(y_1, y_2)\|_2 \le 2U/\zeta\}$ where $\zeta = \sqrt{1-\delta^2/4}$ and the bound $F = F(U, \zeta)$ is independent of $\beta$ and the dimension $p$. Then 
$$a^-\|\gamma - \gamma^0\|_2 \le \P(\s(\tau + \beta^TX) \neq \s(\tau^0 + X^T\beta^0)) $$ 
holds for $\tau \in (-U,U)$ and for all $\beta: \|\beta - \beta^0\|_2 \le \delta$.  
\end{lemma}
 As the wedge condition is only valid in a neighborhood of the true $\gamma^0$, we need to establish the consistency of the maximum score estimator in order to prove the rate of convergence results. 

\begin{lemma}
\label{consistency}
Under Assumption \ref{ass:low_noise_binary} and Assumption \ref{ass:wedge_intercept_binary} we have $$\|\hat{\gamma} - \gamma^0\|_2 \overset{P} \to 0$$ when $p/n \rightarrow 0$. Furthermore, under Assumption \ref{ass:sparsity_binary}, the result continues to hold when $p \gg n$. 
\end{lemma}

We next argue that the rate of convergence results in (Theorem \ref{rate-manski-p-less-n} and Theorem \ref{high-dim-rate}) hold for the intercept model by slight modifications to the previous proofs. 
\begin{theorem}
Under Assumption \ref{ass:low_noise_binary}, \ref{ass:wedge_intercept_binary} and \ref{ass:sparsity_binary} we have: 
$$\|\hat{\gamma} - \gamma^0\|_2 = O_P\left(\frac{r_n}{\sqrt{C_n}} \wedge r_n^2 \right) \,,$$
where 
$$r_n =  \left(\frac{p\sqrt{C_n}\log{(n/pC^2_{n})}}{n}\right)^{-1/3} \wedge \left(\frac{p\log{(n/p)}}{n}\right)^{-1/2} \,.$$ 
for the slowly growing regime $p/n \to 0$, and 
$$r_n =  \left(\frac{n}{V_{s_0}\sqrt{C_n}\log{(n/V_iC^2_{n})}}\right)^{2/3} \vee \left(\frac{n}{V_{s_0}\log{(n/V_{s_0})}}\right), \ \ \ V_{s_0} \sim s_0\log{(ep/s_0)} \,.$$
for the fast growing regime $p \gg n$. 
\end{theorem}

\noindent
Under the intercept model, our class of classifier is:
$$\mathcal{F} = \mathcal{F}_{\gamma} = \{f_{\gamma}: \mathbb{R}^p \rightarrow \{-1,1\}, f_{\gamma}(x) = \s(\tau + x^T\beta)\}$$

\noindent
The VC dimension of this class is $d+1$ (For $p \gg n$ the VC dimension is at most $(s_0 + 1)\log{(p+1)}$). By the same arguments as in the proof of Proposition \ref{alpha-kappa} we can show that 
$$S(\gamma^0) - S(\gamma) \ge \left[c_1^2\frac{\left\|\gamma - \gamma^0\right\|^2_2}{C_n}\mathds{1}_{\left(d_{\Delta}(\gamma, \gamma^0) \le 2t^*C_n\right)} + 2t^*c_1\left\|\gamma - \gamma^0\right\|_2\mathds{1}_{\left(d_{\Delta}(\gamma, \gamma^0) > 2t^*C_n\right)}\right]$$
As before, we next apply Theorem \ref{massart}. The first condition of the theorem remains valid as our parameter space $(-U, U) \times S^{p-1}$ admits a countable dense subset. The distance function $d_{\Delta}(f, f^*)$ changes to the following: 
$$d_{\Delta}(f_{\tau_1, \beta_1}, f_{\tau_2, \beta_2}) = \P^{1/2}(\s(\tau_1 + X^T\beta_1) \neq \s(\tau_2 + \beta_2)) \,.$$
The remainder of the proof remains completely unchanged as can be verified by inspection. 
\begin{remark}
It is not clear whether the minimax upper bound results in Theorems \ref{rate-manski-p-less-n} and \ref{high-dim-rate} hold. Recall that, to prove the minimax upper bound in these theorems, we used an exponential tail bound on the probability that $\|\hat{\beta} - \beta^0\| > t$ for every $t > 0$, derived via Theorem \ref{massart}, using the fact that the wedge condition Assumption \ref{ass:wedge_binary} held for all $\beta$. In the intercept model, the wedge condition only holds on a restricted part of the parameter space, and the exponential tail bound cannot be established for all $t$. Nevertheless, the minimax lower bound rates obtained in 
Theorems \ref{minimax-lower-bound} and \ref{minimax-high-lower} remain exactly the same, as we can take $\tau = 0$ in the minimax constructions that arise in their proofs. Of course, the space of distributions changes, as we have introduced the intercept. We can rewrite these results as follows. 
\end{remark} 

\begin{theorem}
\label{minimax-lower-bound-intercept}
For the slowly growing regime $p/n \to 0$, we have : 
$$\inf_{\hat{\gamma}_n}\sup_{\gamma \equiv \gamma(P)}\mathbb{E}_{\gamma}\left(\|\hat{\gamma}_n - \gamma^0\|_2^2\right) \ge K_{L}\left[\left(\frac{pC^2_{n}}{n}\right)^{2/3} \vee \left(\frac{p}{n}\right)^2\right]$$
for some constant $K_{L}$ that does not depend on $(n,p)$. For $C_{n} = C$ fixed, the lower bound is of the order $(p/n)^{2/3}$. The supremum is taken over all distributions $P$ corresponding to binary response models 
satisfying Assumptions \ref{ass:low_noise_binary} and \ref{ass:wedge_intercept_binary} for some regression parameter $\gamma \in (-U, U) \times \mathcal{S}^{p-1}$ (viewed as a functional of $P$) but with $t^*, a^-, C_{n}$ held fixed. 
\end{theorem}

\begin{theorem}
For the fast growth regime $p \gg n$, we have: 
$$\inf_{\hat{\gamma}} \sup_{\gamma \equiv \gamma(P)}\E_{\gamma} \left(\|\hat{\gamma} -\gamma^0\|^2_2\right) \ge \tilde{K}_{L}\left[\left(\frac{s_0\log{(p/s_0)}C^2_{n}}{n}\right)^{2/3} \vee \left(\frac{s_0\log{(p/s_0)}}{n}\right)^2\right]$$
for some constant $\tilde{K}_{L} > 0$ not depending on $(n,p,s_0)$. For the case $C_{n} = C$ fixed, the lower bound is of the order of $\left(\frac{s_0\log{(p/s_0)}}{n}\right)^{2/3}$. The supremum is taken over the same class of distributions as in Theorem \ref{minimax-lower-bound-intercept}. 
\end{theorem}

\subsubsection{Which distributions satisfy Assumption \ref{ass:wedge_intercept_binary} ?} As stated in lemma \ref{lower bound} we need the joint density of $(Y_1, Y_2)$ to be lower bounded by some non negative constant to establish the lower bound.  Here we show that, under fairly general restrictions, any elliptically symmetric distribution and satisfies the assumption. 
\begin{lemma}
Suppose the distribution of $X$ belongs to a consistent family of elliptical distribution with mean 0 i.e. the density of $X$ has the form: $$f_X(x) = |\Sigma_p|^{-\frac{1}{2}}g_p(x^T\Sigma_p^{-1}x)$$ with $\Sigma$ being a full rank matrix. If $g_2$ (density generator of two dimensional marginal of $X$) is decreasing function on $\mathbb{R}^+$ with $g_2(x) > 0$ for all $x$ and there exists constants $0 <\lambda^- < \lambda^+ < \infty$ such that $$\lambda^- \le \lambda_{min}(\Sigma_p) \le \lambda_{max}(\Sigma_p) \le \lambda^+$$ for all  $p$ then $X$ satisfies the assumption of Lemma \ref{lower bound}. 
\end{lemma}  
\begin{proof}
The density of $Y = PX$ is $f_Y(y) = |\bar{\Sigma}|^{-\frac{1}{2}}g(y^T\bar{\Sigma}^{-1}y)$ where $\bar{\Sigma} = P\Sigma P^T$. Then density of $(Y_1, Y_2)$ is $f_{Y_1, Y_2}(y_1, y_2) = |\Sigma_1|^{-\frac{1}{2}}g_2((y_1, y_2)^T\Sigma_1^{-1}(y_1,y_2))$, where $\Sigma_1$ is the leading $2 \times 2$ block of $\bar{\Sigma}$. Now, if we confine ourselves on a ball of radius $2U/\zeta$ then: 
\allowdisplaybreaks
\begin{align*}
f_{Y_1, Y_2}(y_1, y_2) & = |\Sigma_1|^{-\frac{1}{2}}g_2((y_1, y_2)^T\Sigma_1^{-1}(y_1,y_2)) \\
& \ge \frac{1}{\lambda_{max}(\Sigma)}g_2\left(\frac{\|(y_1, y_2)\|^2}{\lambda_{min}(\Sigma)}\right) \\
& \ge \frac{1}{\lambda^+}g_2\left(\frac{4U^2}{\zeta^2\lambda^-}\right)
\end{align*}
Hence $F(U,\zeta) =  \frac{1}{\lambda^+}g_2\left(\frac{4U^2}{\zeta^2\lambda^-}\right)$ and assumption (A2:intercept) is satisfied. 
\end{proof}

\begin{lemma}
\label{log_concave_lb}
Suppose the elements of the random vector $X = (X_1, \dots X_p)$ are independent and each component has a log concave density symmetric around 0 and variance 1. Then, there exists constants $\epsilon_0, R > 0$ such that $f_{Y_1, Y_2} (y_1, y_2) \ge \epsilon_0$ on a circle of radius $R$. Hence, Assumption (A2:intercept) is satisfied for all $\zeta$ such that $2U/\zeta \le R$.  
\end{lemma}
\begin{proof}
Denote the density of $X_i$ as $f_i$. From the strong unimodality property of log concave densities, each $f_i$ has mode at 0. Also we have 

\begin{align}
\label{coneqn1} \frac{1}{12} \le  Var(X_i)f_i^2(0) \le 1
\end{align}

\noindent
for all $i \in \{1,2, \dots, p\}$. [See equation (2.2)  of \cite{bobkov2015concentration}]. Hence $f_i(0) \ge 1/\sqrt{12}$ under variance = 1. Note that, as each component of $X$ has a symmetric strongly unimodal density, so does $a^TX$ for any $a \in \mathbb{R}^p$. Consider $Y = P_{\beta}X$ as defined before Lemma \ref{lower bound}. Let $Z_{\phi} = Y_1\cos(\phi) + Y_2\sin(\phi)$, then $Z_{\phi}$ is also strongly unimodal with mode at $0$ (Recall that density of linear combination of random variables with log concave density is also log concave and any symmetric log concave density has mode at 0). As marginals of log-concave density is log-concave, the density of $Y_1, Y_2$ is also log-concave i.e. $$f_{Y_1, Y_2}(y_1, y_2) = e^{g(y_1, y_2)} \ \forall \ y_1, y_2 \in \mathbb{R}$$ where $g$ is a log-concave function on $\mathbb{R}^2$ with mode at 0. Then by the Jacobian transformation: $$f_{Z_{\phi}}(0) = \int_{-\infty}^{\infty}e^{g(-x\sin(\phi), x\cos(\phi))} \ dx = \int_{-\infty}^{\infty} e^{g_{\phi}(x)} \ dx$$ where $g_{\phi}(x) = g(-x\sin(\phi), x\cos(\phi))$. Some properties of $g_{\phi}(x)$'s are immediate: 
\begin{enumerate}
\item $g_{\phi}$ is concave.
\item $g_{\phi}$ is symmetric around 0 as $g$ is symmetric around 0.  
\end{enumerate} 
As $(Y_1, Y_2)$ has a two dimensional log concave density in $\mathbb{R}^2$ and $\E(\|Y\|^2) = 2$, there exists an absolute constant $b$ such that $f_{Y_1,Y_2}(y_1,y_2) \le b \ \forall \ y_1, y_2 \in \mathbb{R}$. (e.g. see \cite{ball1988logarithmically}). Next, we show that, there exists a universal constant $\epsilon_0 > 0$ such that $e^{g(-x \sin(\phi), x\cos(\phi))} \ge \epsilon_0$ for all $|x| \le \frac{1}{2b\sqrt{13}}$, for all $\phi \in [0, 2\pi)$ which implies $f_{Y_1, Y_2}(y_1, y_2) \ge \epsilon_0$ on a circle $R = \frac{1}{2b\sqrt{13}}$. Fix $\phi \in [0, 2\pi)$. Denote $e^{g_{\phi}(1/2\sqrt{13}b)}$ by  $\epsilon$. Then, due to concavity, $g_{\phi}(x)$ lies below the line joining $(0, g_{\phi}(0))$ and $\left((1/2b\sqrt{13}), g_{\phi}(1/2b\sqrt{13})\right)$ i.e. $$g_{\phi}(x) \le g_{\phi}(0) - 2b\sqrt{13} \ x \ (g_{\phi}(0) - \log{\epsilon})$$ for $x > 1/(2b\sqrt{13})$. This implies: 
\begin{align*}
1/\sqrt{12} \le f_{Z_{\phi}(0)} & = \int_{-\infty}^{\infty} e^{g_{\phi}(x)} \ dx \\
& = 2 \int_{0}^{\infty} e^{g_{\phi}(x)} \ dx \\
& = 2 \left[ \int_{0}^{1/2b\sqrt{13}} e^{g_{\phi}(x)} \ dx + \int_{1/2b\sqrt{13}}^{\infty} e^{g_{\phi}(x)} \ dx \right] \\
& \le \frac{1}{\sqrt{13}} + 2e^{g_{\phi}(0)} \int_{1/2b\sqrt{13}}^{\infty} e^{- 2b\sqrt{13} \ x \ (g_{\phi}(0) - \log{\epsilon})} \ dx \\
& = \frac{1}{\sqrt{13}} + \frac{2e^{g_{\phi}(0)}}{2b\sqrt{13}(g_{\phi}(0) - \log{\epsilon})}e^{-(g_{\phi}(0) - \log{\epsilon})} \\
& =  \frac{1}{\sqrt{13}} + \frac{2\epsilon}{2b\sqrt{13}(g_{\phi}(0) - \log{\epsilon})} \\
& = \frac{1}{\sqrt{13}} + \frac{2\epsilon}{2b\sqrt{13}(1/4\pi - \log{\epsilon})} \,.\\
\end{align*}
The last equation follows from the lower bound on the mode of two dimensional log concave density (see Lemma 6 of \cite{ball1988logarithmically}). Hence $\frac{2\epsilon}{2b\sqrt{13}(1/4\pi - \log{\epsilon})} \ge (1/\sqrt{12} - 1/\sqrt{13})$. Define $\Psi(s) = \frac{2s}{2b\sqrt{13}(1/4\pi - \log{s})}$. As $\Psi(s)$ is strictly increasing on $(0,1)$ we conclude $\epsilon \ge  \Psi ^{-1} (1/\sqrt{12} - 1/\sqrt{13}) = \epsilon_0$. This immediately implies $e^{g_{\phi}(x)} \ge \epsilon_0$ for $|x| \le 1/2\sqrt{13}b$ from the fact that $g_{\phi}(x)$ has mode 0. Also the value of $\epsilon_0$ does not depend on $\phi$, which implies, $e^{g_{\phi}(x)} \ge \epsilon_0 \ \forall \ |x| \le 1/2b\sqrt{13} \ \forall \phi \in [0, 2\pi)$. This completes the proof with $R = 1/2b\sqrt{13}$. 
\end{proof}

\begin{comment}
In the next lemma, we prove that Assumption (B3) and Assumption (B4) is satisfied under the same setting as Lemma \ref{log_concave_lb}: 
\begin{lemma}
\label{log_concave_ub}
Under the same setting as previous lemma, i.e. the components of the random vector $X = (X_1, \dots X_p)$ are independent and each component has a log concave density symmetric around 0 with variance being 1, Assumption (B3) and (B4) is satisfied. 
\end{lemma}
\begin{proof}
Assumption (B3) follows directly from \ref{coneqn1} as the density $a_1Y_1 + a_2Y_2 = X^T\beta^0$ is log-concave with variance being 1, which gives $d_1 = 1/2$. For assumption (B4), we use the following result: 
\begin{lemma}
If $f$ is a log-concave density in $R^d$, then there exists $a>0, b \in \mathbb{R}$ such that $f(x) \le e^{-a\|x\|_2 + b}$ for all $x \in \mathbb{R}^d$. 
\end{lemma} 
For some discussion and the proof of the lemma see \cite{cule2010theoretical}. Hence in equation $\ref{sphdom}$ we can use $f_{Z_1, Z_2}(z_1, z_2) = \frac{e^{-a\|y\|+b}}{D}$ where $D = \int_{\mathbb{R}^2}e^{-a\|y\|+b} \ dy_1dy_2$. 
\end{proof}
\end{comment}

\subsection{Proof of Lemma \ref{consistency}}
First we show that as $p/n \rightarrow 0$, $$\sup_{\gamma \in (-U, U) \times S^{p-1}} \|S_n(\gamma) - S(\gamma)\|_2 \overset{P} \to  0$$ i.e. the class of functions $\mathcal{G} = \mathcal{G}_p = \{g_{\gamma}: \R^p \times \{-1,1\} \rightarrow \{-1,1\}, g_{\gamma}(x,y) = y \s(\tau + x^T{\beta})\}$ is Glivenko-Cantelli class which is equivalent to showing (for details see \cite{pollard1981limit}):
\begin{enumerate}
\item There exists $G$, an envelope of $\mathcal{G}$ such that $P^*G \le \infty$.
\\
\item $\lim_{n \rightarrow \infty} \frac{E^*\left(\log{(N(\epsilon, \mathcal{G}_{m}, L_2(\P_n)))}\right)}{n} = 0$ for all $M < \infty, \epsilon > 0$, where $N(\epsilon, \mathcal{G}_{m}, L_2(\P_n))$ is the $\epsilon$ covering number of the set $\mathcal{G}_m = \{g_{\beta}\mathds{1}_{G \le M}: g_{\beta} \in \mathcal{G}\}$ with respect to $L_2(\P_n)$ norm.
\end{enumerate}
Clearly $G \equiv 1$ is an integrable envelope of $\mathcal{G}$. Now $\mathcal{G}$ is VC class of VC dimension $v = (p+1)$. Hence, we have:$$\sup_{Q} N(\epsilon, \mathcal{G}, L_2(Q)) \le Kv\left(\frac{4\sqrt{e}}{\epsilon}\right)^{2v}$$ for some universal constant $K$ and $0 \le \epsilon \le 1$. Using this, we have: $$ \frac{E^*\left(\log{(N(\epsilon, \mathcal{G}_{m}, L_2(\P_n)))}\right)}{n} \le \frac{\log{(kv)}}{n} + \frac{2v}{n}\log{\left(\frac{4\sqrt{e}}{\epsilon}\right)} \rightarrow 0$$ if $v/n \rightarrow 0 \iff p/n \rightarrow 0$ which completes the proof.

In the previous step we have established that $S_n(\gamma) \rightarrow S(\gamma)$ uniformly over $\gamma$. Now we need to prove $\hat{\gamma} = \argmax_{\gamma} S_n(\gamma)$ converges to $\gamma^0 = \argmax_{\gamma}S(\gamma)$. Towards that we need the following Lemma:

\begin{lemma}
\label{beta-distance}Given any $0 \le \epsilon_1 < \epsilon_2 \le 2$ and $\gamma_1$ such that $\|\gamma_1 - \gamma^0\|_2 = \epsilon_2$, we can find $\gamma_2$ with $\|\gamma_2 - \gamma^0\|_2 \le \epsilon_1$ such that $$S(\gamma^0) -S(\gamma_1) \ge S(\gamma^0) - S(\gamma_2)$$
\end{lemma}
\noindent
We defer the proof of this lemma to the next subsection. Using the same proof as Proposition \ref{alpha-kappa} we have:
$$S(\gamma^0) - S(\gamma) \ge \left[c_1^2\frac{\left\|\gamma - \gamma^0\right\|^2_2}{C_n}\mathds{1}_{\left(d_{\Delta}(\gamma, \gamma^0) \le 2t^*C_n\right)} + 2t^*c_1\left\|\gamma - \gamma^0\right\|_2\mathds{1}_{\left(d_{\Delta}(\gamma, \gamma^0) > 2t^*C_n\right)}\right]$$
which is now true for $\|\gamma - \gamma^0\|_2 \le \delta$ under the assumptions of Theorem \ref{consistency}. Suppose $0 \le \epsilon < \delta$, then using Lemma \ref{beta-distance} we have: 
\begin{align*}
\inf_{\gamma: \|\gamma - \gamma^0\|_2 > \epsilon}S(\gamma^0)-S(\gamma) 
& = \inf_{\epsilon < \|\gamma - \gamma^0\| \le \delta} S(\gamma^0) - S(\gamma) \\
& \ge c_1^2 \frac{\epsilon^2}{C_n} \wedge 2t^*c_1 \epsilon
\end{align*}
\allowdisplaybreaks
\begin{align*}
\P\left(\|\hat \gamma - \gamma^0\|_2 > \epsilon\right) & = \P\left(\sup_{\|\gamma - \gamma^0\|_2 > \epsilon}\left(S_n(\gamma) - S_n(\gamma^0)\right) > 0\right) \\
& =  \P\left(\sup_{\|\gamma - \gamma^0\|_2 > \epsilon}\left((S_n-S)(\gamma) -(S_n -S)(\gamma^0) + S(\gamma - \gamma^0) \right) > 0\right) \\
& \le \P\left(\sup_{\|\gamma - \gamma^0\|_2 > \epsilon}\left((S_n-S)(\gamma) -(S_n -S)(\gamma^0)\right) >  \inf_{\|\gamma - \gamma^0\|_2 > \epsilon} \left(S(\gamma^0) - S(\gamma) \right)\right) \\
& \le \P\left(\sup_{\|\gamma - \gamma^0\|_2 > \epsilon}\left((S_n-S)(\gamma) -(S_n -S)(\gamma^0)\right) >  \inf_{\epsilon < \|\gamma - \gamma^0\| \le \delta} \left(S(\gamma^0) - S(\gamma)\right)\right) \\
& \le \P\left(\sup_{\|\gamma - \gamma^0\|_2 > \epsilon}\left((S_n-S)(\gamma) -(S_n -S)(\gamma^0)\right) >  c_1^2 \frac{\epsilon^2}{C_n} \wedge 2t^*c_1 \epsilon\right) \\
& \rightarrow 0 \hspace*{0.1in}[\because \mathcal{G} \,\, \text{is a GC Class}]
\end{align*}
which completes the proof for $p/n$ going to 0.

The same proof works for $p \gg n$ under our assumption $(s_0\log{p})/n \rightarrow 0$, because, what is really needed in the above proof is the condition $V/n \rightarrow 0$ where $V$ is the VC dimension of the set of classifiers under consideration. When $p \gg n$, $V = O(s_0\log{p})$ under the sparsity assumption, and therefore by our assumption $V/n \rightarrow 0$ in this case as well. $\Box$

\subsection{Proof of Lemma \ref{beta-distance}} Under the assumption that $\med(\epsilon|X) = 0$ in our model, we have for any $\beta$: $$S(\gamma) - S(\gamma^0) = 2\int_{X_{\gamma}}|\E(Y|X)| \ dF_X$$ where $X_{\gamma} = \{x: \s(\tilde{x}^T\gamma) \neq \s(\tilde{x}^T\gamma^0)\}$ with $\tilde{x} = (1, x^T)^T$. Now a fix $\gamma_1$ with $\|\gamma_1 - \gamma^0\|_2 = \epsilon_1$. Define $\gamma_2 = \frac{\lambda \gamma_1 + (1 - \lambda)\gamma^0}{\|\lambda \beta_1 + (1 - \lambda)\beta^0\|_2}$ for some $\lambda \in (0,1/2)$ which will be chosen later. Suppose $x \in X_{\gamma_2}$: 
\\\\
{\bf Case 1: } Suppose $x^T\gamma_2 > 0 > x^T\gamma^0$. Then $$\frac{\lambda}{\|\lambda \beta_1 + (1 - \lambda)\beta^0\|_2}x^T\gamma_1 = \gamma_2^Tx - \frac{1-\lambda}{\|\lambda \beta_1 + (1 - \lambda)\beta^0\|_2}x^T\gamma^0 > 0 \iff x^T\gamma_1 > 0$$
\noindent
{\bf Case 2: } Suppose $x^T\gamma_2 < 0 < x^T\gamma^0$. Then $$\frac{\lambda}{\|\lambda \beta_1 + (1 - \lambda)\beta^0\|_2}x^T\gamma_1 = \gamma_2^Tx - \frac{1-\lambda}{\|\lambda \beta_1 + (1 - \lambda)\beta^0\|_2}x^T\gamma^0 < 0 \iff x^T\gamma_1 < 0$$
\noindent
Hence $X_{\gamma_2} \subseteq X_{\gamma_1}$. Now $\|\lambda \beta_1 + (1 - \lambda)\beta^0\|_2 \ge (1 -2\lambda)$ by triangle inequality and using the fact that $\|\beta_1\| = \|\beta^0\| = 1$. Therefore, 
\allowdisplaybreaks
\begin{align*}
\|\gamma_2 - \gamma^0\|_2 & = \left\|\frac{\lambda \gamma_1 + (1 - \lambda)\gamma^0}{\|\lambda \beta_1 + (1 - \lambda)\beta^0\|_2} - \gamma^0\right\|_2 \\
& = \left\|\frac{\lambda (\gamma_1-\gamma^0) + (1 - \|\lambda \beta_1 + (1 - \lambda)\beta^0\|_2)\beta^0}{\|\lambda \beta_1 + (1 - \lambda)\gamma^0\|_2}\right\|_2 \\
& \le \frac{\lambda (\epsilon_1+2)}{1-2\lambda}
\end{align*}
To conclude the proof we choose $\lambda$ such that $\lambda (\epsilon_1+2)/ (1 - 2\lambda) = \epsilon_2$ i.e. $\lambda = \epsilon_2/(\epsilon_1 + 2 + 2\epsilon_2)$.

\subsection{Proof of Lemma \ref{lower bound}} 
From the transformation $Y = PX$, we can write $a_1Y_1 + a_2Y_2 = X^T\beta^0$ and $X^T\beta = a_1Y_1 - a_2Y_2$ where $a_1 = \frac{1}{2}\|\beta + \beta^0\|_2, a_2 = \frac{1}{2}\|\beta - \beta^0\|_2$. We divide the proof into three cases: 
\\\\
\textbf{Case 1:} Suppose $\tau \neq \tau^0, \beta \neq \beta^0$. The probability of the wedge shaped region can be written as: $$\P(a_1Y_1 + a_2Y_ 2 \ge -\tau^0, a_1Y_1 - a_2 Y_2 \le - \tau) + \P(a_1Y_1 + a_2Y_2 \le - \tau^0, a_1Y_1 - a_2Y_2  \ge - \tau)$$ which is the probability of the region between the straight lines: $a_1Y_1 + a_2Y_2 + \tau^0 = 0$ and $a_1Y_1 - a_2Y_2 + \tau = 0$. The intersection of these two lines is $I = (-(\tau + \tau^0)/2a_1, (\tau - \tau^0)/2a_2)$, the line $a_1Y_1 + a_2Y_2 + \tau^0 = 0$ meets the $X$-axis at $J = (-\tau_0/a_1, 0)$ and the line $a_1Y_1 - a_2Y_2 + \tau = 0$ meets the $X$-axis at $K = (-\tau/a_1, 0)$. From our assumptions $\|\beta - \beta^0\|_2 \le \delta $ we have $a_1 = \frac{1}{2}\|\beta + \beta^0\|_2 \ge \sqrt{1 - \delta^2/4} = \zeta$ (say). Hence, $|\tau|/a_1 \le U/\zeta$ for all $\tau$, indicating that the intersection points with the $X$- axis (denoted by J,K) lie within a circle of radius $2U/\zeta$ around origin. 
\\\\
{\bf Case 1.1: }Suppose the point I is inside the circle of radius $2U/\zeta$. The points $J,K$ are inside by definition. 
\begin{figure}[H]
\centering
\includegraphics[scale = 0.4]{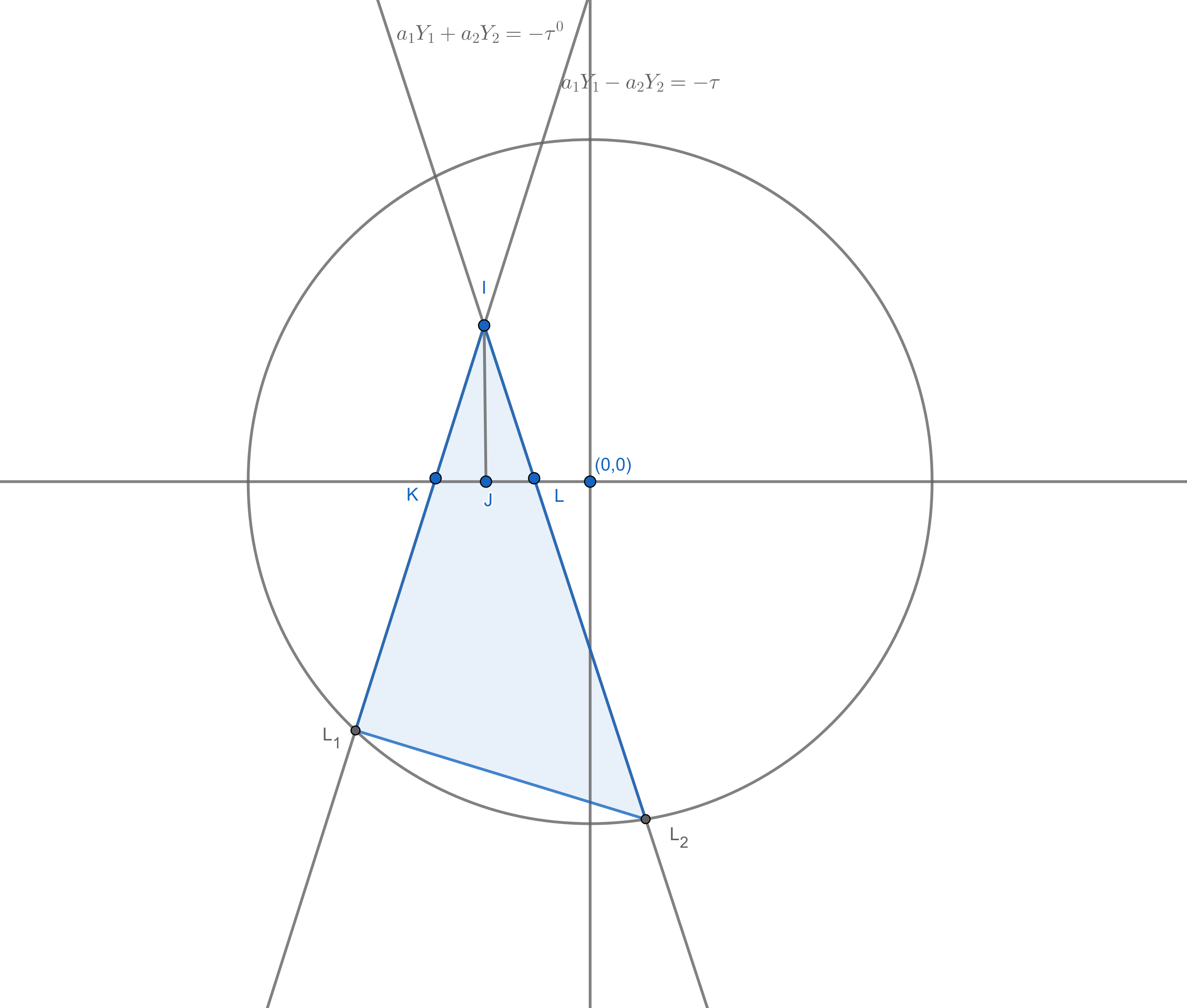}
\end{figure}
Denote $L$ to be the midpoint of $KJ$. If we denote the angle $\angle KIJ $ to be $\theta$, then $\theta = 2(\tan{(-a_1/a_2)}) + \pi$ (which directly follows from the slope of the lines and from the observation that $\Delta IKJ$ is isoceles) and $\tan(\theta/2) = KL/LI$. The length of the side $LI \le 4U/\zeta$ (diameter of the circle) which implies:  
\begin{align}
\label{intlb1}\tan{(\theta/2)} \ge \frac{KL}{(4U/\zeta)} =\zeta \frac{|\tau - \tau^0|}{8U|a_1|} \ge \zeta\frac{|\tau - \tau^0|}{8U} 
\end{align}
as $|a_1| \le 1$. On the other hand we have following upper bound on $\tan{(\theta /2 )}$: 
\allowdisplaybreaks
\begin{align}
\tan{(\theta/2)} &= \tan{(\tan^{-1}(-a_1/a_2) + \pi/2)} \notag\\
& = \cot{(\tan^{-1}(a_1/a_2))} \notag\\
& = \frac{\cos{(\tan^{-1}(a_1/a_2))}}{\sin{(\tan^{-1}(a_1/a_2))}} \notag\\
\label{intlb3}& = \frac{\|\beta - \beta^0\|}{\|\beta + \beta^0\|_2} \\
\label{intlb2} & \le \frac{1}{2\zeta}\|\beta - \beta^0\|_2
\end{align}
Combining \ref{intlb1} and \ref{intlb2} we have $\|\beta - \beta^0\|_2 \ge \frac{\zeta^2}{4U}|\tau - \tau^0|$. Define $L_1$ to be the point where extended $IK$ meets the circle and $L_2$ to be the point where extended $IJ$ meets the circle. ($L_2$ may be equal to $J$). The triangle $\Delta IL_1L_2$ is inside the circle and $$\text{Area} (\Delta IL_1L_2) = \frac{1}{2}IL_1 \times IL_2 \times \sin{\theta}$$ Now $IL_1, IL_2 \ge U/\zeta$ as the maximum possible distance of $K,J$ from the origin is $U/\zeta$ and $I$ is on the opposite side of $L_1, L_2$ with respect to the $X$-axis. Hence, $\text{Area}(\Delta IL_1L_2) \ge \frac{U^2}{2\zeta^2}\sin{\theta}$. Next, 
\begin{align}
\label{intlb4}\sin{\theta} = 2\sin{(\theta/2)}\cos{(\theta/2)} = 2 \times \frac{1}{2}\|\beta + \beta^0\| \times \frac{1}{2}\|\beta - \beta\|_2 \ge \zeta\|\beta - \beta^0\|
\end{align}
Recall that from $\ref{intlb3}$ it is easy to see $\sin{(\theta/2)} = \frac{1}{2}\|\beta - \beta\|_2, \cos{(\theta/2)} = \frac{1}{2}\|\beta + \beta^0\|_2$. Hence, we have $\text{Area}(\Delta IL_1L_2) \ge \frac{U^2}{2\zeta} \|\beta - \beta^0\|$. which implies that :
\allowdisplaybreaks
\begin{align}
& P(\s(\tau + \beta^TX) \neq \s(\tau^0 + X^T\beta^0)) \notag \\
& \ge P(\s(\tau + \beta^TX) \neq \s(\tau^0 + X^T\beta^0) \cap \text{Circle}) \notag \\
& \ge F(U,\zeta)\frac{U^2}{2\zeta}\|\beta - \beta^0\|_2 \notag \\
& \ge F(U,\zeta)\frac{U^2}{2\zeta}\left[\frac{1}{2}\|\beta - \beta^0\|_2+ \frac{\zeta^2}{8U}|\tau - \tau^0| \right] \hspace{0.2in} [\because \|\beta - \beta^0\|_2 \ge \frac{\zeta^2}{4U}|\tau - \tau^0|] \notag \\
& \ge F(U,\zeta)\frac{U^2}{2\zeta}\left(\frac{1}{2} \wedge \frac{\zeta^2}{8U}\right)\left[\|\beta - \beta^0\|_2 + |\tau - \tau^0|\right] \notag \\
\label{inteqn1}
& \ge a_1^- \|\gamma - \gamma^0\|_2 
\end{align}
\\\\
{\bf Case 1.2: }Suppose the intersection point $I$ is outside of the circle. 
\begin{figure}[H]
\centering
\includegraphics[scale = 0.4]{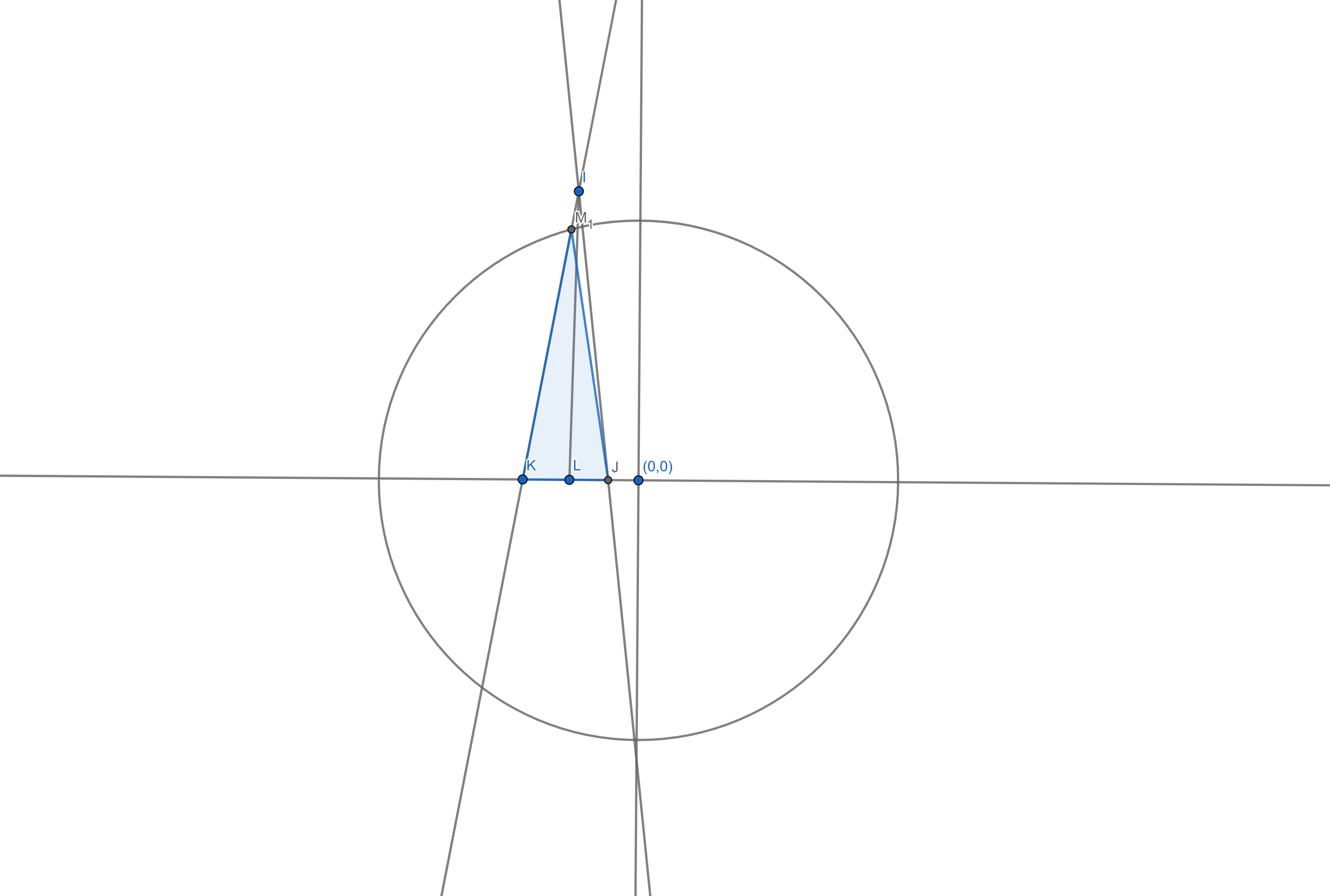}
\end{figure}
\noindent
Here, the length of $LI$ is $\ge \sqrt{3}\frac{U}{\zeta}$ as $I$ is outside the circle and the maximum possible distance of $L$ from the origin is $U/\zeta$. Using this, we have: $$\tan{(\theta/2)} \le \frac{KL}{(\sqrt{3}U/\zeta)} =\zeta \frac{|\tau - \tau^0|}{2\sqrt{3}U|a_1|} \le \frac{|\tau - \tau^0|}{2\sqrt{3}U}$$ Also, from equation \ref{intlb3}, we obtain $\tan{(\theta/2)} \ge (1/2)\|\beta - \beta^0\|$. Combining these bounds, we have, $\|\beta - \beta^0\|_2 \le \frac{1}{\sqrt{3}U}|\tau - \tau^0|$. Let the line $a_1Y_1 - a_2Y_2 + \tau = 0$ cuts the circle at $M_1, M_2$. Consider the triangle $\Delta M_1KJ$. Then the area of this triangle is: $$\text{Area} (\Delta MKJ)= \frac{1}{2} M_1K \times KJ \times \sin{\phi}$$ where $\phi = \angle M_1KJ$. By the same logic as before, $M_1K \ge \frac{U}{\zeta}$, $KJ = |\tau - \tau^0|/a_1$ and $\sin{\phi} = \sin{(\tan^{-1}(a_1/a_2))} = a_1 $. Hence, area of $\Delta M_1KJ \ge \frac{U}{2\zeta}|\tau - \tau^0|$. Using this, we have:
\allowdisplaybreaks
\begin{align}
& P(\s(\tau + \beta^TX) \neq \s(\tau^0 + X^T\beta^0)) \notag\\
& \ge P(\s(\tau + \beta^TX) \neq \s(\tau^0 + X^T\beta^0) \cap \text{Circle}) \notag\\
& \ge F(U,\zeta)\frac{U}{2\zeta}|\tau - \tau^0| \notag \\
& \ge F(U,\zeta)\frac{U}{2\zeta}\left[\frac{1}{2}|\tau - \tau^0|+ \frac{\sqrt{3}U}{2} \|\beta - \beta^0\|_2\right] \notag \\
& \ge F(U,\zeta)\frac{U}{2\zeta}\left(\frac{1}{2} \wedge  \frac{\sqrt{3}U}{2}\right)\left[\|\beta - \beta^0\|_2 + |\tau - \tau^0|\right] \notag \\
\label {inteqn2}
& \ge a_2^- \|\gamma - \gamma^0\|_2 
\end{align}
\\\\
{\bf Case 2: }Suppose $\tau = \tau^0$ and $\beta \neq \beta^0$. Then the lines $a_1Y_1 + a_2Y_2 = -\tau^0$ and $a_1Y_1 - a_2Y_2 = - \tau (\equiv \tau^0)$ meet on the $X$-axis, i.e. $I = K=J = (-\tau^0/a_1)$. 
\begin{figure}[H]
\centering
\includegraphics[scale = 0.3]{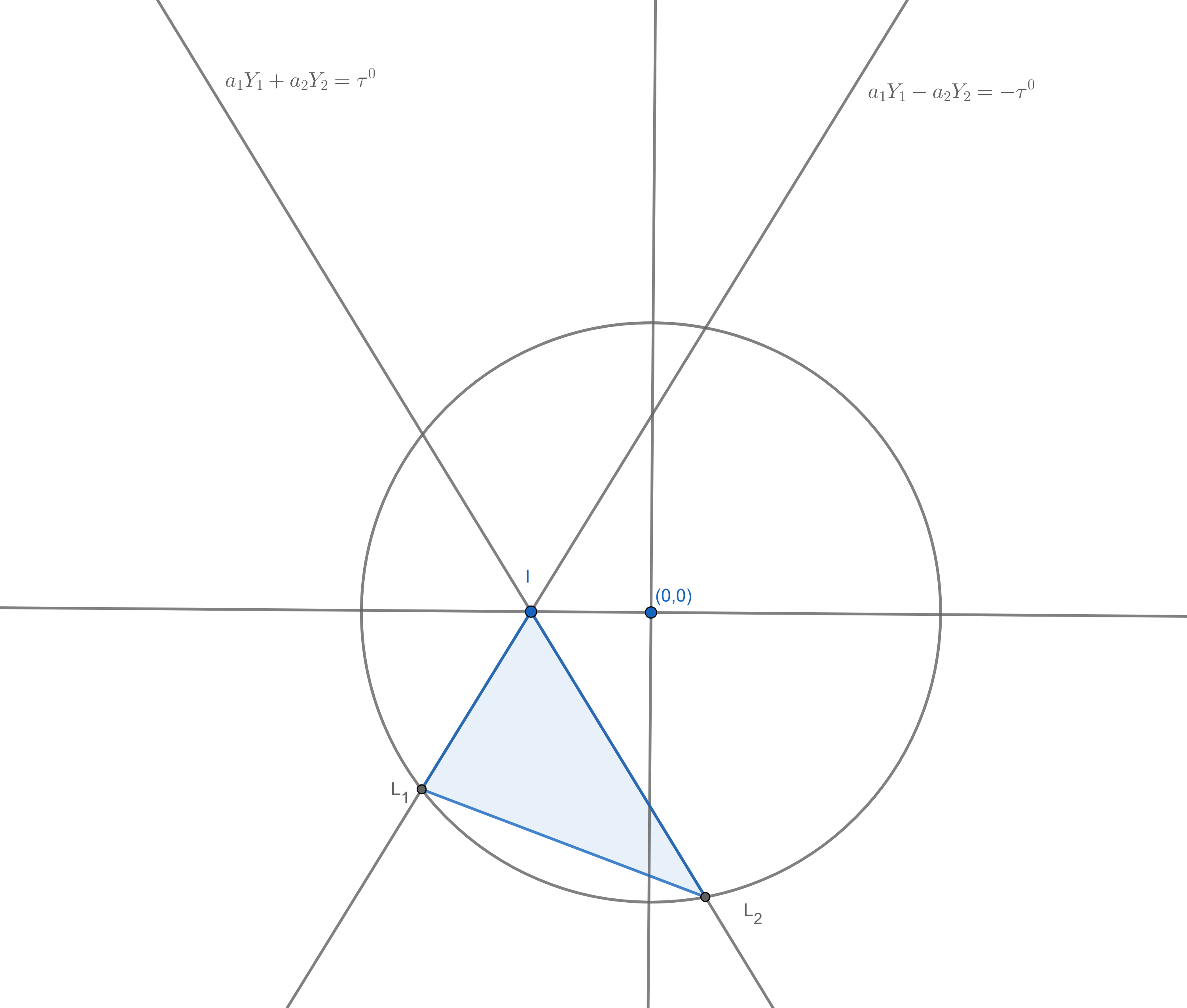}
\end{figure}
\noindent
Consider the triangle $\Delta IL_1L_2$ where $L_1$ and $L_2$ are the intersection points of the lines with the circle. Now the maximum possible distance of $I$ from the origin is $U/\zeta$ which implies $IL_1, IL_2 \ge U/\zeta$. From \ref{intlb4} we have $\sin \angle L_1 I L_2 \ge \zeta \|\beta - \beta^0\|_2$. Combining these, we get: 
\allowdisplaybreaks
\begin{align}
\label{inteqn3}
P(\s(\tau^0 + X^T\beta) \neq \s(\tau^0 + X^T\beta^0)) \ge a_3^-\|\beta - \beta^0\|_2 = a_3^- \|\gamma - \gamma^0\|_2
\end{align}
\\\\
{\bf Case 3: }Finally suppose $\beta = \beta^0$ and $\tau \neq \tau^0$. \begin{figure}[H]
\centering
\includegraphics[scale = 0.4]{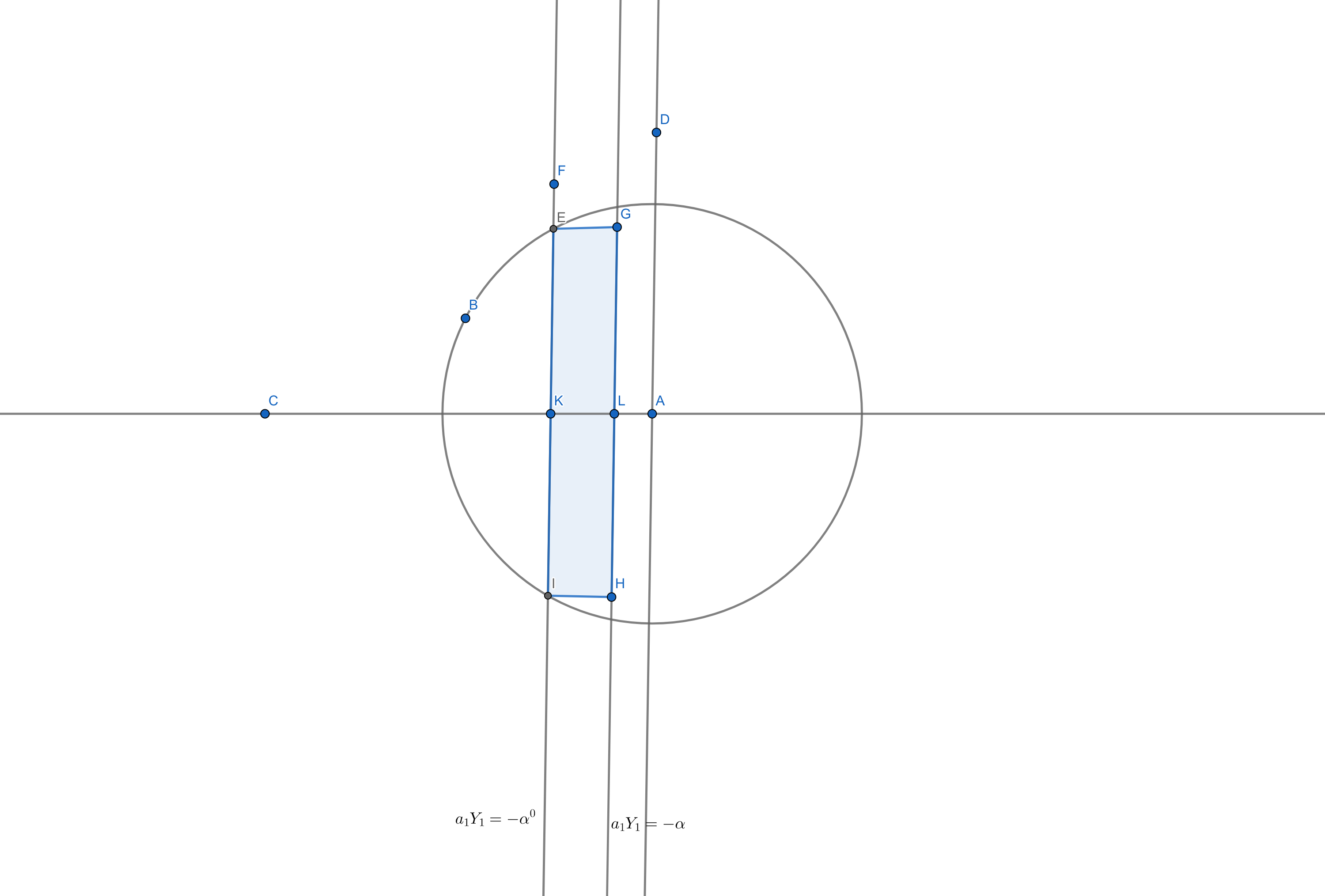}
\end{figure}
\noindent
Consider the rectangle $\Box EIHG$. Here $EG = IH = KL = |\tau - \tau^0|/a_1 = |\tau - \tau^0|$. Also $EI = GH = 2EK \ge \frac{\sqrt{3}U}{\zeta}$. Hence $$\Box EIHG = EG \times EK \ge \frac{\sqrt{3}U}{\zeta}|\tau - \tau^0|$$ which establishes: 
\allowdisplaybreaks
\begin{align}
\label{inteqn4}
P(\s(\tau + X^T\beta^0) \neq \s(\tau^0 + X^T\beta^0)) \ge a_4^-|\tau - \tau^0| = a_4^- \|\gamma - \gamma^0\|_2
\end{align}  
Combining equations \ref{inteqn1}, \ref{inteqn2}, \ref{inteqn3} and \ref{inteqn4} we conclude that Assumption (A2:upper) is valid for this intercept model with $a^- = a_1^- \wedge a_2^- \wedge a_3^- \wedge a_4^-$.

\section{Another version of rate theorem}
\label{sec:new_Talagarand}
In this section we present a version of Theorem 3.2.5 of \cite{vdvw96} which provides the rate of convergence of a generic $M$-estimator along with an 
exponential tail bound, under appropriate conditions.  This theorem can be applied instead of Theorem \ref{massart} to establish rate of convergence along with finite sample concentration bound.

\begin{theorem}
Suppose $\mathbb{M}_n$ be a stochastic processes indexed by a set $\Theta$ and $M: \Theta \rightarrow \mathbb{R}$ be a deterministic process which takes the form: $M(\theta) = Pf_{\theta}$ and $\M_n = \P_nf_{\theta}$. Define $\mathcal{F} = \{f_{\theta}: \theta \in \Theta\}$. Assume for every $\theta$ in a neighborhood of $\theta_0$:  $$ M(\theta) - M(\theta_0) \gtrsim d^{\gamma}(\theta, \theta_0)$$Suppose that for every $n$ and for sufficiently small $\delta$, the centered process $\mathbb{M}_n - M$ satisfies: $$\E^*\left(\sup_{d(\theta, \theta_0) \le \delta} \sqrt{n}\left|(\mathbb{M}_n-M)(\theta) - (\mathbb{M}_n-M)(\theta_0)\right|\right) \lesssim \phi_n(\delta)$$ for functions $\phi_n$ such that $\delta \rightarrow \phi_n(\delta)/\delta^{\alpha}$ is decreasing for some $0 < \alpha < \gamma$. Let $\{r_n\}$ satisfies $$r_n^{\gamma} \phi_n\left(\frac{1}{r_n}\right) = \sqrt{n}$$ for every $n$. If the sequence $\hat{\theta}_n$ takes value in $\Theta$ and satisfies $\mathbb{M}_n(\hat{\theta}_n) \le \mathbb{M}_n(\theta_0) - O_p(r_n^{-2})$ and $d(\hat{\theta}_n ,\theta_0)$ converges to $0$ in outer probability, then $r_n d(\hat{\theta}_n, \theta_0) = O_p^*(1)$. If all the above conditions are valid for all $\delta$ and $\theta$, then we don't need consistency and we can obtain the following finite sample concentration bound: $$\P(r_n d(\hat \theta, \theta_0) > t)  \le Ct^{\alpha - \gamma} \ \forall \ t > 1 \,.$$ 

In addition, assume $\|f_{\theta}\|_{\infty} \le U$ (w.l.o.g. take $U = 1/4$ for simplicity of notation) for all $\theta \in \Theta$ and the existence of $0 < \beta < 2\gamma$ such that: $$Var(f_{\theta} - f_{\theta_0}) \le d^{\beta}(\theta, \theta_0) \ \forall \ \theta \in \Theta \,.$$ 
Then, the following exponential concentrations obtain, for all $t > (1/2)2^{-\frac{\gamma + 1}{\gamma - \alpha}} \vee 1/2$: 
\begin{enumerate}
\item If $\gamma <  \beta < 2\gamma$ and $\liminf_{n \rightarrow \infty} nr_n^{-\gamma} > 0$, then $$\P(r_n d(\hat \theta, \theta_0) > t)  \le C\exp{(-ct^{2\gamma - \beta})} \,.$$

%\sum_{i=1}^{\infty} \exp{\left(-t^{\gamma}\left[nr_n^{-\gamma}2^{i\gamma}2^{-2\gamma}\frac{\left(1 - 2^{\gamma}(t2^i)^{-\gamma+\alpha}\right)^2}{(10/3)(t2^i)^{-\gamma+\alpha}+2(t2^i)^{\beta - \gamma}r_n^{\gamma - \beta}+ (2^{1-\gamma}/3)}\right]\right)}$$
\item If $0 < \beta \le \gamma$ and $\liminf_{n \rightarrow 0} nr_n^{\beta-2\gamma} > 0$, then $$\P(r_n d(\hat \theta, \theta_0) > t)  \le C\exp{(-ct^{\gamma})} \,.$$ 
\end{enumerate}
Here the constants $C,c$ may be different in Case 1 and Case 2, but they don't depend on $n$. 

\end{theorem}

\begin{proof}
For simplicity let's assume the conditions are valid for $\delta, \theta$. We establish the finite sample concentration here. Fix $t > 1$: Define $C_i = \{\theta: d(\theta, \theta_0) \le tr_n^{-1}2^i\}$ for $i \in \mathbb{N}$.  Also define $g_{\theta}(X) = f_{\theta_0}(X) - f_{\theta}(X) - Pf_{\theta_0} + Pf_{\theta}$ and without loss of generality assume $\|g_{\theta}\|_{\infty} \le 1$. We also need the following quantities to apply Talagrand's inequality: $$\mu_{n,i} = \E\left(\sup_{d(\theta, \theta_0) \le tr_n^{-1}2^i} \left\{\sum_{i=1}^ng_{\theta}(X_i)\right\}\right) \le \sqrt{n} \phi_n(tr_n^{-1}2^i) \le \sqrt{n}\phi_{n}(r_n^{-1})(t2^{i})^{\alpha}$$ $$\sigma_{n,i}^2 = \sup_{d(\theta, \theta_0) \le tr_n^{-1}2^i} Var(g_{\theta}) = \sup_{d(\theta, \theta_0) \le tr_n^{-1}2^i} Var(f_{\theta} - f_{\theta_0}) \le (tr_n^{-1}2^i)^{\beta} \overset{\Delta} = \tilde{\sigma}^2_{n,i}$$ $$\nu_{n,i} = 2 \sqrt{n}\phi_{n}(r_n^{-1})(t2^{i})^{\alpha}+ n\tilde{\sigma}^2_{n,i} = 2 \sqrt{n}\phi_{n}(r_n^{-1})(t2^{i})^{\alpha}+n(tr_n^{-1}2^i)^{\beta} $$
\allowdisplaybreaks
\begin{align*}
& \P\left(r_n d(\hat \theta, \theta_0)> t\right)  \\
& = \P\left(d(\hat \theta, \theta_0)> tr_n^{-1}\right) \\
& =  \P\left(\sup_{\theta \in C_0^c} \left\{\M_n(\theta_0) - \M_n(\theta)\right\} \ge 0\right) \\
& \le \sum_{i=1}^{\infty} \P\left(\sup_{\theta \in C_{i-1}^c \cap C_i} \left\{\M_n(\theta_0) - \M_n(\theta)\right\} \ge 0\right) \\
& \le  \sum_{i=1}^{\infty} \P\left(\sup_{\theta \in C_{i-1}^c \cap C_i} \left\{(\M_n - M)(\theta_0 - \theta)\right\} + \sup_{\theta \in C_{i-1}^c \cap C_i}\left\{\M(\theta_0 - \theta)\right\} \ge 0\right) \\
& \le  \sum_{i=1}^{\infty} \P\left(\sup_{\theta \in C_{i-1}^c \cap C_i} \left\{(\M_n - M)(\theta_0 - \theta)\right\} \ge \inf_{\theta \in C_{i-1}^c \cap C_i}\left\{\M(\theta - \theta_0)\right\} \right) \\
& \le  \sum_{i=1}^{\infty} \P\left(\sup_{\theta \in C_{i-1}^c \cap C_i} \left\{(\M_n - M)(\theta_0 - \theta)\right\} \ge \inf_{\theta \in C_{i-1}^c \cap C_i}d^{\gamma}(\theta, \theta_0) \right) \\
& \le  \sum_{i=1}^{\infty} \P\left(\sup_{\theta \in C_i} \left\{(\M_n - M)(\theta_0 - \theta)\right\} \ge (tr_n^{-1}2^{i-1})^{\gamma} \right) \\
& \le  \sum_{i=1}^{\infty} \P\left(\sup_{\theta \in C_i} \left\{\frac{1}{n}\sum_{i=1}^n\left(f_{\theta_0}(X_i) - f_{\theta}(X_i) - Pf_{\theta_0} + Pf_{\theta}\right)\right\} \ge (tr_n^{-1}2^{i-1})^{\gamma} \right) \\
& \le  \sum_{i=1}^{\infty} \P\left(\sup_{\theta \in C_i} \left\{\frac{1}{n}\sum_{i=1}^ng_{\theta}(X_i)\right\} \ge (tr_n^{-1}2^{i-1})^{\gamma} \right) \\
& \le  \sum_{i=1}^{\infty} \P\left(\sup_{d(\theta, \theta_0) \le tr_n^{-1}2^i} \left\{\sum_{i=1}^ng_{\theta}(X_i)\right\} \ge n(tr_n^{-1}2^{i-1})^{\gamma} \right) \\
& \le  \sum_{i=1}^{\infty} \P\left(\sup_{d(\theta, \theta_0) \le tr_n^{-1}2^i} \left\{\sum_{i=1}^n g_{\theta}(X_i)\right\} -\mu_{n,i} \ge n(tr_n^{-1}2^{i-1})^{\gamma}-\mu_{n,i} \right) \\
& \le \sum_{i=1}^{\infty} \P\left(\sup_{d(\theta, \theta_0) \le tr_n^{-1}2^i} \left\{\sum_{i=1}^n g_{\theta}(X_i)\right\} - \mu_{n,i}  \ge n(tr_n^{-1}2^{i-1})^{\gamma} - \sqrt{n}\phi_{n}(r_n^{-1})(t2^{i})^{\alpha} \right) \\
& \le \sum_{i=1}^{\infty} \exp{\left(-\frac{\left(n(tr_n^{-1}2^{i-1})^{\gamma} - \sqrt{n}\phi_{n}(r_n^{-1})(t2^{i})^{\alpha}\right)^2}{4 \sqrt{n}\phi_{n}(r_n^{-1})(t2^{i})^{\alpha}+2n(tr_n^{-1}2^i)^{\beta}+ 2(n(tr_n^{-1}2^{i-1})^{\gamma} - \sqrt{n}\phi_{n}(r_n^{-1})(t2^{i})^{\alpha})/3}\right)} \\
& \le \sum_{i=1}^{\infty} \exp{\left(-\frac{\left(n(tr_n^{-1}2^{i-1})^{\gamma} - \sqrt{n}\phi_{n}(r_n^{-1})(t2^{i})^{\alpha}\right)^2}{4 \sqrt{n}\phi_{n}(r_n^{-1})(t2^{i})^{\alpha}+2n(tr_n^{-1}2^i)^{\beta}+ 2(n(tr_n^{-1}2^{i-1})^{\gamma} - \sqrt{n}\phi_{n}(r_n^{-1})(t2^{i})^{\alpha})/3}\right)} \\
& \le \sum_{i=1}^{\infty} \exp{\left(-r_n^{-\gamma}\frac{\left(n(t2^{i-1})^{\gamma} - \sqrt{n}r_n^{\gamma}\phi_{n}(r_n^{-1})(t2^{i})^{\alpha}\right)^2}{4 \sqrt{n}r_n^{\gamma}\phi_{n}(r_n^{-1})(t2^{i})^{\alpha}+2n(t2^i)^{\beta}r_n^{\gamma - \beta}+ 2(n(t2^{i-1})^{\gamma} - \sqrt{n}r_n^{\gamma}\phi_{n}(r_n^{-1})(t2^{i})^{\alpha})/3}\right)} \\
& \le \sum_{i=1}^{\infty} \exp{\left(-r_n^{-\gamma}\frac{\left(n(t2^{i-1})^{\gamma} - n(t2^{i})^{\alpha}\right)^2}{4 n(t2^{i})^{\alpha}+2n(t2^i)^{\beta}r_n^{\gamma - \beta}+ 2(n(t2^{i-1})^{\gamma} - n(t2^{i})^{\alpha})/3}\right)} \\
& \le \sum_{i=1}^{\infty} \exp{\left(-nr_n^{-\gamma}\frac{\left((t2^{i-1})^{\gamma} - (t2^{i})^{\alpha}\right)^2}{4 (t2^{i})^{\alpha}+2(t2^i)^{\beta}r_n^{\gamma - \beta}+ 2((t2^{i-1})^{\gamma} - (t2^{i})^{\alpha})/3}\right)} \\
& \le \sum_{i=1}^{\infty} \exp{\left(-nr_n^{-\gamma}\frac{\left((t2^{i-1})^{\gamma} - (t2^{i})^{\alpha}\right)^2}{4 (t2^{i})^{\alpha}+2(t2^i)^{\beta}r_n^{\gamma - \beta}+ 2((t2^{i-1})^{\gamma} - (t2^{i})^{\alpha})/3}\right)} \\
& \le \sum_{i=1}^{\infty} \exp{\left(-nr_n^{-\gamma}\frac{\left((t2^{i-1})^{\gamma} - (t2^{i})^{\alpha}\right)^2}{(10/3) (t2^{i})^{\alpha}+2(t2^i)^{\beta}r_n^{\gamma - \beta}+ (2/3)(t2^{i-1})^{\gamma}}\right)} \\
& \le \sum_{i=1}^{\infty} \exp{\left(-nr_n^{-\gamma}(t2^{i})^{\alpha}2^{-2\gamma}\frac{\left((t2^i)^{\gamma-\alpha} - 2^{\gamma}\right)^2}{(10/3)+2(t2^i)^{\beta - \alpha}r_n^{\gamma - \beta}+ (2^{1-\gamma}/3)(t2^i)^{\gamma - \alpha}}\right)} \\
& \le \sum_{i=1}^{\infty} \exp{\left(-nr_n^{-\gamma}(t2^{i})^{\alpha}2^{-2\gamma}\frac{(t2^i)^{\gamma-\alpha}\left(1 - 2^{\gamma}(t2^i)^{-\gamma+\alpha}\right)^2}{(10/3)(t2^i)^{-\gamma+\alpha}+2(t2^i)^{\beta - \gamma}r_n^{\gamma - \beta}+ (2^{1-\gamma}/3)}\right)} \\
& \le \sum_{i=1}^{\infty} \exp{\left(-nr_n^{-\gamma}(t2^{i})^{\gamma}2^{-2\gamma}\frac{\left(1 - 2^{\gamma}(t2^i)^{-\gamma+\alpha}\right)^2}{(10/3)(t2^i)^{-\gamma+\alpha}+2(t2^i)^{\beta - \gamma}r_n^{\gamma - \beta}+ (2^{1-\gamma}/3)}\right)} \\
& \le \exp{(-t^{\gamma})}\sum_{i=1}^{\infty} \exp{\left(-t^{\gamma}\left[nr_n^{-\gamma}2^{i\gamma}2^{-2\gamma}\frac{\left(1 - 2^{\gamma}(t2^i)^{-\gamma+\alpha}\right)^2}{(10/3)(t2^i)^{-\gamma+\alpha}+2(t2^i)^{\beta - \gamma}r_n^{\gamma - \beta}+ (2^{1-\gamma}/3)}-1\right]\right)} \\
& \le \sum_{i=1}^{\infty} \exp{\left(-t^{\gamma}\left[nr_n^{-\gamma}2^{i\gamma}2^{-2\gamma}\frac{\left(1 - 2^{\gamma}(t2^i)^{-\gamma+\alpha}\right)^2}{(10/3)(t2^i)^{-\gamma+\alpha}+2(t2^i)^{\beta - \gamma}r_n^{\gamma - \beta}+ (2^{1-\gamma}/3)}\right]\right)}
\end{align*}

\noindent
Now we manipulate the last sum. For ease of understanding we divide the rest of the proof into three parts. First, assume $\gamma = \beta$. From the assumption of the theorem, there exists $c > 0$ such that $\liminf_{n \rightarrow \infty} nr_n^{-\gamma} \ge c$. As $0 < \alpha < \gamma$, for all $t > 1/2$, $(t2^i)^{-\gamma + \alpha} < 1$.

\allowdisplaybreaks
\begin{align*}
& \sum_{i=1}^{\infty} \exp{\left(-t^{\gamma}\left[nr_n^{-\gamma}2^{i\gamma}2^{-2\gamma}\frac{\left(1 - 2^{\gamma}(t2^i)^{-\gamma+\alpha}\right)^2}{(10/3)(t2^i)^{-\gamma+\alpha}+2(t2^i)^{\beta - \gamma}r_n^{\gamma - \beta}+ (2^{1-\gamma}/3)}\right]\right)} \\
& \le \sum_{i=1}^{\infty} \exp{\left(-t^{\gamma}\left[c2^{i\gamma}2^{-2\gamma}\frac{\left(1 - 2^{\gamma}(t2^i)^{-\gamma+\alpha}\right)^2}{(10/3)+ (2^{1-\gamma}/3)}\right]\right)} \\
& \le \sum_{i=1}^{\infty} \exp{\left(-t^{\gamma}\left[c_12^{i\gamma}\left(1 - 2^{\gamma}(t2^i)^{-\gamma+\alpha}\right)^2\right]\right)} 
\end{align*}

\noindent
Finally assume that $t > (1/2)2^{-\frac{\gamma + 1}{\gamma - \alpha}}$. Then $2^{\gamma}(t2^i)^{-\gamma+\alpha} < 1/2$ for all $i \ge 1$ which implies $\left(1 - 2^{\gamma}(t2^i)^{-\gamma+\alpha}\right)^2 \ge 1/4$. Putting this we get:

\allowdisplaybreaks
\begin{align*}
\sum_{i=1}^{\infty} \exp{\left(-t^{\gamma}\left[c_12^{i\gamma}\left(1 - 2^{\gamma}(t2^i)^{-\gamma+\alpha}\right)^2\right]\right)} & \le  \sum_{i=1}^{\infty} \exp{\left(-\frac{c_1}{4}t^{\gamma}2^{i\gamma}\right)}  \\
& \le \sum_{i=1}^{\infty} \exp{\left(-\frac{c_1}{4}t^{\gamma}i\right)} \\
& \le \frac{\exp{\left(-\frac{c_1}{4}t^{\gamma}\right)}}{1 - \exp{\left(-\frac{c_1}{4}t^{\gamma}\right)}} \\
& \le C\exp{\left(-\frac{c_1}{4}t^{\gamma}\right)}
\end{align*}

\noindent
Next we solve the series for the case when $0 < \beta < \gamma$. We assume here $\liminf_{n \rightarrow \infty} nr_n^{\beta - 2\gamma} = c > 0$ as stated in the theorem. Like before. lets assume $t > (1/2)2^{-\frac{\gamma + 1}{\gamma - \alpha}} \vee 1/2$. Then $2^{\gamma}(t2^i)^{-\gamma+\alpha} < 1/2$ for all $i \ge 1$ which implies $\left(1 - 2^{\gamma}(t2^i)^{-\gamma+\alpha}\right)^2 \ge 1/4$. Also, as $\beta < \gamma$, we have $r_n^{\gamma + \beta} < 1$ for all large $n$. Hence we have:

\allowdisplaybreaks
\begin{align*}
& \sum_{i=1}^{\infty} \exp{\left(-t^{\gamma}\left[nr_n^{\beta-2\gamma}2^{i\gamma}2^{-2\gamma}\frac{\left(1 - 2^{\gamma}(t2^i)^{-\gamma+\alpha}\right)^2}{(10/3)(t2^i)^{-\gamma+\alpha}r_n^{-\gamma + \beta}+2(t2^i)^{\beta - \gamma}+ (2^{1-\gamma}/3)r_n^{-\gamma +\beta}}\right]\right)} \\
& \le \sum_{i=1}^{\infty} \exp{\left(-t^{\gamma}\left[c2^{i\gamma}2^{-2\gamma}\frac{(1/4)}{(10/3)+2+ (2^{1-\gamma}/3)}\right]\right)} \\
& \le \sum_{i=1}^{\infty} \exp{\left(-t^{\gamma}\left[c_12^{i\gamma}\right]\right)}  \\
& \le C\exp{\left(-\frac{c_1}{4}t^{\gamma}\right)} \hspace{0.3in} [\text{Just like previous calculation}]
\end{align*}

\noindent 
Finally let's assume $\gamma < \beta < 2\gamma$. Then we have the assumption $\liminf_{n \rightarrow 0} nr_n^{-\gamma} = c > 0$. Assume $t > (1/2)2^{-\frac{\gamma + 1}{\gamma - \alpha}} \vee 1/2$. Then $2^{\gamma}(t2^i)^{-\gamma+\alpha} < 1/2$ for all $i \ge 1$ which implies $\left(1 - 2^{\gamma}(t2^i)^{-\gamma+\alpha}\right)^2 \ge 1/4$. Here as $\beta > \gamma$ we have $r_n^{\gamma - \beta} < 1$ for all large $n$. Then we have:

\allowdisplaybreaks
\begin{align*}
& \sum_{i=1}^{\infty} \exp{\left(-t^{\gamma}\left[nr_n^{-\gamma}2^{i\gamma}2^{-2\gamma}\frac{\left(1 - 2^{\gamma}(t2^i)^{-\gamma+\alpha}\right)^2}{(10/3)(t2^i)^{-\gamma+\alpha}+2(t2^i)^{\beta - \gamma}r_n^{\gamma - \beta}+ (2^{1-\gamma}/3)}\right]\right)} \\
& \le \sum_{i=1}^{\infty} \exp{\left(-t^{\gamma}\left[c2^{i\gamma}2^{-2\gamma}\frac{(1/4)}{(10/3)+2(t2^i)^{\beta - \gamma}+ (2^{1-\gamma}/3)}\right]\right)} \\
&  \le \sum_{i=1}^{\infty} \exp{\left(-t^{2\gamma - \beta}\left[c2^{i(2\gamma-\beta)}2^{-2\gamma}\frac{(1/4)}{(10/3)(t2^i)^{-\beta + \gamma}+2+ (2^{1-\gamma}/3)(t2^i)^{-\beta + \gamma}}\right]\right)} \\
&  \le \sum_{i=1}^{\infty} \exp{\left(-t^{2\gamma - \beta}\left[c2^{i(2\gamma-\beta)}2^{-2\gamma}\frac{(1/4)}{(10/3)+2+ (2^{1-\gamma}/3)}\right]\right)} \\
&  \le \sum_{i=1}^{\infty} \exp{\left(-t^{2\gamma - \beta}\left[c_12^{i(2\gamma-\beta)}\right]\right)} \\
&  \le \sum_{i=1}^{\infty} \exp{\left(-c_1t^{2\gamma - \beta}i\right)} \\
& \le \frac{\exp{\left(-c_1t^{2\gamma - \beta}\right)}}{1 - \exp{\left(-c_1t^{2\gamma - \beta}\right)}} \\
& \le C\exp{\left(-c_1t^{2\gamma - \beta}\right)}
\end{align*}

\begin{remark} 
Here I have used the inequality $2^{i\gamma} > i$ and $2^{i(\gamma - \beta)} > i$ for notational simplicity. For more exact bound, one can use the fact that $a^i \ge i e\log{a}$ for all $i$ for all $a > 1$.  
\end{remark}

\end{proof}
\end{appendix}

\bibliographystyle{plainnat}
\bibliography{mybib}

\begin{thebibliography}{48}
\providecommand{\natexlab}[1]{#1}
\providecommand{\url}[1]{\texttt{#1}}
\expandafter\ifx\csname urlstyle\endcsname\relax
  \providecommand{\doi}[1]{doi: #1}\else
  \providecommand{\doi}{doi: \begingroup \urlstyle{rm}\Url}\fi

\bibitem[Abrevaya and Huang(2005)]{abrevaya2005}
Jason Abrevaya and Jian Huang.
\newblock On the bootstrap of the maximum score estimator.
\newblock \emph{Econometrica}, 73\penalty0 (4):\penalty0 1175--1204, 2005.

\bibitem[Assouad(1983)]{assouad1983deux}
Patrice Assouad.
\newblock Deux remarques sur l'estimation.
\newblock \emph{Comptes rendus des s{\'e}ances de l'Acad{\'e}mie des sciences.
  S{\'e}rie 1, Math{\'e}matique}, 296\penalty0 (23):\penalty0 1021--1024, 1983.

\bibitem[Bajari et~al.(2008)Bajari, Fox, and Ryan]{bajari2008evaluating}
Patrick Bajari, Jeremy~T Fox, and Stephen~P Ryan.
\newblock Evaluating wireless carrier consolidation using semiparametric demand
  estimation.
\newblock \emph{Quantitative Marketing and Economics}, 6\penalty0 (4):\penalty0
  299, 2008.

\bibitem[Ball(1988)]{ball1988logarithmically}
Keith Ball.
\newblock Logarithmically concave functions and sections of convex sets in rn.
\newblock \emph{Studia Math}, 88\penalty0 (1):\penalty0 69--84, 1988.

\bibitem[Bickel et~al.(2009)Bickel, Ritov, Tsybakov,
  et~al.]{bickel2009simultaneous}
Peter~J Bickel, Ya’acov Ritov, Alexandre~B Tsybakov, et~al.
\newblock Simultaneous analysis of lasso and dantzig selector.
\newblock \emph{The Annals of Statistics}, 37\penalty0 (4):\penalty0
  1705--1732, 2009.

\bibitem[Bobkov and Chistyakov(2015)]{bobkov2015concentration}
Sergey~G Bobkov and Gennadiy~P Chistyakov.
\newblock On concentration functions of random variables.
\newblock \emph{Journal of Theoretical Probability}, 28\penalty0 (3):\penalty0
  976--988, 2015.

\bibitem[Bousquet(2002)]{bousquet2002bennett}
Olivier Bousquet.
\newblock A bennett concentration inequality and its application to suprema of
  empirical processes.
\newblock \emph{Comptes Rendus Mathematique}, 334\penalty0 (6):\penalty0
  495--500, 2002.

\bibitem[Briesch et~al.(2002)Briesch, Chintagunta, and
  Matzkin]{briesch2002semiparametric}
Richard~A Briesch, Pradeep~K Chintagunta, and Rosa~L Matzkin.
\newblock Semiparametric estimation of brand choice behavior.
\newblock \emph{Journal of the American Statistical Association}, 97\penalty0
  (460):\penalty0 973--982, 2002.

\bibitem[B{\"u}hlmann and Van De~Geer(2011)]{buhlmann2011statistics}
Peter B{\"u}hlmann and Sara Van De~Geer.
\newblock \emph{Statistics for high-dimensional data: methods, theory and
  applications}.
\newblock Springer Science \& Business Media, 2011.

\bibitem[Chernoff(1964)]{chernoff1964estimation}
Herman Chernoff.
\newblock Estimation of the mode.
\newblock \emph{Annals of the Institute of Statistical Mathematics},
  16\penalty0 (1):\penalty0 31--41, 1964.

\bibitem[Feng et~al.(2019)Feng, Ning, and Zhao]{feng2019nonregular}
Huijie Feng, Yang Ning, and Jiwei Zhao.
\newblock Nonregular and minimax estimation of individualized thresholds in
  high dimension with binary responses.
\newblock \emph{arXiv preprint arXiv:1905.10888}, 2019.

\bibitem[Fox(2007)]{fox2007semiparametric}
Jeremy~T Fox.
\newblock Semiparametric estimation of multinomial discrete-choice models using
  a subset of choices.
\newblock \emph{The RAND Journal of Economics}, 38\penalty0 (4):\penalty0
  1002--1019, 2007.

\bibitem[Fox and Bajari(2013)]{fox2013measuring}
Jeremy~T Fox and Patrick Bajari.
\newblock Measuring the efficiency of an fcc spectrum auction.
\newblock \emph{American Economic Journal: Microeconomics}, 5\penalty0
  (1):\penalty0 100--146, 2013.

\bibitem[Friedman et~al.(2001)Friedman, Hastie, and
  Tibshirani]{friedman2001elements}
Jerome Friedman, Trevor Hastie, and Robert Tibshirani.
\newblock \emph{The elements of statistical learning}, volume~1.
\newblock Springer series in statistics New York, 2001.

\bibitem[Greenshtein et~al.(2004)Greenshtein, Ritov,
  et~al.]{greenshtein2004persistence}
Eitan Greenshtein, Ya'Acov Ritov, et~al.
\newblock Persistence in high-dimensional linear predictor selection and the
  virtue of overparametrization.
\newblock \emph{Bernoulli}, 10\penalty0 (6):\penalty0 971--988, 2004.

\bibitem[Horowitz(1992)]{horowitz1992}
Joel~L Horowitz.
\newblock A smoothed maximum score estimator for the binary response model.
\newblock \emph{Econometrica: journal of the Econometric Society}, pages
  505--531, 1992.

\bibitem[Huang et~al.(2008)Huang, Ma, and Zhang]{huang2008adaptive}
Jian Huang, Shuangge Ma, and Cun-Hui Zhang.
\newblock Adaptive lasso for sparse high-dimensional regression models.
\newblock \emph{Statistica Sinica}, pages 1603--1618, 2008.

\bibitem[Kim et~al.(1990)Kim, Pollard, et~al.]{kp90}
Jeankyung Kim, David Pollard, et~al.
\newblock Cube root asymptotics.
\newblock \emph{The Annals of Statistics}, 18\penalty0 (1):\penalty0 191--219,
  1990.

\bibitem[Klein and Spady(1993)]{klein1993efficient}
Roger~W Klein and Richard~H Spady.
\newblock An efficient semiparametric estimator for binary response models.
\newblock \emph{Econometrica: Journal of the Econometric Society}, pages
  387--421, 1993.

\bibitem[Lee(1995)]{lee1995semiparametric}
Lung-Fei Lee.
\newblock Semiparametric maximum likelihood estimation of polychotomous and
  sequential choice models.
\newblock \emph{Journal of Econometrics}, 65\penalty0 (2):\penalty0 381--428,
  1995.

\bibitem[Mammen et~al.(1999)Mammen, Tsybakov, et~al.]{mammentsy}
Enno Mammen, Alexandre~B Tsybakov, et~al.
\newblock Smooth discrimination analysis.
\newblock \emph{The Annals of Statistics}, 27\penalty0 (6):\penalty0
  1808--1829, 1999.

\bibitem[Manski(1975)]{manski1975}
Charles~F Manski.
\newblock Maximum score estimation of the stochastic utility model of choice.
\newblock \emph{Journal of econometrics}, 3\penalty0 (3):\penalty0 205--228,
  1975.

\bibitem[Manski(1985)]{manski1985}
Charles~F Manski.
\newblock Semiparametric analysis of discrete response: Asymptotic properties
  of the maximum score estimator.
\newblock \emph{Journal of econometrics}, 27\penalty0 (3):\penalty0 313--333,
  1985.

\bibitem[Manski and Thompson(1986)]{manski1986}
Charles~F Manski and T~Scott Thompson.
\newblock Operational characteristics of maximum score estimation.
\newblock \emph{Journal of Econometrics}, 32\penalty0 (1):\penalty0 85--108,
  1986.

\bibitem[Massart()]{massart2007concentration}
Pascal Massart.
\newblock Concentration inequalities and model selection.

\bibitem[Massart et~al.(2006)Massart, N{\'e}d{\'e}lec, et~al.]{massart2006risk}
Pascal Massart, {\'E}lodie N{\'e}d{\'e}lec, et~al.
\newblock Risk bounds for statistical learning.
\newblock \emph{The Annals of Statistics}, 34\penalty0 (5):\penalty0
  2326--2366, 2006.

\bibitem[Miolane and Montanari(2018)]{miolane2018distribution}
L{\'e}o Miolane and Andrea Montanari.
\newblock The distribution of the lasso: Uniform control over sparse balls and
  adaptive parameter tuning.
\newblock \emph{arXiv preprint arXiv:1811.01212}, 2018.

\bibitem[Mukherjee et~al.(2019)Mukherjee, Banerjee, and
  Ritov]{mukherjee2019non}
Debarghya Mukherjee, Moulinath Banerjee, and Ya'acov Ritov.
\newblock Non-standard asymptotics in high dimensions: Manski's maximum score
  estimator revisited.
\newblock \emph{arXiv preprint arXiv:1903.10063}, 2019.

\bibitem[Patra et~al.(2018)Patra, Seijo, and Sen]{patra2018}
Rohit~Kumar Patra, Emilio Seijo, and Bodhisattva Sen.
\newblock A consistent bootstrap procedure for the maximum score estimator.
\newblock \emph{Journal of Econometrics}, 2018.

\bibitem[Politis et~al.(1999)Politis, Romano, and Wolf]{politis1999subsampling}
Dimitris~N Politis, Joseph~P Romano, and Michael Wolf.
\newblock Subsampling springer series in statistics, 1999.

\bibitem[Pollard(1981)]{pollard1981limit}
David Pollard.
\newblock Limit theorems for empirical processes.
\newblock \emph{Zeitschrift f{\"u}r Wahrscheinlichkeitstheorie und verwandte
  Gebiete}, 57\penalty0 (2):\penalty0 181--195, 1981.

\bibitem[Raskutti et~al.(2011)Raskutti, Wainwright, and
  Yu]{raskutti2011minimax}
Garvesh Raskutti, Martin~J Wainwright, and Bin Yu.
\newblock Minimax rates of estimation for high-dimensional linear regression
  over $\ell_q $-balls.
\newblock \emph{IEEE transactions on information theory}, 57\penalty0
  (10):\penalty0 6976--6994, 2011.

\bibitem[Sen(2018)]{sen2018gentle}
Bodhisattva Sen.
\newblock A gentle introduction to empirical process theory and applications.
\newblock 2018.

\bibitem[Sen et~al.(2010)Sen, Banerjee, Woodroofe,
  et~al.]{sen2010inconsistency}
Bodhisattva Sen, Moulinath Banerjee, Michael Woodroofe, et~al.
\newblock Inconsistency of bootstrap: The grenander estimator.
\newblock \emph{The Annals of Statistics}, 38\penalty0 (4):\penalty0
  1953--1977, 2010.

\bibitem[Seo and Otsu(2015)]{seo2015asymptotics}
MYUNG~HWAN Seo and TAISUKE Otsu.
\newblock Asymptotics for maximum score method under general conditions.
\newblock Technical report, LSE Working Paper, 2015.

\bibitem[Seo et~al.(2018)Seo, Otsu, et~al.]{seo2018local}
Myung~Hwan Seo, Taisuke Otsu, et~al.
\newblock Local m-estimation with discontinuous criterion for dependent and
  limited observations.
\newblock \emph{The Annals of Statistics}, 46\penalty0 (1):\penalty0 344--369,
  2018.

\bibitem[Tibshirani(1996)]{tibshirani1996regression}
Robert Tibshirani.
\newblock Regression shrinkage and selection via the lasso.
\newblock \emph{Journal of the Royal Statistical Society: Series B
  (Methodological)}, 58\penalty0 (1):\penalty0 267--288, 1996.

\bibitem[Tsybakov et~al.(2004)]{2004optimal}
Alexander~B Tsybakov et~al.
\newblock Optimal aggregation of classifiers in statistical learning.
\newblock \emph{The Annals of Statistics}, 32\penalty0 (1):\penalty0 135--166,
  2004.

\bibitem[Tsybakov(2009)]{tsybakov2009introduction}
Alexandre~B Tsybakov.
\newblock Introduction to nonparametric estimation. revised and extended from
  the 2004 french original. translated by vladimir zaiats, 2009.

\bibitem[Van~de Geer et~al.(2014)Van~de Geer, B{\"u}hlmann, Ritov, Dezeure,
  et~al.]{van2014asymptotically}
Sara Van~de Geer, Peter B{\"u}hlmann, Ya’acov Ritov, Ruben Dezeure, et~al.
\newblock On asymptotically optimal confidence regions and tests for
  high-dimensional models.
\newblock \emph{The Annals of Statistics}, 42\penalty0 (3):\penalty0
  1166--1202, 2014.

\bibitem[Van Der~Vaart and Wellner(1996)]{vdvw96}
Aad~W Van Der~Vaart and Jon~A Wellner.
\newblock Weak convergence.
\newblock In \emph{Weak convergence and empirical processes}, pages 16--28.
  Springer, 1996.

\bibitem[Vapnik and Chervonenkis(1974)]{vapnik1974theory}
Vladimir Vapnik and Alexey Chervonenkis.
\newblock Theory of pattern recognition, 1974.

\bibitem[Wei and Huang(2010)]{wei2010consistent}
Fengrong Wei and Jian Huang.
\newblock Consistent group selection in high-dimensional linear regression.
\newblock \emph{Bernoulli: official journal of the Bernoulli Society for
  Mathematical Statistics and Probability}, 16\penalty0 (4):\penalty0 1369,
  2010.

\bibitem[Yan and Yoo(2019)]{yan2019semiparametric}
Jin Yan and Hong~Il Yoo.
\newblock Semiparametric estimation of the random utility model with
  rank-ordered choice data.
\newblock \emph{Journal of Econometrics}, 2019.

\bibitem[Yuan and Lin(2006)]{yuan2006model}
Ming Yuan and Yi~Lin.
\newblock Model selection and estimation in regression with grouped variables.
\newblock \emph{Journal of the Royal Statistical Society: Series B (Statistical
  Methodology)}, 68\penalty0 (1):\penalty0 49--67, 2006.

\bibitem[Zhang et~al.(2008)Zhang, Huang, et~al.]{zhang2008sparsity}
Cun-Hui Zhang, Jian Huang, et~al.
\newblock The sparsity and bias of the lasso selection in high-dimensional
  linear regression.
\newblock \emph{The Annals of Statistics}, 36\penalty0 (4):\penalty0
  1567--1594, 2008.

\bibitem[Zhao and Yu(2006)]{zhao2006model}
Peng Zhao and Bin Yu.
\newblock On model selection consistency of lasso.
\newblock \emph{Journal of Machine learning research}, 7\penalty0
  (Nov):\penalty0 2541--2563, 2006.

\bibitem[Zhu et~al.(2004)Zhu, Rosset, Tibshirani, and Hastie]{zhu20041}
Ji~Zhu, Saharon Rosset, Robert Tibshirani, and Trevor~J Hastie.
\newblock 1-norm support vector machines.
\newblock In \emph{Advances in neural information processing systems}, pages
  49--56, 2004.

\end{thebibliography}

\end{document}